\documentclass[]{revtex4-1}

\usepackage{graphicx}
\usepackage{epsfig}
\usepackage{epstopdf} 
\usepackage[inkscapelatex=false]{svg}
\usepackage{appendix}
\usepackage{csquotes}
\usepackage{xcolor}
\usepackage{siunitx}
\usepackage{makecell}   
\usepackage{placeins}   
\usepackage{booktabs}
\usepackage{array}
\usepackage{multirow}

\usepackage{mathtools}
\usepackage{amsfonts}
\usepackage{amsthm}
\usepackage{mathrsfs}

\usepackage{comment}
\usepackage{cancel}
\usepackage[normalem]{ulem}

\usepackage{hyperref} 

\DeclareUnicodeCharacter{03B3}{$\gamma$}

\DeclareUnicodeCharacter{03B4}{$\delta$}

\newcommand{\matteo}[1]{\textcolor{blue}{#1}}

 \newtheorem{Proposition}{Proposition}

 \newtheorem{Definition}{Definition}

\DeclareSIUnit\Molar{\textsc{m}}

\makeatletter
\def\@email#1#2{%
 \endgroup
 \patchcmd{\titleblock@produce}
  {\frontmatter@RRAPformat}
  {\frontmatter@RRAPformat{\produce@RRAP{*#1\href{mailto:#2}{#2}}}\frontmatter@RRAPformat}
  {}{}
}%
\makeatother
\begin{document}

\preprint{AIP/123-QED}

\title{Combining standard chemoradiotherapy with CAR-T cell therapy in malignant gliomas: Insights from an impulsive mathematical framework and virtual trials}

\author{Dmitry Sinelshchikov}
\affiliation{Instituto Biofisika (UPV/EHU, CSIC), University of the Basque Country, Leioa, 48940, Spain}
\affiliation{Ikerbasque, Basque Foundation for Science, Bilbao 48009, Spain.}

\author{Nikols Amaru Mora Millán}
\affiliation{Instituto Biofisika (UPV/EHU, CSIC), University of the Basque Country, Leioa, 48940, Spain}
\author{Miguel Perales-Patón}
\affiliation{Mathematical Oncology Laboratory (MOLAB), Instituto de Matemática Aplicada a la Ciencia y la Ingeniería, University of Castilla-La Mancha, Ciudad Real, Spain}
\affiliation{Department of Mathematics, Escuela Técnica Superior de Ingeniería Industrial, University of Castilla-La Mancha, Ciudad Real, Spain}
\affiliation{Laboratorio de Oncología Matemática, Instituto de Investigación Sanitaria de Castilla-La Mancha (IDISCAM), Spain}
\author{Juan Belmonte-Beitia}
\affiliation{Mathematical Oncology Laboratory (MOLAB), Instituto de Matemática Aplicada a la Ciencia y la Ingeniería, University of Castilla-La Mancha, Ciudad Real, Spain}
\affiliation{Department of Mathematics, Escuela Técnica Superior de Ingeniería Industrial, University of Castilla-La Mancha, Ciudad Real, Spain}
\affiliation{Laboratorio de Oncología Matemática, Instituto de Investigación Sanitaria de Castilla-La Mancha (IDISCAM), Spain}

\author{Matteo Italia}
\affiliation{Laboratorio de Oncología Matemática, Instituto de Investigación Sanitaria de Castilla-La Mancha (IDISCAM), Spain}
\affiliation{Mathematical Oncology Laboratory (MOLAB), Instituto de Matemática Aplicada a la Ciencia y la Ingeniería, University of Castilla-La Mancha, Ciudad Real, Spain}
\affiliation{Department of Mathematics, Escuela Técnica Superior de Ingeniería Industrial, University of Castilla-La Mancha, Ciudad Real, Spain}

\date{\today}

\begin{abstract}
The extreme therapeutic resistance of malignant gliomas motivates the development of multimodal treatment strategies. We present a mechanistic mathematical framework to investigate the coupled dynamics of radiotherapy, temozolomide, and chimeric antigen receptor (CAR) T-cell therapy. The model is formulated as a nine-dimensional impulsive dynamical system that captures proliferative-quiescent tumor states, treatment-specific resistance, tissue damage, TMZ pharmacokinetics, and the lymphotoxic interferences exerted by conventional regimens on engineered T cells. Mathematical analysis establishes the existence of biologically meaningful solutions, maps invariant structures, and identifies threshold conditions under which sustained therapy locally stabilizes the 
tumor-free equilibrium.

After study-matched benchmarking against clinical trials, including the standard Stupp protocol, we deploy heterogeneous virtual patient cohorts to evaluate five combined CAR-T--Stupp sequences. At the population level, the immunotherapeutic adjunct yields a consistent but modest survival extension, increasing median overall survival by $0.8$--$0.9$ months with minimal sensitivity to temporal ordering or scheduling perturbations. 

Crucially, these population-level medians mask profound individual heterogeneity: under combined protocols, $19\%$--$22\%$ of virtual patients extend survival by more than $60$ days, and $\sim 3\%$ gain over a year. A low intrinsic tumor proliferation rate $r_1$ emerges as the primary determinant of response, bolstered by robust T-cell expansion and attenuated microenvironmental suppression. These findings suggest that the clinical promise of multimodal integration depends on identifying responsive tumor phenotypes for personalized treatments rather than searching for a single universally optimal timeline.
\end{abstract}

\maketitle

\begin{quotation}
\begin{quotation}
Malignant gliomas remain difficult to treat because tumor cells differ
in their proliferative state and treatment sensitivity, while
radiotherapy, temozolomide, and CAR-T cells can both cooperate and
interfere with one another. Here, we use a mechanistic mathematical
model to study how these three therapies interact and how their order
and timing affect survival in heterogeneous virtual patients. The model
predicts that adding CAR-T therapy to standard chemoradiotherapy provides
a modest improvement in median survival at the population level, with
relatively small differences among the treatment schedules considered.
However, this population average hides substantial variability: some
virtual patients derive markedly larger benefits than others. Tumor
proliferation, CAR-T expansion, tumor-mediated CAR-T inactivation, and
pre-existing CAR-T resistance emerge as key determinants of this
heterogeneity, suggesting that identifying responsive tumor phenotypes
may be more important than searching for a single universally optimal
treatment sequence.
\end{quotation}
\end{quotation}

\section{Introduction}
Malignant gliomas (MGs) are diffusely infiltrating primary tumors of the central nervous system characterized by aggressive growth, marked biological heterogeneity, and frequent recurrence despite multimodal treatment, with long-term survival and durable disease control achieved only in a small minority of patients. Contemporary classifications no longer regard MGs as a single histological entity but instead define gliomas through integrated diagnoses combining histological and molecular features. These molecularly defined entities differ substantially in their natural history, therapeutic sensitivity, and prognosis \cite{Louis2021WHO,Weller2024Glioma}. Throughout this work, we use the term MGs as an operational umbrella for infiltrative gliomas with aggressive clinical behavior, for which radiotherapy (RT), alkylating chemotherapy (CT), and emerging cellular therapies constitute relevant treatment options.

The management of MGs depends on the underlying molecular diagnosis, tumor location, patient age, functional status (e.g., ECOG or Karnofsky performance status), and extent of surgical resection. Nevertheless, maximal safe resection, whenever feasible, followed by RT and systemic CT remains a central therapeutic strategy for many patients \cite{Weller2021EANO}. Temozolomide (TMZ) is one of the most widely used alkylating agents in this setting, either concomitantly with RT, as maintenance treatment, or at recurrence. Surgery followed by RT with concomitant and adjuvant TMZ, commonly known as the Stupp protocol, represents the current standard of care for the most aggressive forms of MG \cite{Stupp2005,Hegi2005}. Despite such intensive multimodal treatment, durable tumor control remains uncommon, emphasizing the need for therapies capable of overcoming residual disease, phenotypic adaptation, and treatment resistance.

Therapeutic failure in MGs cannot be attributed solely to insufficient treatment intensity. The standard Stupp protocol \cite{Stupp2005} is a one-size-fits-all strategy and does not explicitly account for the substantial inter-patient variability in tumor growth, proliferative state, treatment resistance, or immune response. This heterogeneity provides a strong rationale for investigating optimized and potentially personalized treatment schedules. 
These tumors comprise dynamically interacting cellular populations with different proliferative capacities, phenotypic states, and treatment susceptibilities \cite{Patel2014,Neftel2019,Quail2017}. Consequently, therapy can alter not only the total tumor burden but also the relative abundance of phenotypic subpopulations, allowing initially minor therapy-resistant subpopulations to drive subsequent progression. Chimeric antigen receptor (CAR) T-cell therapy has emerged as a promising strategy to complement conventional treatments by providing an antigen-specific mechanism of tumor-cell elimination. CAR-T cells are genetically engineered lymphocytes that recognize selected surface antigens and, following antigen engagement, undergo activation and proliferative expansion. After their clinical success in hematological malignancies \cite{Maude2018}, several CAR constructs have been investigated against MG-associated targets, including IL13R$\alpha$2, EGFR and EGFRvIII, HER2, GD2, and B7-H3 \cite{Li2025}. For a comprehensive overview of recent advancements and clinical applications of CAR-T cell therapies specifically tailored for MGs, we refer the reader to recent reviews \cite{cart_review_MG,review_cart_diffuse_midline_glioma,montoya2024roadmapCART,neurooncoReview2024}.

Early-phase clinical studies have demonstrated the feasibility of CAR-T cell administration in patients with different forms of MG and have reported tumor trafficking, disease stabilization, and occasional objective responses \cite{Brown2024,Bagley2025,Choi2024,Monje2025}. However, durable control remains uncommon. The efficacy of CAR-T therapy in solid tumors is limited by antigen heterogeneity and loss, restricted tumor infiltration, insufficient expansion and persistence, T-cell dysfunction, and local immunosuppression \cite{Li2025,ORourke2017}. These limitations suggest that CAR-T cells may be more effective when integrated with conventional treatments rather than being administered as an isolated intervention.

Combining CAR-T cells with RT and TMZ may exploit complementary vulnerabilities within heterogeneous tumors. TMZ can affect tumor populations that evade antigen-specific recognition, whereas CAR-T cells can target antigen-positive populations that are intrinsically resistant or have acquired resistance to TMZ. TMZ-induced lymphodepletion may also promote the expansion of subsequently administered CAR-T cells, and conventional treatments can modify tumor antigen or stress-ligand expression \cite{Suryadevara2018}. Similarly, RT can reduce tumor burden, alter the tumor microenvironment, and promote immune-cell recruitment and infiltration. Synergistic activity between RT and CAR-T cells has been observed in preclinical glioma models \cite{Weiss2018}. Conversely, TMZ is cytotoxic to proliferating lymphocytes, and CAR-T cells are highly radiosensitive \cite{Suryadevara2018,paganetti2023review}. The therapeutic outcome may therefore depend strongly on the order, spacing, and temporal overlap of the three treatments.

The proliferative state of the tumor constitutes an additional source of treatment heterogeneity. TMZ and RT preferentially damage actively proliferating cells, whereas quiescent or slowly cycling populations can be relatively protected and subsequently contribute to tumor repopulation \cite{Roos2007,Pawlik2004,Tejero2019,Liau2017,Rabe2020,Xie2022}. The fraction of proliferative cells can be partially characterized by the $\mbox{Ki-67}$ index, providing a clinically interpretable connection between tumor biology and mathematical compartmentalization. Moreover, lethally damaged tumor cells may remain temporarily within the measurable tumor burden before being progressively cleared. Distinguishing proliferative, quiescent, resistant, and damaged populations is therefore relevant for representing both treatment response and relapse dynamics.

Mechanistic mathematical models provide a systematic framework for studying these nonlinear interactions and comparing treatment schedules that would be difficult to assess exhaustively in experimental or clinical studies \cite{Altrock2015,Harkos2025}. By explicitly representing tumor growth, phenotypic transitions, treatment pharmacokinetics, resistance, and therapeutic-cell dynamics, these models can help determine when different treatments cooperate or interfere with one another. Sampling model parameters from biologically plausible ranges also enables the generation of heterogeneous Virtual Patient (VP) cohorts, the performance of \textit{in silico} trials, and the identification of model-derived prognostic or predictive biomarkers \cite{Wang2024}.

The first framework studied the combination of TMZ and CAR-T cell therapy in MGs by distinguishing tumor populations according to their sensitivity or resistance to each treatment \cite{Sinelshchikov2025}. It incorporated TMZ-induced resistance, antigen-dependent CAR-T killing and expansion, tumor-mediated CAR-T inactivation, TMZ pharmacokinetics, and the detrimental effect of TMZ on CAR-T cells. Its results suggested that treatment sequencing could be exploited to control populations with complementary resistance profiles. A subsequent framework described MGs dynamics under combined RT and TMZ and incorporated proliferative and quiescent tumor cells, treatment-induced damage, phenotypic transitions, and TMZ resistance \cite{PeralesPaton2026}. The model was calibrated and validated using \textit{in vivo} data and subsequently employed to virtually investigate alternative chemoradiotherapy schedules.

Neither framework, however, allows the simultaneous investigation of RT, TMZ, and CAR-T therapy while accounting for treatment-specific resistance, proliferative heterogeneity, and the potentially detrimental effects of conventional treatments on CAR-T cells.
Building on these foundations, we propose a unified nine-dimensional ordinary differential equation (ODE) model for the treatment of MGs with RT, TMZ, and CAR-T cells. The tumor is classified simultaneously according to proliferative status and treatment sensitivity, distinguishing proliferative and quiescent populations that are resistant to CAR-T cells, resistant to TMZ, or sensitive to both treatments. Reversible transitions between proliferative and quiescent states are related to the $\mbox{Ki-67}$ index, while a damaged-cell compartment represents lethally injured tumor cells before their clearance. The model also describes TMZ pharmacokinetics, TMZ-induced resistance, antigen-dependent CAR-T cytotoxicity and expansion, tumor-mediated CAR-T inactivation, and the lymphotoxic effects of TMZ and RT. Treatment administrations are incorporated as impulsive inputs, while RT effects on proliferative tumor cells are represented through a common surviving fraction based on the linear-quadratic (LQ) formalism, while CAR-T-cell radiosensitivity is treated separately.

We first investigate the mathematical properties of the model, including the consistency with the original model, the existence and non-negativity of biologically meaningful solutions, tumor-free equilibria, invariant structures, and the dynamics of the proliferative fraction. A continuous-treatment formulation is then used to derive parameter regimes under which sustained treatment can stabilize a tumor-free state. The numerical model is benchmarked against published data for untreated disease and for TMZ, CAR-T, RT, and combined RT--TMZ treatments. Finally, heterogeneous VP cohorts are used to compare protocols in which CAR-T administrations are delivered before, after, or between different phases of an RT-TMZ reference regimen. Overall survival (OS), population dynamics, and associations between model parameters and therapeutic benefit are evaluated to identify promising treatment schedules and candidate biomarkers. This framework provides a mechanistic tool for investigating synergy and antagonism among RT, TMZ, and CAR-T therapy and for supporting the optimization and potential personalization of multimodal treatment strategies for MGs.

\section{Methods}

\subsection{Background}
We develop a compartmental ODE model describing the response of MGs to RT, TMZ, and CAR-T cell therapy. The model represents the tumor as a collection of phenotypically homogeneous subpopulations differing in their proliferative state and treatment sensitivity. Spatial dependencies and stochastic effects are neglected for simplicity. 

To capture tumor lifecycle dynamics and therapeutic evasion, the cancer mass is structured into three distinct viable phenotypes, each represented by a coupled proliferative-quiescent compartment pair describing the continuous biological flux between active division and dormancy \cite{dormantstateandback}. Within this framework, therapeutic sensitivity is strictly restricted to the actively dividing state, as traditional cytotoxic therapies predominantly eliminate cells with proliferative capacity \cite{santos2021dormancy}. Similarly, although CAR T-cell cytotoxicity is not intrinsically cell-cycle dependent, its effective action is restricted here to proliferative cells as a simplifying assumption. This choice is motivated by evidence that quiescent tumor cells can constitute immunotherapy-resistant reservoirs, they exhibit a severe downregulation of global protein synthesis and often reside deep within the solid tumor mass, factors that significantly compromise CAR T-cell efficacy \cite{Ma,baldominos2022quiescent,Majzner2020,kringel2023chimeric}. Furthermore, these quiescent subpopulations establish immunosuppressive, hypoxic niches characterized by HIF-1$\alpha$ stabilization and the accumulation of metabolic waste products like lactic acid, which actively drive the functional exhaustion and inactivation of infiltrating T-cells \cite{Tejero2019,baldominos2022quiescent}.
Consequently, quiescent subpopulations are modeled as fully resistant to treatment. In our non-spatial framework, implementing this intrinsic resistance serves as an elegant surrogate to simulate spatial heterogeneity and physical barriers to CAR T-cell infiltration. In a physical solid tumor, proliferating cells reside on the vascularized outer periphery where nutrient gradients are optimal, whereas quiescent cells cluster within the deep, hypoxic core. Because CAR T-cells face severe trafficking constraints and rarely penetrate this dense inner mass, their activity is largely restricted to the tumor surface. By defining quiescent cells as refractory to killing, the model implicitly captures this penetration gradient, mapping the geometric confinement of CAR T-cells to the periphery and the immune sheltering of the core. Each quiescent compartment thus acts as a protected reservoir, and cells must transition back into the corresponding proliferative compartment to become vulnerable to therapy.

Specifically, within the first phenotype, the coupled pair $\{S, Q\}$ denotes cells that are sensitive to all treatments only when proliferating ($S$) \cite{chemodamage,radiodamage,Feins}, whereas the quiescent pool ($Q$) remains invulnerable until it re-enters the active cycle. Specularly, the coupled pair $\{R_C, Q_C\}$ represents a phenotype that has developed resistance to CAR-T cell killing (antigen escape or loss \cite{Bodnar2023amcs}) but remains sensitive to TMZ through its proliferative fraction ($R_C$). Conversely, the pair $\{R_E, Q_E\}$ accounts for clonal lineages that have escaped TMZ toxicity but remain susceptible to CAR-T immunotherapy via their actively dividing compartment ($R_E$). Finally, the mathematical architecture is completed by three auxiliary variables representing treatment pharmacokinetics and pathology clearance: lethally damaged tumor cells awaiting physical resorption ($D$), active CAR-T cell effectors ($C$), and the effective concentration of TMZ ($E$).

Only the proliferative compartments (\(S\), \(R_C\), and \(R_E\)) contribute directly to tumor growth, which is limited by a common carrying capacity \(K\) \cite{Gerlee2013,Budia2019}. The TMZ-sensitive populations \(S\) and \(R_C\) share the same proliferation rate, whereas TMZ-resistant cells may proliferate at a different rate, reflecting the variable fitness associated with acquired drug resistance \cite{Rabe2020,Sinelshchikov2025}. Cells reversibly transition between proliferative and quiescent states within each treatment-sensitivity phenotype \cite{dormantstateandback}. Quiescent cells (\(Q\), \(Q_C\), and \(Q_E\)) do not divide and display reduced effective sensitivity compared with proliferative cells \cite{Pawlik2004,Tejero2019,Rabe2020,baldominos2022quiescent}; here, quiescent cells are fully resistant.
The initial balance between proliferative and quiescent cells is related to the $\mbox{Ki-67}$ index, which is used as a proxy for the tumor proliferative fraction \cite{dahlrot2021prognostic}.

CAR-T-sensitive cells are assumed to express the antigen or combination of antigens recognized by the therapeutic cells \cite{Bodnar2023amcs}. Thus, CAR-T cells eliminate the proliferative populations \(S\) and \(R_E\), whereas \(R_C\) cells escape CAR-T-mediated killing due to absent or insufficient target-antigen expression.

TMZ resistance is induced in TMZ sensitive cells ($S$ and $R_C$) upon exposure to TMZ \cite{Rabe2020}. Thus, the model incorporates sensitivity and resistance to TMZ, the standard chemotherapeutic agent for gliomas \cite{friedman2000temozolomide,Stupp2005}. Over time, a subset of tumor cells may remain sensitive and responsive to treatment, while others acquire resistance, reflecting the complex and dynamic nature of tumor heterogeneity under dual therapy \cite{Rabe2020}. TMZ, in turn, damages the proliferative TMZ-sensitive populations \(S\) and \(R_C\), while \(R_E\) cells are unaffected by the drug.

Following our previous TMZ--CAR-T framework \cite{Sinelshchikov2025}, cells entering the TMZ-resistant compartment are assumed to remain susceptible to CAR-T-mediated killing, and TMZ-induced resistance is represented through a direct TMZ-dependent transition from the proliferative TMZ-sensitive populations $S$ and $R_C$ to the TMZ-resistant compartment $R_E$.

We also consider the dynamics of the applied treatments. TMZ decays exponentially between consecutive administrations following first-order pharmacokinetics \cite{agarwala2000temozolomide}, it acts only on TMZ-sensitive tumor cells and can induce resistance in these cells \cite{Rabe2020}. TMZ may also damage CAR-T cells because of their proliferative activity, although TMZ-induced lymphodepletion prior to CAR-T administration may enhance subsequent CAR-T expansion and antitumor activity \cite{Suryadevara2018}.
Instead, 
CAR-T cells expand after encountering the antigen-positive populations \(S\) and \(R_E\), undergo natural death or functional inactivation, and are suppressed through interactions with the tumor and its microenvironment, or by RT and TMZ \cite{Brown2,Fonkoua,Sinelshchikov2025}. Tumor-mediated CAR-T inactivation depends on the total tumor burden reflecting the immunosuppressive, hypoxic tumor niches \cite{Tejero2019,baldominos2022quiescent}. 

The compartment $D$ contains tumor cells that have received lethal damage from TMZ, CAR-T cells, or RT, remaining temporarily within the measurable tumor burden before physical clearance. This representation accounts for the characteristic kinetic delay between treatment-induced injury and the actual disappearance of damaged tissue \cite{hirose2005akt, vakifahmetoglu2008death, programmedcelldeathglioma,joiner2009basic}.

RT is administered over a much shorter timescale than tumor growth and is therefore represented through instantaneous state changes \cite{PeralesPaton2026}. It damages the proliferative populations \(S\), \(R_C\), and \(R_E\), with phenotype-specific surviving fractions defined through the LQ formalism \cite{fowler1989linear}. Irradiated cells that do not survive are transferred to \(D\). The direct lethal effect of RT on quiescent cells is neglected, but each RT fraction recruits a fraction of \(Q\), \(Q_C\), and \(Q_E\) into their corresponding proliferative compartments, representing accelerated repopulation after irradiation \cite{Fogarty2017,Zimmerman2013,He2018}. Because lymphocytes are highly radiosensitive, RT also directly reduces the CAR-T cell population \cite{paganetti2023review}.

Thus, the three therapies act on partially overlapping tumor populations. TMZ and CAR-T cells control complementary resistance proliferative phenotypes, while RT targets all proliferative tumor compartments, and it induces accelerated repopulation---modifying the tumor populations exposed to subsequent treatment. Moreover, TMZ and RT reduce CAR-T abundance. These competing interactions provide the biological basis for investigating the effects of treatment order and timing. 

\subsection{The model}
\begin{figure}[ht!]
    \centering
    \includegraphics[width=0.475\textwidth]{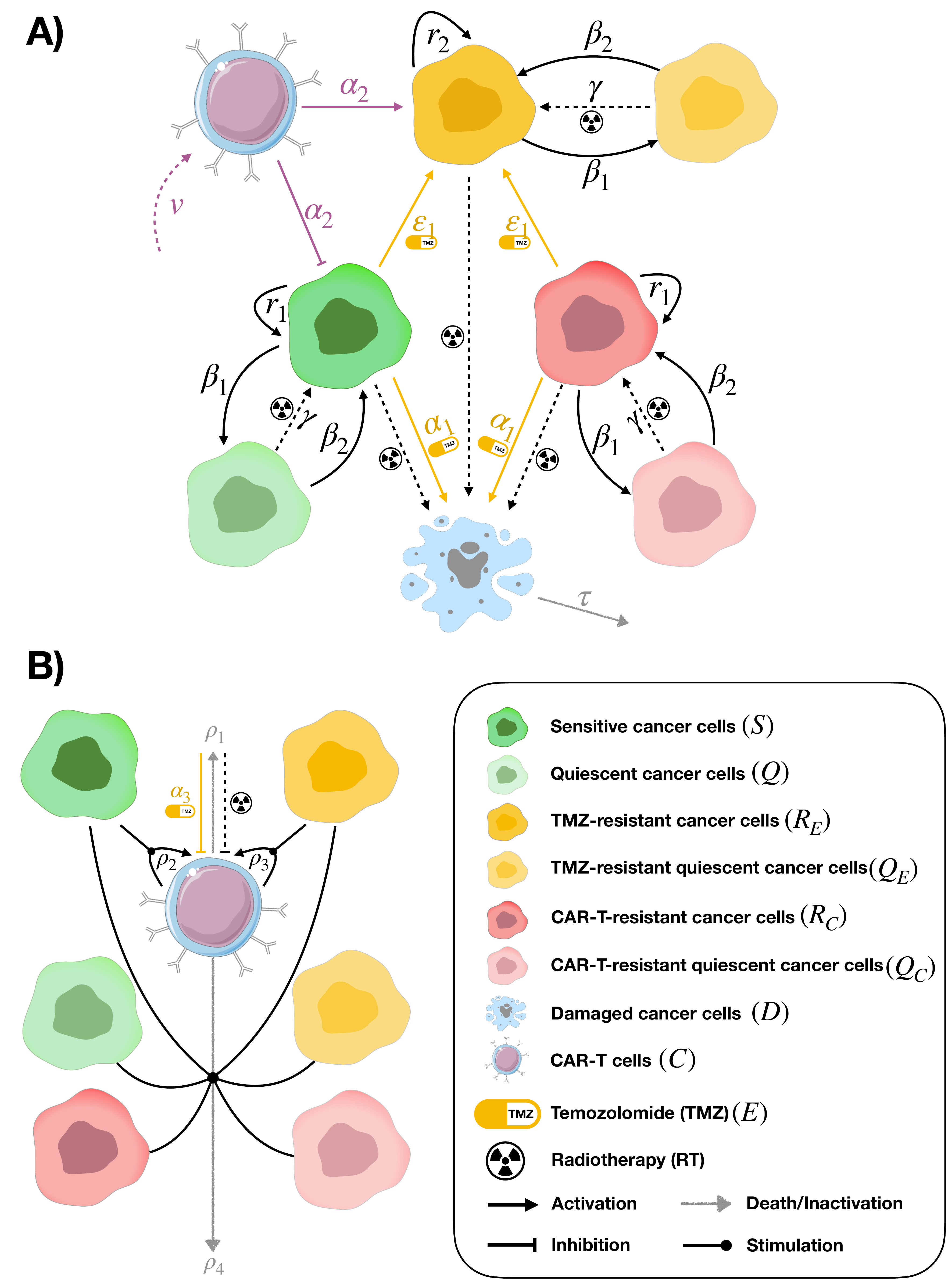}
    \caption{Graphical representation of the mathematical model. \textbf{A)} Tumor-cell dynamics, including proliferative populations ($S$, $R_C$, and $R_E$), their corresponding quiescent compartments ($Q$, $Q_C$, and $Q_E$), reversible proliferative--quiescent transitions, tumor growth, treatment-induced damage, temozolomide (TMZ) resistance, and clearance of damaged cells ($D$). TMZ acts on TMZ-sensitive proliferative cells, CAR-T cells target antigen-positive proliferative cells, and radiotherapy (RT) affects proliferative tumor cells while recruiting quiescent cells back into proliferation.  \textbf{B)} CAR-T-cell dynamics, including stimulation by antigen-positive tumor cells, natural loss or inactivation, tumor-mediated suppression, and the effects of TMZ and RT. Solid arrows denote continuous interactions represented in the ODE system~\eqref{eq:new_model4}, whereas dashed arrows indicate impulsive treatment effects applied at discrete administration times.}
    \label{fig:tumor_dynamics}
\end{figure}
\subsubsection{Model equations}
Figure~\ref{fig:tumor_dynamics} schematically summarizes the model structure. \ref{fig:tumor_dynamics}.A) depicts the tumor subpopulations, their transitions between proliferative and quiescent states, and the effects of RT, TMZ, and CAR-T therapy, whereas \ref{fig:tumor_dynamics}.B) focuses on CAR-T-cell dynamics and their interactions with the tumor and the considered treatments. Building on the previously developed RT--TMZ and TMZ--CAR-T frameworks \cite{PeralesPaton2026,Sinelshchikov2025}, we integrate phenotypically structured tumor dynamics, reversible proliferative--quiescent transitions, treatment-induced damage, and CAR-T-cell dynamics into a unified formulation. The resulting model is a nine-dimensional hybrid dynamical system.
Eq.~\eqref{eq:new_model4} describes the time-continuous tumor dynamics at time points where treatments are not administered (see Section \ref{sec:modeling_treatments} for treatment administration time points). 
RT effects are not included explicitly in the ODE model; instead, it is included separately as instantaneous state updates at the prescribed irradiation times (see system of Eqs.~\eqref{eq:systemRT}). Similarly, CT and CAR-T administrations are incorporated separately as instantaneous updates to drug efficacy ($E$) and CAR T cells (see Eq.~\eqref{eq:tmz_application} and Eq.~\eqref{eq:cart_application}). However, CAR T cells and drug efficacy ($E$) are incorporated into the ODE model described by Eq.~\eqref{eq:new_model4} as time-dependent treatment effects acting on these dynamics, i.e., Eq.~\eqref{eq:C} and Eq.~\eqref{eq:E}. The governing ODE system is given by:

\begin{subequations}
\label{eq:new_model4}
\renewcommand{\theequation}{\theparentequation.\arabic{equation}}
\begin{align}
\dot{S} &=
r_1 S \left(1-\frac{T}{K}\right)
-\beta_1 S+\beta_2 Q
-(\alpha_1+\epsilon_1)ES-\alpha_2 CS,
\label{eq:S}
\\
\dot{Q} &=
\beta_1 S-\beta_2 Q,
\label{eq:Q}
\\
\dot{R}_C &=
r_1 R_C \left(1-\frac{T}{K}\right)
-\beta_1 R_C+\beta_2 Q_C
-(\alpha_1+\epsilon_1)ER_C,
\label{eq:RC}
\\
\dot{Q}_C &=
\beta_1 R_C-\beta_2 Q_C,
\label{eq:QC}
\\
\dot{R}_E &=
r_2 R_E \left(1-\frac{T}{K}\right)-\beta_1 R_E+\beta_2 Q_E-\alpha_2 CR_E+\epsilon_1(S+R_C)E,
\label{eq:RE}
\\
\dot{Q}_E &=
\beta_1 R_E-\beta_2 Q_E,
\label{eq:QE}
\\
\dot{D} &=\alpha_1(S+R_C)E+\alpha_2(S+R_E)C-\tau D,
\label{eq:D}
\\
\dot{C} &=
-\rho_1 C
+\frac{\rho_2 SC}{g_1+S}
+\frac{\rho_3 R_E C}{g_2+R_E}
-\rho_4\frac{TC}{g_3+C}
-\alpha_3 EC,
\label{eq:C}
\\
\dot{E} &=-\mu E,
\label{eq:E}
\end{align}
\end{subequations}
where $T=S+Q+R_C+Q_C+R_E+Q_E+D$ represents the total amount of tumor cells.

Equation~\eqref{eq:S} describes the dynamics of tumor cells \(S\), which are proliferative and sensitive to both TMZ and CAR-T therapy. In the absence of treatment, these cells proliferate at a rate \(r_1\) according to a logistic growth law with a carrying capacity \(K\). The logistic factor accounts for competition for space and resources generated by all tumor compartments, including damaged cells that have not yet been cleared (i.e, \(S,\,Q,\,R_{C},\,Q_{C},\,R_{E},\,Q_{E},\,\text{and }D\)). The terms \(-\beta_1S\) and \(+\beta_2Q\) describe the reversible transition from the proliferative to the quiescent state and the corresponding return from \(Q\) to \(S\), respectively. TMZ exposure produces two distinct effects: \(-\alpha_1ES\) represents irreversible treatment-induced damage, with the corresponding cells transferred to \(D\), whereas \(-\epsilon_1ES\) accounts for TMZ-induced acquisition of resistance, with these cells entering the \(R_E\) compartment. Finally, \(-\alpha_2CS\) represents CAR-T-mediated killing of antigen-positive sensitive cells, which also contributes to the damaged population.

Equation~\eqref{eq:Q} governs the quiescent population \(Q\) associated to S. Cells enter this compartment from \(S\) at a rate \(\beta_1\) and return to the proliferative state at a rate \(\beta_2\). Quiescent cells do not proliferate and are not directly affected by TMZ or CAR-T cells in the continuous dynamics.

Equation~\eqref{eq:RC} describes proliferative tumor cells resistant to CAR-T therapy but sensitive to TMZ. These cells proliferate logistically with the same intrinsic rate \(r_1\) as \(S\) and undergo reversible transitions with the corresponding quiescent compartment \(Q_C\), represented by the terms \(-\beta_1R_C\) and \(+\beta_2Q_C\). Since \(R_C\) cells are resistant to CAR-T cells, no CAR-T-mediated killing term is included. TMZ, however, remains effective against this population: \(-\alpha_1ER_C\) accounts for lethal TMZ-induced damage, whereas \(-\epsilon_1ER_C\) represents the acquisition of TMZ resistance. The latter cells are transferred to \(R_E\). Similarly to Eq.~\eqref{eq:Q}, 
Equation~\eqref{eq:QC} describes the corresponding transitions between proliferative \(R_C\) and quiescent \(Q_C\) cells resistant to CAR-T therapy. 

Equation~\eqref{eq:RE} governs proliferative TMZ-resistant cells. Their untreated growth follows the same logistic limitation as the other proliferative populations, but with a proliferation rate \(r_2\), which may differ from \(r_1\). The terms \(-\beta_1R_E\) and \(+\beta_2Q_E\) account for reversible transitions between proliferative and quiescent TMZ-resistant states. Since \(R_E\) cells are resistant to TMZ, no TMZ-induced killing term is present. However, they retain the antigen recognized by CAR-T cells and are therefore eliminated through the term \(-\alpha_2CR_E\). Finally, \(+\epsilon_1(S+R_C)E\) represents the influx of newly TMZ-resistant cells generated from the two TMZ-sensitive proliferative populations. Similarly to Eqs.~\eqref{eq:Q} and~\eqref{eq:QC}, Equation~\eqref{eq:QE} describes the corresponding transitions between proliferative \(R_E\) and quiescent \(Q_E\) cells resistant to TMZ.

Equation~\eqref{eq:D} describes the damaged tumor-cell population. The term \(\alpha_1(S+R_C)E\) refers to TMZ-sensitive cells that receive lethal damage from TMZ, while \(\alpha_2(S+R_E)C\) accounts for antigen-positive tumor cells that are lethally affected by CAR-T therapy. Damaged cells are progressively removed from the tumor mass at a rate \(\tau\), represented by \(-\tau D\). RT-induced damaged cells do not appear explicitly in Eq.~\eqref{eq:D}, since RT is modeled as an instantaneous intervention; the corresponding cells are transferred to \(D\) at irradiation times through the discrete treatment updates described below.

Equation~\eqref{eq:C} describes CAR-T cell dynamics. The term \(-\rho_1C\) represents their natural death or functional inactivation. CAR-T cells proliferate following antigen recognition, represented by the saturating terms
\(\rho_2SC/(g_1+S)\) and \(\rho_3R_EC/(g_2+R_E)\), corresponding to stimulation by the two antigen-positive proliferative populations \(S\) and \(R_E\), respectively. The parameters \(\rho_2\) and \(\rho_3\) determine the maximal stimulation rates, whereas \(g_1\) and \(g_2\) are the corresponding half-saturation constants. Tumor-mediated CAR-T inactivation is represented by
$- (\rho_4 TC)/(g_3+C)$, which increases with the total tumor burden and saturates with CAR-T abundance. Here, \(\rho_4\) quantifies the strength of tumor-induced CAR-T inactivation, and \(g_3\) controls its saturation. Finally, \(-\alpha_3EC\) accounts for the cytotoxic effect of TMZ on CAR-T cells.

Equation~\eqref{eq:E} describes TMZ pharmacokinetics between consecutive administrations. The effective drug concentration \(E\) decays exponentially according to first-order kinetics, with a clearance rate \(\mu\). TMZ administrations are subsequently represented as instantaneous increases in \(E\).

\subsubsection{Initial and final conditions}

To initialize the numerical simulations, we assume a baseline total tumor population denoted by $T_0$, which can be estimated via volumetric imaging in practical applications. To model intratumoral heterogeneity, we define $\delta_1$ and $\delta_2$ as the initial fractions of tumor cells resistant to TMZ and CAR-T cell therapy, respectively \cite{Sinelshchikov2025}. 

We assume that the fraction of actively dividing cells remains consistent across all phenotypic subpopulations and is governed by the clinical $\mbox{Ki-67}$ proliferation index \cite{PeralesPaton2026}. Consequently, the initial cellular distribution between the coupled proliferative and quiescent compartments is defined as follows:
\begin{subequations} \label{eq:initial_conditions}
\begin{align}
    S(0)   &= \mbox{Ki-67}(1-\delta_{1}-\delta_{2})T_{0}, \\
    Q(0)   &= (1-\mbox{Ki-67})(1-\delta_{1}-\delta_{2})T_{0}, \\
    R_{E}(0) &= \mbox{Ki-67}\delta_{1}T_{0}, \\
    Q_{E}(0) &= (1-\mbox{Ki-67})\delta_{1}T_{0}, \\
    R_{C}(0) &= \mbox{Ki-67}\delta_{2}T_{0}, \\
    Q_{C}(0) &= (1-\mbox{Ki-67})\delta_{2}T_{0}.
\end{align}
\end{subequations}
Furthermore, assuming a previously untreated baseline, the compartments tracking treatment dynamics and damaged cell populations are initially set to zero:
\begin{equation}
    D(0) = 0, \quad C(0) = 0, \quad E(0) = 0.
\end{equation}
Note that these initial conditions for CAR-T cells ($C$) and TMZ efficacy ($E$) can be adjusted to positive values when simulating clinical scenarios where prior therapeutic cycles have already occurred.

Computations are terminated either upon reaching context-specific end-times (such as full remission) or when the tumor burden leads to patient death. Following established clinical metrics, lethal progression is defined by the tumor volume reaching a fatal threshold of $K/5$, equivalent to approximately $10^{12}$ cells \cite{Sinelshchikov2025}. While an absolute lethal cell count varies clinically for gliomas, a burden of $10^{12}$ cells corresponds approximately to a $100$-gram mass or a volume of $300\,\text{cm}^3$, which has been widely reported as a terminal threshold \cite{Delobel_PLOS}.

\subsubsection{Modeling treatment interventions}
\label{sec:modeling_treatments}

In order to model treatment applications, we adopt the impulsive initial conditions approach as in our previous works \cite{Sinelshchikov2025,PeralesPaton2026}. This formulation is biologically justified by their rapid pharmacokinetics and short administration timescales relative to tumor growth dynamics \cite{neuropharmacokinetics, ballesta2014multiscale}. Mathematically, these treatments introduce discrete discontinuities into the continuous ordinary differential equation framework, where $F(t^\pm) := \lim_{t \to t^\pm} F(t)$ denotes the standard one-sided limits immediately before and after therapy.

RT fractions are delivered at times $t_i$ ($i=1,\dots,N_{RT}$). Consistent with the biological rationale, radiation induces lethal DNA damage exclusively in proliferating compartments ($S, R_E, R_C$), transferring a fraction of affected cells into the damaged compartment $D$, while simultaneously triggering a fraction $\gamma$ of quiescent cells ($Q, Q_E, Q_C$) to re-enter active proliferation \cite{kuznetsov2023optimization}. While CAR-T cells ($C$) within the microenvironment are subject to radiation-induced depletion, the biochemical efficacy of TMZ ($E$) is assumed to remain unaffected. The impulsive updates at each RT fraction are modeled as follows:
\begin{subequations} \label{eq:systemRT}
    \begin{align}
        S(t_i^{+}) &= SF\cdot S(t_i^-) + \gamma Q(t_i^-),\label{saltoS}\\
        Q(t_i^+) &= (1-\gamma)Q(t_i^-),\label{saltoSQ}\\
        R_E(t_i^+) &= SF \cdot R_E(t_i^-) + \gamma Q_{E}(t_i^-),\label{saltoRE}\\
        Q_{E}(t_i^+) &= (1-\gamma)Q_{E}(t_i^-),\label{saltoREQ}\\
        R_C(t_i^+) &= SF \cdot R_C(t_i^-) + \gamma Q_{C}(t_i^-),\label{saltoRC}\\
        Q_{C}(t_i^+) &= (1-\gamma)Q_{C}(t_i^-),\label{saltoRCQ}\\
        D(t_i^+) &= D(t_i^-)+(1-SF)S(t_i^-)\nonumber\\&+(1-SF)R_C(t_i^-)+(1-SF)R_E(t_i^-),\label{saltoD}\\
        C(t_i^+) &= SF_C \cdot C(t_i^-).\label{saltoCART}
    \end{align}
\end{subequations}

The surviving fractions $SF$ and $SF_C$ of the tumor subpopulations and of CAR-T cells, respectively, are determined by the classical LQ model:
\begin{equation} \label{eq:LQ_model}
    SF = e^{-(\alpha d_{\rm RT} + \beta d_{\rm RT}^2)}, \quad SF_C=e^{-\alpha_T\,\cdot\,d_{\rm RT}}
\end{equation}
where $d_{\rm RT}$ represents the radiation dose per fraction (Gy), while $\alpha$ ($\text{Gy}^{-1}$) and $\beta$ ($\text{Gy}^{-2}$) quantify the linear non-repairable lethal lesions and the quadratic misrepair of sublethal damage, respectively \cite{fowler1989linear}, and $\alpha_T$ ($\text{Gy}^{-1}$) quantifies the lymphocyte radiosensitivity using a linear dose-response curve \cite{paganetti2023review}. 

TMZ administrations are represented as impulsive boluses delivered at specific times $t_j$ ($j=1,\dots,N_{CT}$). When TMZ is applied, we assume continuity for all variables except the normalized drug efficacy $E$, which satisfies:
\begin{equation}\label{eq:tmz_application}
    E(t_{j}^{+})=E(t_{j}^{-})+E_{0}, \quad \text{with } X(t_j^+) = X(t_j^-) \quad \forall X \neq E,
\end{equation}
where $E_0$ denotes the normalized amount of administered TMZ.

The injection of CAR-T cells is similarly modeled as a discrete, instantaneous influx at times $t_k$ ($k=1,\dots,N_{CAR}$). The application of CAR-T at day $t_{k}$ corresponds to:
\begin{equation}\label{eq:cart_application}
    C(t_{k}^{+})=C(t_{k}^{-})+v, \quad \text{with } X(t_k^+) = X(t_k^-) \quad \forall X \neq C,
\end{equation}
where $v$ represents the number of CAR-T cells given to a patient, while the remaining state variables are assumed to be continuous.

\bigskip
A combination of conditions \eqref{eq:systemRT}, \eqref{eq:cart_application}, and \eqref{eq:tmz_application} is used when a combination of these treatments is modelled, applying the respective updates simultaneously at overlapping treatment time points.

\subsubsection{Model parameters} \label{sec:model_param}

Model parameter values were obtained from published experimental and clinical studies and from previously developed mathematical frameworks for MG treatment. Table~\ref{tab:model_parameters} summarizes the parameter values and sampling ranges considered throughout the study, together with their biological interpretation and corresponding sources. 

The CAR-T-related parameters are primarily inherited from the framework on which the CAR-T component of the present model is based \cite{Sinelshchikov2025}. The interval for $\rho_1$ corresponds to a characteristic activated CAR-T-cell lifetime between approximately 7 and 30 days \cite{Ghorashian}. The stimulation rates $\rho_2$ and $\rho_3$, the saturation constants $g_1$, $g_2$, and $g_3$, and the tumor-killing coefficient $\alpha_2$ follow previous mathematical descriptions of CAR-T-cell dynamics in solid tumors \cite{Ode2}. Although $\rho_2$ and $\rho_3$ are retained as distinct parameters in the model formulation, in the reference virtual cohort we set $\rho_3=\rho_2$, following our previous TMZ-CAR-T framework \cite{Sinelshchikov2025}. The CAR-T dose $v$ is varied over the range $10^7$--$10^9$ cells to encompass clinically relevant orders of magnitude considered in previous modeling and clinical studies \cite{goff2019pilot,Sinelshchikov2025}.

The proliferation rates $r_1$ and $r_2$ describe the intrinsic growth of proliferative tumor cells. Although $r_1$ and $r_2$ are retained as distinct parameters in the mathematical formulation, following our previous
TMZ-CAR-T framework \cite{Sinelshchikov2025}, we impose $r_2=r_1/2$ for each VP in the reference virtual cohort, rather than sampling $r_2$ independently. The relationship between these intrinsic rates and the effective growth rates
of the previous model is addressed in Appendix \ref{sec:app_2}, indicated as DF (derived from)
in Table \ref{tab:model_parameters}.
TMZ-related parameters are largely based on the previously developed resistance model \cite{Delobel_PLOS}. The parameters $\alpha_1$ and $\epsilon_1$ quantify, respectively, TMZ-induced lethal damage and the acquisition of TMZ resistance, while $\mu$ determines the first-order decay of the effective drug concentration. Direct quantitative measurements of TMZ cytotoxicity against therapeutic CAR-T cells in MGs are not available. Following the assumption introduced in the previous TMZ--CAR-T model \cite{Sinelshchikov2025}, we therefore set
$\alpha_3=\alpha_1$ rather than sampling the two parameters independently. 

The initial resistant fractions also represent sources of biological
uncertainty. The parameter $\delta_1$ denotes the initial fraction of tumor
cells resistant to TMZ. Since the prevalence of pre-existing TMZ-resistant
cells is not sufficiently characterized to define a patient-level
distribution, we follow the previous TMZ--CAR-T framework and assume
$\delta_1\in[10^{-4},0.1]$, allowing for a small pre-existing resistant
population while maintaining a predominance of TMZ-sensitive cells at
treatment onset \cite{Sinelshchikov2025}. The parameter $\delta_2$ denotes
the initial fraction of tumor cells resistant to CAR-T-mediated killing.
Its range, $\delta_2\in[0.1,0.3]$, is motivated by the substantial
heterogeneity in target-antigen expression observed in MGs
\cite{ORourke2017}.

The proliferative fraction is characterized through the $\mbox{Ki-67}$ index, for which values between $0.1$ and $0.5$ are considered \cite{dahlrot2021prognostic}. Following the proliferative--quiescent framework \cite{ayala2021optimal,PeralesPaton2026}, $\beta_1$ and $\beta_2$ are not treated as independent parameters. For a certain $\mbox{Ki-67}$ value, the quiescent-to-proliferative transition rate is related to $\beta_1$ and $r_1$ through
\begin{equation}
\label{eq:flowQP}
\beta_2=
\mathrm{\mbox{Ki-67}}
\left(
\frac{\beta_1}{1-\mbox{Ki-67}}-r_1
\right).
\end{equation}
Accordingly, admissible parameter combinations must satisfy
\begin{equation}
\beta_1\geq r_1\left(1-\mathrm{Ki\text{-}67}\right), 
\end{equation}
which guarantees $\beta_2\geq0$. This relationship defines the basal balance between proliferative and quiescent cells used to initialize the model.

\begin{table*}
\caption{Model parameters, sampling ranges, units, and corresponding sources. DF means derived from.}
\label{tab:model_parameters}
\centering
\small
\begin{tabular}{llccl}
\textbf{Par.} & \textbf{Description} & \textbf{Range} & \textbf{Unit} & \textbf{Ref.} \\
\hline

\multicolumn{5}{@{}l}{\textit{Tumor growth, phenotypic switching, and initial composition}}\\
\addlinespace[2pt]
$r_1$
& Intrinsic proliferation rate of $S$ and $R_C$
& $[0.001,\,0.25]$
& d$^{-1}$
& DF \cite{Ode2,Sinelshchikov2025}
\\
$r_2$
& Intrinsic proliferation rate of TMZ-resistant cells $R_E$
&  $\frac{r_1}{2}$ 
& d$^{-1}$
& DF \cite{Sinelshchikov2025}
\\
$K$
& Tumor carrying capacity
& $5\times10^{12}$
& cells
& \cite{Forys}
\\
$\beta_1$
& Transition rate from proliferative to quiescent state
& $[0.2,\,0.6]$
& d$^{-1}$
& \cite{segura2022optimal,PeralesPaton2026}
\\
$\beta_2$
& Transition rate from quiescent to proliferative state
& Eq.~\eqref{eq:flowQP} 
& d$^{-1}$
& \cite{ayala2021optimal}
\\
$\mathrm{Ki\text{-}67}$
& Basal proliferative fraction
& $[0.1,\,0.5]$
& --
& \cite{dahlrot2021prognostic,PeralesPaton2026}
\\
$\delta_1$
& Initial fraction of TMZ-resistant tumor cells
& $[10^{-4},\,0.1]$
& --
& 
\cite{Sinelshchikov2025}
\\
$\delta_2$
& Initial fraction of CAR-T-resistant tumor cells
& $[0.1,\,0.3]$
& --
&\cite{ORourke2017}
\\
$\tau$
& Damaged-cell clearance rate
& $\left[r_1/100,\,2r_1\right]$
& d$^{-1}$
& \cite{joiner2009basic}
\\
$T_0$
& Initial tumor size
& $[10^7,\,10^{9}]$
& cells
& \cite{Sinelshchikov2025}
\\
\hline
\addlinespace[4pt]
\multicolumn{5}{@{}l}{\textit{CAR-T-cell dynamics}}\\
\addlinespace[2pt]
$\rho_1$
& Natural CAR-T-cell loss/inactivation rate
& $[1/30,\,1/7]$
& d$^{-1}$
& \cite{Ghorashian}
\\
$\rho_2$
& Maximal CAR-T stimulation rate by $S$
& $[0.1,\, 9]$
& d$^{-1}$
& \cite{Ode2}
\\
$\rho_3$
& Maximal CAR-T stimulation rate by $R_E$
& $[0.1,\, 9]$
& d$^{-1}$
& \cite{Ode2}
\\
$\rho_4$
& Tumor-mediated CAR-T inactivation rate
& $[0.01,\,0.2]$
& d$^{-1}$
& \cite{Sinelshchikov2025}
\\
$g_1$
& $S$ population for half-maximal CAR-T stimulation
& $1\times10^{10}$
& cells
& \cite{Ode2}
\\
$g_2$
& $R_E$ population for half-maximal CAR-T stimulation
& $1\times10^{10}$
& cells
& \cite{Ode2}
\\
$g_3$
& CAR-T concentration for half-maximal tumor inactivation
& $2\times10^9$
& cells
& \cite{Ode2}
\\
$\alpha_2$
& CAR-T killing coefficient against antigen-positive tumor cells
& $2.5\times10^{-10}$
& d$^{-1}$ cell$^{-1}$
& \cite{Ode2}
\\
$v$
& CAR-T cells administered per infusion
& $[10^7,\,10^9]$
& cells
& \cite{goff2019pilot,Sinelshchikov2025}
\\
\hline
\addlinespace[4pt]
\multicolumn{5}{@{}l}{\textit{Temozolomide dynamics and response}}\\
\addlinespace[2pt]

$\alpha_1$
& TMZ-induced tumor-cell killing rate
& $[0.1,\,1]$
& d$^{-1}$
& \cite{Delobel_PLOS}
\\
$\epsilon_1$
& TMZ-induced transition rate to resistance
& $[0.1,\,0.6]$
& d$^{-1}$
& \cite{Delobel_PLOS}
\\
$\alpha_3$
& TMZ-induced CAR-T-cell killing rate
& $\alpha_3=\alpha_1$
& d$^{-1}$
& \cite{Sinelshchikov2025}
\\
$\mu$
& TMZ first-order decay rate
& $8.32$
& d$^{-1}$
& \cite{Delobel_PLOS}
\\
$E_0$
& Normalized TMZ efficacy increment per administration
& $[0,\,1]$
& --
& \cite{PeralesPaton2026}
\\
\hline
\addlinespace[4pt]
\multicolumn{5}{@{}l}{\textit{Radiotherapy response}}\\
\addlinespace[2pt]
$\gamma$
& Fraction of quiescent cells induced to proliferate by  RT
& $[0,\,1]$
& --
& \cite{PeralesPaton2026}
\\
$\alpha$
& Linear LQ coefficient for proliferative MG cells
& $[0.1,\,0.2]$
& Gy$^{-1}$
&\cite{Pedicini2014,vanLeeuwen2018}
\\
$\frac{\alpha}{\beta}$
& Ratio of linear and quadratic LQ coefficients
& $[5,\,10]$
& Gy
&\cite{Pedicini2014,vanLeeuwen2018}
\\
$\alpha_T$
& Linear radiosensitivity coefficient for T cells
& $[0.43,\,0.49]$
& Gy$^{-1}$
& \cite{paganetti2023review}
\\
$d_{\mathrm{RT}}$
& RT dose per fraction
& $2$
& Gy
& \cite{Stupp2005}
\\

\end{tabular}
\end{table*}

Finally, RT-related parameters account for interpatient variability in radiosensitivity. For proliferative MG cells, the LQ parameters are informed by clinical estimates: Pedicini et al. reported $\alpha=0.12~\mathrm{Gy}^{-1}$ (95\% CI: $0.10$--$0.14~\mathrm{Gy}^{-1}$) and $\alpha/\beta=8~\mathrm{Gy}$ (95\% CI: $5.0$--$10.8~\mathrm{Gy}$) for GBM \cite{Pedicini2014}. Based on these estimates and the broader clinical ranges reviewed in \cite{vanLeeuwen2018}, we sample $\alpha$ from $[0.1,0.2]~\mathrm{Gy}^{-1}$ and $\alpha/\beta$ from $[5,10]~\mathrm{Gy}$, with $\beta$ subsequently computed as $\beta=\frac{\alpha}{\alpha/\beta}$. Instead, CAR-T radiosensitivity is 
based on the T-cell radiosensitivity estimates reviewed by Paganetti \cite{paganetti2023review}: linear dose-response curve with $\alpha_T\in[0.43,0.49]~\mathrm{Gy}^{-1}$. The standard RT dose per fraction is fixed at $d_{\mathrm{RT}}=2$ Gy in the reference CRT protocols \cite{Stupp2005}.

\subsection{Virtual clinic: virtual patients and \textit{in silico} trials}\label{sec:methods_vp_insilico_trials}

Virtual-patient approaches extend mechanistic mathematical models from the analysis of a single parameterization to the study of heterogeneous populations. In this and our previous frameworks \cite{Sinelshchikov2025,PeralesPaton2026}, a VP is a biologically plausible realization of the model defined by a specific set of parameter values and initial conditions, representing one possible disease and treatment-response phenotype within the population \cite{Allen2016VPop,Craig2023VCT}. 
VPs are not intended to reproduce specific real individuals, but rather to capture plausible inter-patient variability in the biological processes represented by the model. By contrast, when a mechanistic model is initialized and personalized using subject-specific data, it can provide the basis for constructing a patient-specific digital twin aimed at representing and predicting the evolution of that particular patient \cite{Wu2022DigitalTwin}. Thus, VP cohorts enable the study of how variability in biologically relevant characteristics propagates into heterogeneity in disease progression and therapeutic response.

Formally, let $\boldsymbol{\theta}$ denote the vector of model parameters, partitioned into fixed parameters $\boldsymbol{\theta}_{F}$ and parameters $\boldsymbol{\theta}_{V}$ that capture inter-patient heterogeneity: $\boldsymbol{\theta}=\left(\boldsymbol{\theta}_{F},\boldsymbol{\theta}_{V}\right).$
Each VP $VP_i$ is defined by a particular realization $\boldsymbol{\theta}_{V}^{(i)}$.  Parameters are dynamically sampled or held constant depending on whether the biological processes they govern exhibit interpatient heterogeneity or relative homogeneity within the clinical setting, respectively \cite{Craig2023VCT}. Specifically, parameters representing highly variable biological mechanisms among patients are allowed to vary to capture diverse progression and response profiles following biologically feasible ranges in Table~\ref{tab:model_parameters}. Crucially, parameters governing therapeutic interventions---such as treatment doses, administration schedules, and sequence order---do not define the intrinsic characteristics of the VPs; rather, they strictly delineate the operational parameters of the specific therapeutic protocol being evaluated.

The VP cohort is constructed using the biologically and clinically supported parameter ranges summarized in Table~\ref{tab:model_parameters}. Whenever sufficiently characterized population distributions are unavailable, parameters are sampled from uniform distributions over the corresponding admissible intervals, as commonly done in exploratory model-based virtual trials \cite{Craig2023VCT,Gevertz2024VCT}. Uniform sampling is used here as a bounded representation of the available uncertainty and should not be interpreted as evidence that the corresponding biological quantities are uniformly distributed in the clinical population. Indeed, the choice of prior parameter distributions can influence both the heterogeneity of the resulting virtual population and the predicted treatment outcomes \cite{Gevertz2024VCT}.

Note that each VP is characterized not only by different kinetic and treatment-response parameters but also by a distinct initial tumor composition. Thus, a literature-informed plausible cohort is constructed by sampling within experimentally and clinically supported ranges, and enforcing the model constraints described above.  Repeating the sampling procedure $N$ times generates a cohort $\mathcal{V}=\left\{VP_1,\ldots,VP_N\right\}.$ An \textit{in silico} trial is then performed by applying a prescribed treatment protocol to every member of $\mathcal{V}$ and recording the corresponding model outcomes. This strategy enables systematic evaluation of treatment schedules over a population of heterogeneous tumor and treatment-response phenotypes and has been used to investigate both treatment efficacy and robustness in mathematical oncology \cite{Barish2017VEPART,Craig2023VCT}.

An important advantage of the computational setting is that the same VP can be exposed independently to different therapeutic strategies \cite{Michael2024VirtualCohorts}. Accordingly, all protocols that are directly compared in this study are simulated using the same underlying VP cohort: the biological parameters and initial conditions of $VP_i$ remain unchanged, while only the treatment protocol is modified. This provides a matched comparison between interventions and prevents differences between independently sampled populations from confounding treatment effects. When we reproduce specific clinical trials in Section \ref{sec:num_simulations}, we sample a new VP cohort matching the size of the clinical study under reproduction. For a protocol $\mathcal{P}$, we denote the simulated survival time of VP $i$ by $ T_{s,i}^{\mathcal{P}}.$

Population-level treatment effects are evaluated from the distribution of simulated survival times. Kaplan--Meier (KM) curves and median OS are used to summarize survival under each protocol, and survival distributions are compared using two-sided log-rank tests, following standard clinical-trial methodology and our previous virtual-trial frameworks~\cite{Sinelshchikov2025,PeralesPaton2026}. Additionally, Cox proportional hazards regression is used to estimate Hazard Ratios (HRs). Since death is the event of interest, an HR $< 1$ denotes a reduction in the mortality rate, signifying therapeutic efficacy relative to the untreated reference.

For visualization of the temporal model dynamics, we additionally define the median virtual patient (MVP) as the parametrization obtained by assigning the median value of each VP parameter distribution. The MVP provides a representative parameter set with which to illustrate the evolution of tumor subpopulations and treatments. It should be noted that this definition does not imply that the MVP corresponds to an actual member of the sampled cohort or to the VP with median survival; population-level conclusions are always obtained from the complete virtual cohort.

The resulting virtual clinic is therefore used as a controlled computational environment in which biologically plausible heterogeneity can be propagated through the mechanistic model and virtual populations can be challenged with alternative therapeutic strategies. The purpose of these simulations is not to provide exact subject-level predictions, but to identify robust population-level trends, treatment schedules worthy of further investigation, and biological parameters potentially associated with prognosis or differential treatment benefit\cite{Viceconti2021InSilico}.

\section{Results}

\subsection{Mathematical analysis}

In this subsection we present the mathematical analysis of model \eqref{eq:new_model4}. First, we introduce the concept of essentially non-negative functions, which will be used to prove the biological relevance of system \eqref{eq:new_model4}. Consider a smooth function $f$ defined over some domain $\Omega\subset\mathbb{R}^n$ and its associated dynamical system
\begin{equation}\label{eq:defgeneraldynamicalsystem}
    \dot{x}=f(x).
\end{equation}

\begin{Definition}
    The function $f$ is said to be essentially non-negative if for every $i=1,\dotsc,n$ and $x=(x_1,\dotsc,x_i=0,\dotsc,x_n)\in\overline{\mathbb{R}}^n_+$, we have $f_i(x)\geq0$.
\end{Definition}

\begin{Proposition}[See Proposition 4.1 from \cite{nonnegativityTheorem}]\label{prop:nonnegativeequiv}
    Consider the dynamical system \eqref{eq:defgeneraldynamicalsystem} and suppose $\overline{\mathbb{R}}^n_+\subset\Omega$. Then, $\overline{\mathbb{R}}^n_+$ is an invariant surface of the system if and only if $f$ is essentially non-negative. 
\end{Proposition}

\begin{Proposition}\label{prop:existenceanduniqueness}
    For every non-negative initial condition, the system \eqref{eq:new_model4} admits a unique non-negative solution defined for every $t\geq0$. Furthermore, the system presents continuous dependence on the initial conditions and parameters.
\end{Proposition}

\begin{proof}
    Consider any non-negative initial condition $\mathbf{x}_0\in\overline{\mathbb{R}}^9_+$. We begin the proof by proving that the non-negative orthant, $\overline{\mathbb{R}}^9_+$, is invariant, which implies that for every initial condition $\mathbf{x}_0\in\overline{\mathbb{R}}^9_+$, any solution must be non-negative. In order to prove this, we use Proposition \eqref{prop:nonnegativeequiv}. We have
    \begin{eqnarray*}
        &f_1(0,Q,R_C,Q_C,R_E,Q_E,D,C,E)=\beta_2Q,\\ &f_2(S,0,R_C,Q_C,R_E,Q_E,D,C,E)=\beta_1S \\
        &f_3(S,Q,0,Q_C,R_E,Q_E,D,C,E)=\beta_2Q_C,\\
        &f_4(S,Q,R_C,0,R_E,Q_E,D,C,E)=\beta_1R_C,\\
        &f_5(S,Q,R_C,Q_C,0,Q_E,D,C,E)=\beta_2Q_E,\\
        &f_6(S,Q,R_C,Q_C,R_E,0,D,C,E)=\beta_1R_E,\\
        &f_7(S,Q,R_C,Q_C,R_E,Q_E,0,C,E)=\\&\alpha_1(S+R_C)E+\alpha_2(S+R_E)C,\\
        &f_8(S,Q,R_C,Q_C,R_E,Q_E,D,0,E)=0, \\
        &f_9(S,Q,R_C,Q_C,R_E,Q_E,D,C,0)=0.
    \end{eqnarray*}
    Since we are working with $x\in\overline{\mathbb{R}}^9_+$ and all parameters are non-negative, we have the first part of the result. From now on, we will consider only non-negative solutions for system \eqref{eq:new_model4}.

    We consider now the existence and uniqueness of solutions. First, notice that our system is smooth in the following domain
    \begin{equation*}
        \Omega=\{x\in\mathbb{R}^9\,|\, g_1+S>0, g_2+R_E>0, g_3+C>0\}
    \end{equation*}
    and we have $\overline{\mathbb{R}}^9_+\subset\Omega$. Thus, by Picard-Lindelöf, we have local existence and uniqueness for every non-negative initial condition. Furthermore, there is a maximal solution defined over some interval $(T_-(0,x_0), T_+(0,x_0))\subset \mathbb{R}$.
    
    We prove now that the solution is defined globally, that is, for every $t\geq0$. To this aim we use the extension theorem (see Corollary 2.15, Ref. \cite{Teschl}). First, notice that the equation related to $E$ has an explicit solution $E(t)=E_0e^{-\mu t}\leq E_0$. Now, using the non-negativity of solutions and parameters, we have 
    \begin{align}
        \dot{S}+\dot{Q}&=r_1S\left(1-\dfrac{T}{K}\right)-(\alpha_1+\epsilon_1)ES-\alpha_2CS\nonumber\\&\leq r_1S\leq r_1(S+Q).
    \end{align}
    Thus, by Grönwall's Inequality, we get 
    \begin{equation}
        S+Q\leq (S_0+Q_0)e^{r_1t}.
    \end{equation}
    Since $S,Q\leq S+Q$ (due to the non-negativity of solutions), we get that both $S$ and $Q$ are bounded by continuous function, that is, bounded for every finite $t\geq0$. 
    
    Using the same reasoning for the CAR-T-resistant subpopulation, we get
    \begin{equation}
        R_C+Q_C\leq (R_{C_0}+Q_{C_0})e^{r_1t},
    \end{equation}
    and therefore $R_C$ and $Q_C$ are also bounded by continuous functions. 

    Now, for the TMZ-resistant population we have
    \begin{eqnarray}
        &\dot{R_E}+\dot{Q_E}\leq r_2R_E+\epsilon_1(S+R_C)E_0 \nonumber\\
        &\leq r_2(R_E+Q_E)+\epsilon_1(S+Q+R_C+Q_C)E_0 \nonumber\\
        &\leq r_2(R_E+Q_E)+\epsilon_1(S_0+Q_0+R_{C_0}+Q_{C_0})E_0e^{r_1t}\nonumber\\
        &\equiv r_2(R_E+Q_E)+Ae^{r_1t},
    \end{eqnarray}
    where $A=\epsilon_1(S_0+Q_0+R_{C_0}+Q_{C_0})E_0$. 
    Therefore, by Grönwall's inequality, we obtain

\begin{equation}
R_E+Q_E
\leq
(R_{E0}+Q_{E0})e^{r_2t}
+
\begin{cases}
\dfrac{A}{r_1-r_2}
\left(e^{r_1t}-e^{r_2t}\right),
& r_1\neq r_2, \\[2mm]
At e^{r_1t},
& r_1=r_2.
\end{cases}
\label{eq:RE_bound}
\end{equation}

Indeed, when $r_1=r_2$, the integral term reduces to
$At e^{r_1t}$. Hence, in either case, $R_E+Q_E$ is bounded by a
continuous function on every finite time interval. 

    Regarding the CAR-T, using the fact that $\frac{a}{a+b}\leq1$ for any $a,b>0$ we have
    \begin{equation}
        \dot{C}\leq (\rho_2+\rho_3)C\Longrightarrow C\leq C_0e^{(\rho_2+\rho_3)t}
    \end{equation}

    Considering the damaged cells, we obtain
    \begin{equation}
        \dot{D}\leq \alpha_1(S+R_C)E+\alpha_2(S+R_E)C\leq g(t)\leq D+g(t).
    \end{equation}
    where $g$ is the combination of the bounds found for $S, R_C, R_E, C$. Therefore, we obtain that 
    \begin{equation}
        D\leq D_0e^t+\int_0^t g(s)e^{t-s}ds.
    \end{equation}
    This implies that the system is bounded by a continuous functions, and therefore, $\limsup_{t\to T^-_+(0, x_0)}\Vert f(x)\Vert<\infty$ and by the Extension theorem, we obtain the globality of the unique solution. 

    Finally, the continuous dependence over the initial conditions and parameters is ensured by noticing that the domain of $f$ can be extended to 
    \begin{equation}
        \tilde{\Omega}=\Omega\times\mathbb{R}^m_{>0}
    \end{equation}
    where $m=18$ is the number of parameters involved in the system, which is an open set containing all non-negative initial conditions. By using Theorems 1 and 2 from Section 2.3 Ref.\cite{perko2001}, we finish the result.
\end{proof}

\begin{Proposition}\label{pr:origin_unstable} The only tumor-free equilibrium of system \eqref{eq:new_model4} is the origin, which is a saddle for the biologically relevant values of the parameters. 
\end{Proposition}
\begin{proof}
Consider equilibria of \eqref{eq:new_model4} that correspond to complete tumor eradication, i.e. when $S=Q=R_{C}=Q_{C}=R_{E}=Q_{E}=0$. This yields that $E=C=D=0$. Thus, the only possible fixed point, which correspond to a tumor-free state is the origin. 

Assume that $r_{1}\neq\beta_{1}+\beta_{2}$ and $r_{2}\neq\beta_{1}+\beta_{2}$, then the eigenvalues of the Jacobian matrix of \eqref{eq:new_model4} evaluated at the $\mathcal{O}=(0,0,\dots,0)$ are 
\begin{equation}
\begin{gathered}
\lambda_{1}=-\rho_{1}, \quad \lambda_{2}=-\mu, \quad \lambda_{3}=-\tau,\\
\lambda_{4,5,6,7}=\frac{r_{1}-\beta_{1}-\beta_{2}}{2}\left[1\pm\sqrt{1+\frac{4\beta_{2}r_{1}}{(r_{1}-\beta_{1}-\beta_{2})^{2}}}\right],\\
\lambda_{8,9}=\frac{r_{2}-\beta_{1}-\beta_{2}}{2}\left[1\pm\sqrt{1+\frac{4\beta_{2}r_{2}}{(r_{2}-\beta_{1}-\beta_{2})^{2}}}\right].
\end{gathered}
\end{equation}
As a result, since all the parameters are positive, it is clear that the origin is a saddle, since all eigenvalues are not zero and the last three pairs of them have opposite signs in each pair. 

Suppose that one or both  of the relations $r_{1}=\beta_{1}+\beta_{2}$, $r_{2}=\beta_{1}+\beta_{2}$ holds. Then, one can verify that for positive values of the parameters, the Jacobian matrix has real eigenvalues with three pairs of eigenvalues of opposite signs. This completes the proof.
\end{proof}

Taking into account the structure of our system, we also get the immediate result
\begin{Proposition}
    System \eqref{eq:new_model4} admits two invariant surfaces defined by $H_1=C$ and $H_2=E$. 
\end{Proposition}

Now let us consider the dynamics of \eqref{eq:new_model4} when there is no treatment, i.e. when $E=C=0$. Recall, that we assume that $D(0)=0$. Thus, the condition $E=C=0$ implies the $D=0$ for $t>0$.

\begin{Proposition}\label{prop:unique_linear_invariant}
Suppose that $r_{2}\neq r_{1}$. Then, at $C=E=D=0$ the hyperplane defined by $H_{1}=T_{1}-K$ is the unique biologically relevant polynomial invariant surface of degree 1 of \eqref{eq:new_model4}, where $T_{1} = S + Q + R_C + Q_C + R_E + Q_E.$.
\end{Proposition}
\begin{proof}
Since under the absence of treatment system \eqref{eq:new_model4} is polynomial, we look for a polynomial invariant surfaces, We begin with degree 1 surface of the form
\begin{equation}\label{eq:F_candidate}
    H_{2} = a_0 + a_1 S + a_2 Q + a_3 R_C + a_4 Q_C + a_5 R_E + a_6 Q_E, \quad a_i \in\mathbb{C},\quad i=0,1,\ldots6.
\end{equation}
Let $\mathcal{X}$ be the vector field associated to \eqref{eq:new_model4} at $C=E=D=0$. Then, $H_{2}$ is an invariant surface if
\begin{equation}\label{eq:Darboux_condition}
\mathcal{X}(F) = k F,
\end{equation}
where $k$ is the associated cofactor. Since we look for polynomial invariants of a polynomial dynamical system, we also assume that $k$ is a polynomial. Since $\deg \mathcal{X}$ is $2$, then the maximal degree of $k$ is 1.

Thus, the cofactor $k$ is necessarily of the form,
\begin{equation}\label{eq:cofactor_k}
    k = c_0 + c_1 S + c_2 Q + c_3 R_C + c_4 Q_C + c_5 R_E + c_6 Q_E,\quad c_{i} \in \mathbb{C},\,i=0,1,\ldots 6.
\end{equation}

By substituting \eqref{eq:F_candidate} and \eqref{eq:cofactor_k} into  \eqref{eq:Darboux_condition} and collecting coefficients at the monomials of the same degrees we obtain a system of algebraic equations.

We begin with homogeneous terms of degree 2. Let $F_1$ and $k_1$ denote the homogeneous parts of degree 1 of the polynomials $F$ and $k$, respectively. Extracting the terms of degree 2 from equation \eqref{eq:Darboux_condition} we get
\begin{equation}\label{eq:degree_2_expansion}
    -\frac{1}{K} \left( a_1 r_1 S + a_3 r_1 R_C + a_5 r_2 R_E \right) T_{1} = k_1 F_1.
\end{equation}
Since the polynomial ring over $\mathbb{C}[S, Q, R_C, Q_C, R_E, Q_E]$ is a unique factorization domain and $T_{1}$ is an irreducible polynomial, $T_{1}$ must strictly divide either $k_1$ or $F_1$. 

If $T$ divides $k_1$, then $k_1 = c T$ for some constant $c$. This would force $F_1$ to be proportional to $(a_1 r_1 S + a_3 r_1 R_C + a_5 r_2 R_E)$, implying that the coefficients for the quiescent compartments are zero ($a_2 = a_4 = a_6 = 0$). Furthermore, using this relation in equation \eqref{eq:degree_2_expansion} yields
\begin{equation}
a_1\left(c+\frac{r_1}{K}\right)=0,\quad  a_3\left(c+\frac{r_1}{K}\right)=0,\quad  a_5\left(c+\frac{r_2}{K}\right)=0.
\end{equation}
If $a_i=0$ for $i=1,3,5$ we get the trivial solution $F\equiv a_0$. Let then $r_1\neq r_2$ and $c=-\frac{r_1}{K}$, with $a_1,a_3\neq 0$ and $a_5=0$. Analyzing the terms of degree 1 of \eqref{eq:Darboux_condition} we get the relation
\begin{equation}
    \mathcal{X}_1(F) = c_0 F_1 + a_0 k_1,
\end{equation}
where $\mathcal{X}_1(F)$ represents the linear terms of the total derivative. Substituting $F_1 = a_1 S + a_3 R_C$ and $k_1 = -\frac{r_1}{K}T_{1}$, we get
\begin{align}
    a_1 (r_1 S - \beta_1 S + \beta_2 Q) + a_3 (r_1 R_C - \beta_1 R_C + \beta_2 Q_C) = c_0 (a_1 S + a_3 R_C) - a_0 \frac{r_1}{K} T_{1}.
\end{align}
Since the variable $Q_E$ does not appear on the left-hand side but is contained within $T_{1}$ on the right-hand side, we get $0 = -a_0 \frac{r_1}{K}$ and therefore $a_0=0$ (recall that $r_1, K>0$). Substituting $a_0 = 0$ back into the equation and equating the coefficients for the quiescent variables $Q$ and $Q_C$, we obtain
\begin{equation}
    a_1 \beta_2 = 0 \quad \text{and} \quad a_3 \beta_2 = 0.
\end{equation}
Since the transition rate is strictly positive ($\beta_2 > 0$), this implies $a_1 = 0$ and $a_3 = 0$, which contradicts our initial assumption and collapses the polynomial to the trivial solution $F \equiv 0$. The symmetric case for $c = -\frac{r_2}{K}$ (where $a_5 \neq 0$ and $a_1 = a_3 = 0$) analogously collapses to $F \equiv 0$.

Recall that $r_{2}\neq r_{1}$ and, hence, we do not consider this case.

As a result, under the case $T_{1}|k_1$ we get $a_i=0$ for $i=1,\dotsc,6$ as the only biologically relevant solution, that is, a trivial invariant surface.

Therefore, for a non-trivial generic solution, $T_{1}$ must divide $F_1$. Given that $F_1$ is also a polynomial of degree 1, they must be strictly proportional: $F_1 = a T_{1}$ for some constant $a \in \mathbb{R}$, that is,
\begin{equation}
    a_1 = a_2 = a_3 = a_4 = a_5 = a_6 \equiv a.
\end{equation}
Substituting $F_1 = a T_{1}$ back into \eqref{eq:degree_2_expansion} and cancelling the common factor $T_{1}$, we obtain the explicit form of the linear part of the cofactor:
\begin{equation}\label{eq:cofactor_k1}
    k_1 = -\frac{1}{K} \left( r_1 S + r_1 R_C + r_2 R_E \right).
\end{equation}

Now we deal with degree 1 terms. Having established that $F_1 = a T_{1}$, our candidate polynomial takes the form $F = a_0 + a T_{1}$. Consequently, the total derivative of $F$ along the vector field is:
\begin{equation}
    \dot{F} = a (\dot{S} + \dot{Q} + \dot{R}_C + \dot{Q}_C + \dot{R}_E + \dot{Q}_E).
\end{equation}
Substituting the equations of the untreated yields
\begin{equation}
    \dot{F} = a \left( r_1 S + r_1 R_C + r_2 R_E \right) \left( 1 - \frac{T_{1}}{K} \right).
\end{equation}
From this expression, the homogeneous part of degree 1 is strictly $a(r_1 S + r_1 R_C + r_2 R_E)$. 

On the right-hand side of the Darboux condition \eqref{eq:Darboux_condition}, the expansion of the product $k F = (c_0 + k_1)(a_0 + a T)$ yields the degree 1 terms $c_0 a T + a_0 k_1$. Equating the degree 1 polynomials from both sides gives the relation
\begin{equation}\label{eq:degree_1_general}
    a(r_1 S + r_1 R_C + r_2 R_E) = c_0 a T_{1} + a_0 k_1.
\end{equation}

We now demonstrate that a vanishing independent term ($a_0 = 0$) leads to the trivial solution. If $a_0 = 0$, equation \eqref{eq:degree_1_general} simplifies to 
\begin{equation}\label{eq:degree_1_a0_zero}
    a(r_1 S + r_1 R_C + r_2 R_E) = c_0 a (S + Q + R_C + Q_C + R_E + Q_E).
\end{equation}
By equating the coefficients of the specific variables on both sides, we form a linear system of constraints. For any of the quiescent variables, the coefficient on the left-hand side is $0$, while on the right-hand side it is $c_0 a$. Thus,
\begin{equation}
    c_0 a = 0.
\end{equation}
For any of the corresponding proliferative variables, equating the coefficients yields
\begin{equation}
    a r_1 = c_0 a.
\end{equation}
Substituting the constraint $c_0 a = 0$ into the equation results in $a r_1 = 0$. Since $r_1 > 0$, the only biologically valid solution is
\begin{equation}
    a = 0.
\end{equation}
Therefore, if $a_0 = 0$, the polynomial collapses into the trivial solution $F = 0$. This confirms that no partial or purely homogeneous linear invariant surfaces can exist within this dynamical system.

Finally, we construct the non-trivial invariant surface of \eqref{eq:new_model4} at $C=E=D=0$. We now analyse the case where the independent term does not vanish ($a_0 \neq 0$). We return to the degree 1 relation established in \eqref{eq:degree_1_general}
\begin{equation}
    a(r_1 S + r_1 R_C + r_2 R_E) = c_0 a T_{1} + a_0 k_1.
\end{equation}
Substituting the explicit expression for the linear cofactor $k_1$ derived in \eqref{eq:cofactor_k1}, the equation becomes
\begin{equation}
    a(r_1 S + r_1 R_C + r_2 R_E) = c_0 a T_{1} - \frac{a_0}{K}(r_1 S + r_1 R_C + r_2 R_E).
\end{equation}
Grouping the terms corresponding to the proliferative variables on the left side, we obtain
\begin{equation}\label{eq:step3_grouped}
    \left( a + \frac{a_0}{K} \right) (r_1 S + r_1 R_C + r_2 R_E) = c_0 a T_{1}.
\end{equation}
Recall that the total tumor population is $T_{1} = S + Q + R_C + Q_C + R_E + Q_E$. The left-hand side of equation \eqref{eq:step3_grouped} depends solely on the proliferative variables and does not contain quiescent terms. However, the right-hand side contains the quiescent variables multiplied by the constant $c_0 a$. 

By equating the coefficients for any quiescent variable, we find:
\begin{equation}
    0 = c_0 a.
\end{equation}
Since we are seeking a non-trivial polynomial ($F \neq 0$), we must have 
\begin{equation}
    c_0 = 0.
\end{equation}
Substituting $c_0 = 0$ into \eqref{eq:step3_grouped} yields
\begin{equation}
    \left( a + \frac{a_0}{K} \right) (r_1 S + r_1 R_C + r_2 R_E) = 0.
\end{equation}
Since the growth rates are strictly positive ($r_1, r_2 > 0$) and the variables are not identically zero, the bracketed term must vanish
\begin{equation}
    a + \frac{a_0}{K} = 0 \implies a_0 = -a K.
\end{equation}
Finally, substituting this relationship into our homogeneous structure $F = a_0 + a T_{1}$ we get
\begin{equation}
    F = -a K + a T_{1} = a (T_{1} - K).
\end{equation}
As a result, up to re-scaling, we find the only invariant surface of degree one,
\begin{equation}
    F = T_{1} - K = 0.
\end{equation}
This completes the proof.

\end{proof}

Now we demonstrate that there is submanifold of $H_{1}$ which consists of equilibria of \eqref{eq:new_model4} at $C=E=D=0$ and it is locally asymptotically stable.

\begin{Proposition}\label{pr:inv_surf}
    Under the absence of therapies ($C=E=D=0$), system \eqref{eq:new_model4} admits a stable invariant surface given by   $H_{2}=\{ (S,Q,R_{C},Q_{C},R_{E},Q_{E}):  S+Q+R_{C}+Q_{C}+R_{E}+Q_{E}-K=0,\quad   
    Q-\eta S=0, \quad Q_C-\eta R_C=0,   \quad  Q_E -\eta R_E=0 \}. $
\end{Proposition}
\begin{proof}
Let $C=D=E=0$ and $\eta=\beta_1/\beta_2$. Then, one can verify that the set defined by
     \begin{multline}
     \label{eq:inv_surf}
   H_{2}=\{ (S,Q,R_{C},Q_{C},R_{E},Q_{E}):  \\ S+Q+R_{C}+Q_{C}+R_{E}+Q_{E}-K=0,  \\
    Q-\eta S=0, \quad Q_C-\eta R_C=0,   \quad  Q_E -\eta R_E=0 \}
     \end{multline}
is invariant with respect to the flow of \eqref{eq:new_model4} when $E=C=D=0$.  The dimension of $H_{2}$ is 2. 

The Jacobian matrix of \eqref{eq:new_model4} on $H_{2}$ is a block upper-triangular one and its structure is
\begin{equation}
J =
\left(
\begin{array}{cc}
A_{6,6} & B_{6,3} \\ 
0 & C_{3,3}
\end{array}
\right),
\end{equation}
where $A_{6,6}$ and $B_{6,3}$ are $6\times6$ and $6\times3$ matrices and $C_{3,3}$ is a $3\times3$ upper triangular matrix. Therefore, we have that
\begin{equation}
    \det (J-\lambda I)=\det(A_{6,6}-\lambda I)\det(C_{3,3}-\lambda I).
\end{equation}
The eigenvalues that correspond to the last three equations of \eqref{eq:new_model4} are on the diagonal of $C_{3,3}$ and have the form
    \begin{align}
    \label{eq:ev_7_9}
    &\lambda_{7}=-\tau,\nonumber\\ 
    &\lambda_{8}=-\rho_1+\rho_{2}+\rho_3-\rho_4\dfrac{K}{g_3}-\frac{g_{1}\rho_{2}}{S+g_{1}}-\frac{g_{2}\rho_{3}(\beta_1+\beta_2)}{\beta_2K-(\beta_1+\beta_2)(S+R_C-g_2)}, \\ 
    &\lambda_{9}=-\mu.\nonumber
        \end{align}
Notice that below we demonstrate that the denominator in $\lambda_{8}$ does not vanish for biologically relevant values of the parameters.
Clearly, $\lambda_{7}$ and $\lambda_{9}$ are negative for biologically relevant values of the parameters. In fact, the same holds for $\lambda_{8}$ since using $S+R_C=K/(1+\eta)-R_E=\beta_2K/(\beta_1+\beta_2)-R_E$, we get
\begin{align}
    \lambda_8&<-\rho_1-\rho_4\frac{K}{g_3}+\rho_{2}+\rho_{3}-\frac{g_{2}\rho_{3}}{R_{E}+g_{2}}<-\rho_1-\rho_4\frac{K}{g_3}+\rho_2+\rho_{3}<0,
\end{align}
The last inequality holds for the ranges of the parameters presented in Table \ref{tab:model_parameters}. One can also see that the denominator in $\lambda_{8}$ is positive for positive values of the parameters.

The rest of eigenvalues of $J$ are eigenvalues of $A_{6,6}$ and have the form
    \begin{equation}
    \begin{gathered}
        \lambda_{1,2}=0,  \quad \lambda_{3,4,5}=-(\beta_1+\beta_2),\quad  \lambda_{6}=-\frac{\beta_2r_2}{\beta_1+\beta_2}+\dfrac{(r_2-r_1)(S+R_C)}{K}.
    \end{gathered}
    \end{equation}

Notice that two first eigenvalues are zero. Clearly, the next three ones are negative. One can also see that $\lambda_6<0$ if $r_2\leq r_1$. If $r_2>r_1$, using $S+R_C=K/(1+\eta)-R_E=\beta_2K/(\beta_1+\beta_2)-R_E$ we have
\begin{align}
    \lambda_6&=\dfrac{\beta_2}{\beta_1+\beta_2}\left(-r_2+(r_2-r_1)\right)-\frac{r_2-r_1}{K}R_C=-\left(\dfrac{r_1}{1+\eta}+\dfrac{(r_2-r_1)R_E}{K}\right)<0.
\end{align}

Thus, we have that at $H_{2}$ all eigenvalues of the Jacobian matrix for \eqref{eq:new_model4} are negative except for two zero eigenvalues. 

Now we demonstrate that eigenspace of these two zero eigenvalues of $A_{6,6}$ coincides with the tangent space of $H_{2}$. Let us introduce the vector space $\mathcal{V}=\{\emph{v}=(\emph{v}_{S},v_{Q},\emph{v}_{R_{C}},\emph{v}_{Q_{C}},\emph{v}_{R_{E}},\emph{v}_{Q_{E}})\subseteq\mathbb{R}^{6}\}$. Then, one can see that
\begin{multline}
\label{eq:TpH1}
    T_{p}H_{2}=\{ v\in\mathbb{R}^{6}:v_{S}+v_{Q}+v_{R_{C}}+v_{Q_{C}}+v_{R_{E}}+v_{Q_{E}}=0, \\ v_{Q}-\eta v_{S}=0, \quad v_{Q_{C}}-\eta v_{R_{C}}=0, \quad v_{Q_{E}}-\eta v_{R_{E}}=0 \}.
\end{multline}
It is straightforward to verify that any vector $\emph{v}\in\mathcal{V}$ of the form \eqref{eq:TpH1} satisfies $A_{6,6}\emph{v}=0$. Thus, for any $\emph{v}\in T_{p}H_{2}$ we have that $A_{6,6}v=0$. Conversely, suppose that $A_{6,6}\emph{v}=0$. Then, from the second, fourth and sixth equations of this algebraic system we obtain the last three conditions from \eqref{eq:TpH1}. Substituting them to the rest of the equations we find that they are reduced to the first condition from \eqref{eq:TpH1}. Therefore, we have that $\ker A_{6,6}=T_pH_{2}$ and the geometric multiplicity of $\lambda_{1,2}$ is $2$. 

As a consequence, we obtain that two zero eigenvalues $\lambda_{1,2}$ correspond to the directions tangent to $H_{2}$. The rest of the eigenvalues are negative. Therefore, according to normal hyperbolicity theorem (see, e.g. \cite{Fenichel1971,Tallapragada2017} ) the invariant surface $H_{2}$ is locally exponentially attracting.

\end{proof}

One can also verify that all nontrivial untreated equilibria with positive
tumor burden lie on the invariant surface $H_1$. Thus, taking into account Propositions \ref{pr:origin_unstable} and \ref{pr:inv_surf} it is clear that in the absence of treatments the tumor free state is not reachable.

Now we consider the opposite situation, when we constantly apply all three treatments (RT, TMZ and CAR-T). The mathematical model that describes this scenario is developed in Section \ref{sec:app_1}. Our aim now is to demonstrate that a cancer-free equilibrium, i.e., a point of the form $P=(0,\dots,0,C^*,E^*)$, exists for a wide range of biologically relevant values of the parameters.

\begin{Proposition}
System \eqref{eq:new_model_const_treatment} with constant treatment admits a locally asymptotically stable tumor-free equilibrium of the form $P=(0,0,\ldots,C^*,E^*)$, where 
    \begin{equation}\label{eq:E^*,C^*}
    E^*= \bar{E}_0/\mu,\qquad
    C^*=\frac{\bar{C}_0}{(\rho_1+\alpha_6)+\alpha_3\bar{E}_0/\mu},
    \end{equation}
    if and only if
    \begin{equation}\label{eq:E_0> and C_0>}
    (\alpha_1+\varepsilon_1) \bar{E}_0 > \mu\,(r_1-\alpha_4),
    \bar{C}_0 > \big((\rho_1+\alpha_6)+\tfrac{\alpha_3}{\mu}\bar{E}_0\big)\,\frac{r_2-\alpha_5}{\alpha_2}.
    \end{equation}
\end{Proposition}
\begin{proof}
Consider the tumor-free equilibrium point $P=(0,\dots,0,C^*,E^*)$ for system \eqref{eq:new_model_const_treatment}. Substituting $P$ into \eqref{eq:new_model_const_treatment} we find that
\begin{equation}
E^*= \bar{E}_0/\mu,\quad C^*=\frac{\bar{C}_0}{(\rho_1+\alpha_6)+\alpha_3\bar{E}_0/\mu}.
\end{equation}

The characteristic polynomial of the Jacobian matrix of \eqref{eq:new_model_const_treatment} at $P$ is
\begin{multline}
\label{eq:charat_polyn}
 (\tau +\lambda)(\mu+\lambda)(\alpha_3E^* +(\rho_1+\alpha_6) + \lambda)\left[(a-\lambda)\left(\beta_2+\gamma_1+\lambda\right)+(\beta_2+\gamma_1)\beta_1\right] \times
\\  \left[(b-\lambda)(\beta_2+\gamma_1+\lambda)+(\beta_2+\gamma_1)\beta_1\right] \left[(a+\alpha_2C^*-\lambda)(-\beta_2-\gamma_1-\lambda)-\beta_1(\beta_2+\gamma_1)\right]=0,
\end{multline}
where 
\begin{align}\label{eq:expressions_for_a_b}
a = r_1-\beta_1-\alpha_4-(\alpha_1+\varepsilon_1)E^*-\alpha_2 C^*, \quad
b = r_2-\beta_1-\alpha_5-\alpha_2 C^*.
\end{align}
One can see that the first three eigenvalues are given by
\begin{align}
   \lambda_1 = -\tau, &\quad 
\lambda_2 = -\mu, \nonumber\\
\lambda_3 = -(\rho_1+\alpha_6) - \alpha_3 E^*&= -(\rho_1+\alpha_6) - \alpha_3 \bar{E}_0/\mu, 
\end{align}
which are always negative for the biologically relevant values of the parameters. 

Now, we focus on the three quadratic equations in square brackets (see \eqref{eq:charat_polyn}). We need to find conditions when real parts of the solutions of these quadratic equations are negative. To this aim we use the Routh-Hurwitz criterion,  which requires that all coefficients of these quadratic equations to be positive. 

From the first quadratic with respect to $\lambda$ factor of \eqref{eq:charat_polyn} we get
\begin{equation}
    \lambda^2 + (\beta_2 + \gamma_1 - a)\lambda - (\beta_2 + \gamma_1)(\beta_1 + a) = 0.
\end{equation}
Applying the Routh-Hurwitz criterion, we find the first stability condition
\begin{equation}\label{eq:cond1}
a < -\beta_1.
\end{equation}

The second quadratic with respect to $\lambda$ factor of \eqref{eq:charat_polyn} is
\begin{equation}
\lambda^2 + (\beta_2 + \gamma_1 - b)\lambda - (\beta_2 + \gamma_1)(b + \beta_1) = 0.
\end{equation}
Applying the Routh--Hurwitz conditions as before, we obtain
\begin{equation}\label{eq:cond2}
b < -\beta_1.
\end{equation}
The last factor of \eqref{eq:charat_polyn}  yields
\begin{equation}
\lambda^2 + [(\beta_2 + \gamma_1) - (a + \alpha_2 C^*)]\lambda - (\beta_2 + \gamma_1)(\beta_1 + a + \alpha_2 C^*) = 0,
\end{equation}
The Routh--Hurwitz conditions give
\begin{equation}\label{eq:cond3}
a + \alpha_2 C^* < -\beta_1.
\end{equation}
Therefore, since \eqref{eq:cond3} implies \eqref{eq:cond1}, we get the following two conditions in terms of $a$ and $b$,
\begin{equation}\label{eq:final_stability_cond}
    b<-\beta_1,\qquad a+\alpha_2C^*<-\beta_1.
\end{equation}
Taking into account \eqref{eq:expressions_for_a_b}, we obtain
\begin{equation}
    r_2<\alpha_5+\alpha_2C^*,\qquad r_1<\alpha_4+(\alpha_1+\epsilon_1)E^*.
\end{equation}
Finally, using \eqref{eq:E^*,C^*}, we obtain that the tumor-free equilibrium is locally asymptotically stable if and only if
\begin{equation}
\label{eq:tumor_free_eq}
    (\alpha_1+\epsilon_1)\bar{E}_0>\mu( r_1-\alpha_4),\qquad \alpha_{2}\bar{C}_0>\Big(\rho_1+\alpha_6+\dfrac{\alpha_3}{\mu}\bar{E}_0\Big )(r_2-\alpha_5).
\end{equation}
As a final remark, notice that if $r_1\leq\alpha_4$ or $r_2\leq\alpha_5$, the corresponding inequality is automatically satisfied. This is a result of RT successfully controlling the tumor through terms $\alpha_4$ and $\alpha_5$. Furthermore, replacing $\bar{E}_0$ in the second condition, we have a bound involving solely the CAR-T,
\begin{equation}
    \bar{C}_0>\left((\rho_1+\alpha_6)+\alpha_3\dfrac{r_1-\alpha_4}{\alpha_1+\epsilon_1}\right)\dfrac{r_2-\alpha_5}{\alpha_2}.
\end{equation}
This completes the proof.
\end{proof}

Therefore, we see that in the case of constant treatment there exists critical values of TMZ and CAR-T dosages that guarantee local stability of the tumor-free equilibrium.  Indeed, regulating 
CAR-T and TMZ dosages ($\bar{E}_{0}$ and $\bar{C}_{0}$), one can satisfy inequalities \eqref{eq:tumor_free_eq}. On the other hand, the satbility of this tumor-free equilibrium can be also regulated by the  efficacy of CAR-T, TMZ and radio therapies (the corresponding parameters are $\alpha_{1}$, $\alpha_{2}$, $\alpha_{4}$ and $\alpha_{5}$).

\subsection{Numerical simulations} \label{sec:num_simulations}

We now investigate clinical scenarios with an impulsive-treatment formulation of the model through \textit{in silico} trials. We first assess the clinical plausibility of the model by benchmarking its predictions against published outcomes for untreated disease and for TMZ, RT, RT+TMZ, and CAR-T therapies. Here, we generate virtual cohorts and perform \textit{in silico} trials reproducing the clinical trials with the highest possible fidelity (at least same number of patients, initial tumor sizes, treatment scheduling and doses). Next, the numerical study is organized around a central question: whether the incorporation of CAR-T cell therapy into standard chemoradiotherapy treatment, the Stupp protocol \cite{Stupp2005}, provides a meaningful therapeutic benefit in a heterogeneous population of VPs, and whether such a benefit is better understood at the population level or through patient-specific treatment response. Here, we use a large heterogeneous virtual cohort to quantify the population-level effect of incorporating CAR-T cells into the Stupp protocol and to determine whether treatment order, dose fractionation, and temporal spacing substantially modify this effect. Finally, we investigate the heterogeneity of individual therapeutic benefit, identify model parameters associated with response to the addition of CAR-T cells to Stupp protocol, and evaluate the theoretical gain achievable through patient-specific protocol selection. 

With the exception of the benchmarking analysis, the exploratory treatment comparisons are performed using the same cohort of $N=10000$ VPs. Thus, the biological parameters and initial conditions of each VP remain unchanged across protocols, and only the treatment schedule is modified. By contrast, when reproducing a specific clinical study, a separate virtual cohort is generated with a sample size matching that of the corresponding clinical cohort.

\subsubsection{Clinical benchmarking}

Here, we assess whether our new model reproduces clinically plausible survival times in established therapeutic settings. We consider four complementary scenarios: the natural evolution of MGs under no treatment (NT), TMZ monotherapy, RT alone and its combination with TMZ according to the Stupp protocol, and CAR-T monotherapy. These comparisons, summarized in Table \ref{tab:clinical_benchmarking}
, correspond to different clinical trials and patients and should therefore be regarded as complementary benchmarking exercises rather than as a single external validation cohort. 

Overall, these complementary comparisons show that the model reproduces clinically plausible survival timescales across untreated disease, chemotherapy, radiotherapy, chemoradiotherapy, and CAR-T therapy. These results should not be interpreted as a complete validation of all model mechanisms or parameters. Rather, they establish a clinically plausible reference framework from which the exploratory comparison of alternative RT-TMZ-CAR-T treatment combinations can be performed.

\paragraph{Untreated disease.}

As an initial clinical plausibility check, we considered the natural evolution of untreated macroscopic disease. To represent a tumor that has not undergone surgical debulking, the initial tumor burden was sampled from $T_0\in[10^{10},10^{11}]$ cells, while the remaining model parameters were generated according to the sampling rules described in Section \ref{sec:methods_vp_insilico_trials} and Table \ref{tab:model_parameters}. For a cohort of $N=10000$ VPs, the model predicted a median survival time of $3.52$ months (95\% CI: $3.43$--$3.60$ months), consistent with the short survival times reported for untreated glioblastoma \cite{Feucht}.

For comparison, reducing the initial tumor burden to $T_0 \in [10^7, 10^9]$ cells to mimic postoperative residual disease extends the median survival time to $8.76$ months (95\% CI: $8.57$--$8.97$) in the absence of adjuvant therapy. The difference between these two scenarios effectively quantifies the survival benefit solely attributable to surgical resection.

\paragraph{TMZ monotherapy.}

We next benchmarked the TMZ component of the model against two clinical studies using TMZ monotherapy. The randomized NOA-08 trial \cite{Wick2012} included $N=195$ patients in the TMZ arm and used a dose-dense schedule consisting of TMZ administration for seven consecutive days followed by seven days without treatment. The reported median OS was $8.6$ months (95\% CI: $7.3$--$10.2$ months). Using the same cohort size and reproducing the corresponding TMZ schedule, the model predicted a median OS of $7.1$ months (95\% CI: $5.5$--$11.2$ months).

As an additional independent comparison, we considered the study by Glantz et al.\ \cite{Glantz2003}, in which $N=32$ patients received TMZ monotherapy using five consecutive treatment days within 28-day cycles. The clinically reported median OS was $6.0$ months, whereas the model predicted $7.0$ months (95\% CI: $5.4$--$11.1$ months). Taken together, these two comparisons indicate that the model reproduces the approximate survival timescale observed under TMZ monotherapy across two distinct treatment schedules and clinical cohorts.

\paragraph{RT and Stupp protocol.}

The randomized trial by Stupp et al.\ \cite{Stupp2005} provides the main clinical benchmark for RT and combined RT+TMZ treatment. In the original study, $286$ patients received RT alone and $287$ patients received RT with concomitant and adjuvant TMZ, with reported median OS times of $12.1$ and $14.6$ months, respectively. Using virtual cohorts with the same sample sizes, the model predicted median survival times of $12.21$ months (95\% CI: $11.05$--$13.74$ months) for RT alone and $14.60$ months (95\% CI: $12.52$--$16.78$ months) for the Stupp protocol. Thus, the model reproduces not only the survival timescale associated with both treatments but also the incremental survival benefit produced by adding TMZ to RT. The corresponding simulated KM curve in Fig.\ref{fig:stupp_trial} further reproduces the separation between the two treatment groups.

\begin{figure}[ht!]
    \centering
\includegraphics[width=0.4\textwidth]{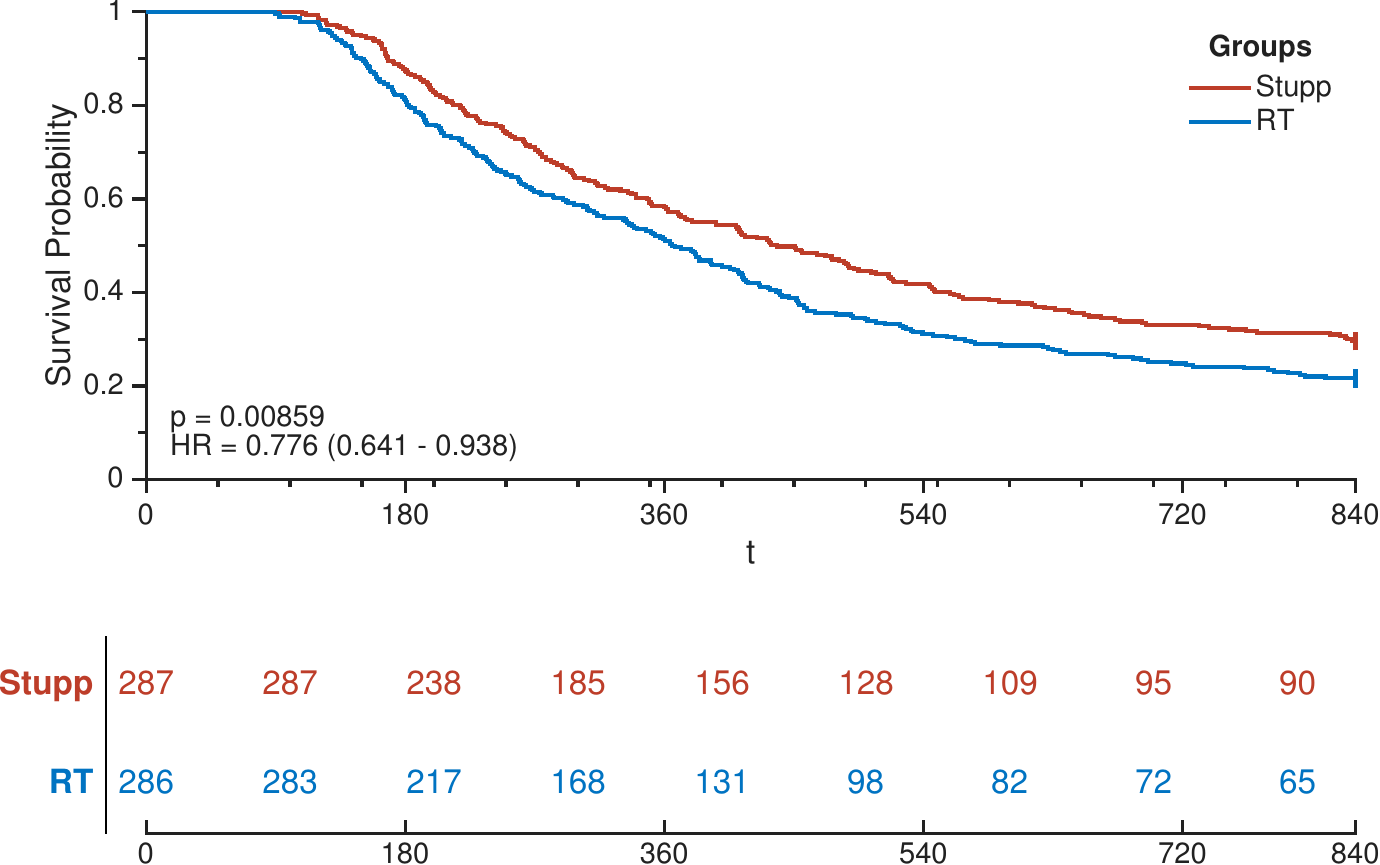}
\caption{Kaplan-Meier survival curves for the study-matched virtual
cohorts receiving radiotherapy alone ($N=286$) or the Stupp
protocol ($N=287$).}
    \label{fig:stupp_trial}
\end{figure}

\paragraph{CAR-T monotherapy.}

Finally, we compare the CAR-T component of the model with the phase~I study by Brown et al.\ \cite{Brown2024}. The clinical analysis included $41$ patients with recurrent glioblastoma evaluable for survival, for whom a median OS of $7.7$ months (95\% CI: $6.0$--$10.1$ months) was reported. Among this $41$ patients two different types of CAR-T cells were used. The group of $27$ patients was treated with CAR-T cells manufactured from cells that have been enriched for the central memory phenotype before genetic engineering, The group of $14$ patients were treated with CAR-T cells manufactured from a broader population of less-differentiated T cells. For out benchmark we use the former group of patients. According to the clinical trial performed in \cite{Brown2024}, this group is divided into three subgroups of $11$, $8$ and $8$ patients treated with dose schedules 1, 2 and 3, respectively (see \cite{Brown2024}). 

We generate there populations of $11$, $8$ and $8$ v.p. The range of initial tumour sizes is taken from \cite{Brown2024} Table 1, Arms 1-4 and converted to the number of cells using cell density $(1/3)\cdot10^{9}\mbox{cell}/\mbox{cm}^{3}$. The dose schedules were taken from Fig. 1b of \cite{Brown2024}.  We perfume virtual clinical trials following respective CAR-T dose schedules. This simulations yield a median OS of $5.05$ months (95\% CI: $2.37$--$7.94$ months), suggesting that the model reproduces the approximate survival timescale observed in this recurrent-disease setting.

However, this benchmark should be interpreted separately from the newly diagnosed/postoperative Stupp comparison, since the Brown et al.\ cohort represents recurrent disease and involves a different clinical setting, tumor burden, and CAR-T administration schedule.

\begin{table*}[ht!]
\centering
\caption{Clinical benchmarking of the model. Clinical and model-predicted median overall survival (OS) are reported in months. Values in brackets correspond to 95\% confidence intervals (CIs) when available. 
}
\label{tab:clinical_benchmarking}
\begin{tabular}{lcccc}
\hline
Treatment (clinical setting) & $N$ & Clinical median OS [CI] (months)& Model median OS [CI] (months) & Ref. \\
\hline
TMZ (NOA-08)
& 195
& $8.6\,[7.3,10.2]$
& $7.1\,[5.5,11.2]$
& \cite{Wick2012} \\

TMZ (Glantz et al.)
& 32
& $6.0$
& $7.0\,[5.4,11.1]$
& \cite{Glantz2003} \\

RT (Stupp et al.)
& 286
& $12.1$ \,[ 11.2,13.0]
& $12.21\,[11.05,13.74]$
& \cite{Stupp2005} \\

RT+TMZ (Stupp et al.)
& 287
& $14.6$ \,[13.2, 16.8]
& $14.60\,[12.52,16.78]$
& \cite{Stupp2005} \\

CAR-T (Brown et al.)
& 27
& $6.1\,[4.8,9.5]$
& $5.05\,[2.37,7.94]$
& \cite{Brown2024} \\
\hline
\end{tabular}
\end{table*}

Table~\ref{tab:clinical_benchmarking} summarizes the clinical and model-predicted survival outcomes across the different treatment settings. Overall, the comparisons indicate that the model reproduces clinically plausible survival timescales under TMZ, RT, RT+TMZ, and CAR-T therapy. These results should not be interpreted as a complete external validation of all model mechanisms or parameters; rather, they establish a clinically plausible reference framework from which the exploratory comparison of alternative RT-TMZ--CAR-T treatment combinations can be performed.

\subsubsection{Exploratory virtual cohort: reference treatments}

Here, we characterize the response of a shared exploratory cohort ($N=10000$ VPs) under the baseline regimens. Unlike the study-matched groups used for clinical benchmarking in the previous subsection, this identical virtual cohort is maintained across all scenarios, enabling a direct head-to-head comparison. 
Our analysis covers the postoperative untreated scenario, RT monotherapy, TMZ monotherapy, the Stupp protocol, and CAR-T monotherapy. 
These simulations establish the baseline efficacy and hallmark tumor dynamics of each approach, providing a reference framework before we explore the integration of CAR-T cells into the standard Stupp schedule~\cite{Stupp2005}.

\paragraph{Postoperative untreated reference.}

We first consider the evolution of the exploratory virtual cohort in the absence of any adjuvant treatment. In this setting, the initial tumor burden is sampled from the postoperative range $T_0\in[10^7,10^9]$ cells, as defined in Table \ref{tab:model_parameters}, and no RT, TMZ, or CAR-T therapy is subsequently administered. For the cohort of $N=10000$ VPs, the median survival time is $8.76$ months (95\% CI: $8.57$--$8.97$ months). This scenario is used throughout the following comparisons as an untreated postoperative reference.

\paragraph{Radiotherapy.}

RT monotherapy substantially improves survival relative to the postoperative untreated reference. 
The median OS increases from $8.76$ months without adjuvant treatment to $12.54$ months (95\% CI: $12.29$--$12.81$ months) under RT, corresponding to an increase of approximately $3.8$ months. Consistently, the KM curves show a clear separation between both groups, with a HR of $0.742$ (95\% CI: $0.72$--$0.764$) for the untreated postoperative reference relative to RT.

The dynamics of the MVP provide a mechanistic illustration of this survival benefit. During the fractionated RT course, the proliferative tumor compartments are repeatedly reduced, while lethally damaged cells transiently accumulate in the compartment $D$. Consistent with clinical observations of delayed tumor regression, the total tumor volume continues to shrink post-RT due to the prolonged clearance and 
mitotic delay characterizing the lethally damaged population $D$. Nevertheless, viable tumor cells remain after treatment completion and subsequently repopulate the tumor. Thus, within the present model, RT produces a substantial delay in tumor progression but does not achieve durable tumor control as a standalone treatment.

\begin{figure}[ht!]
    \centering
    \includegraphics[width=0.4\textwidth]{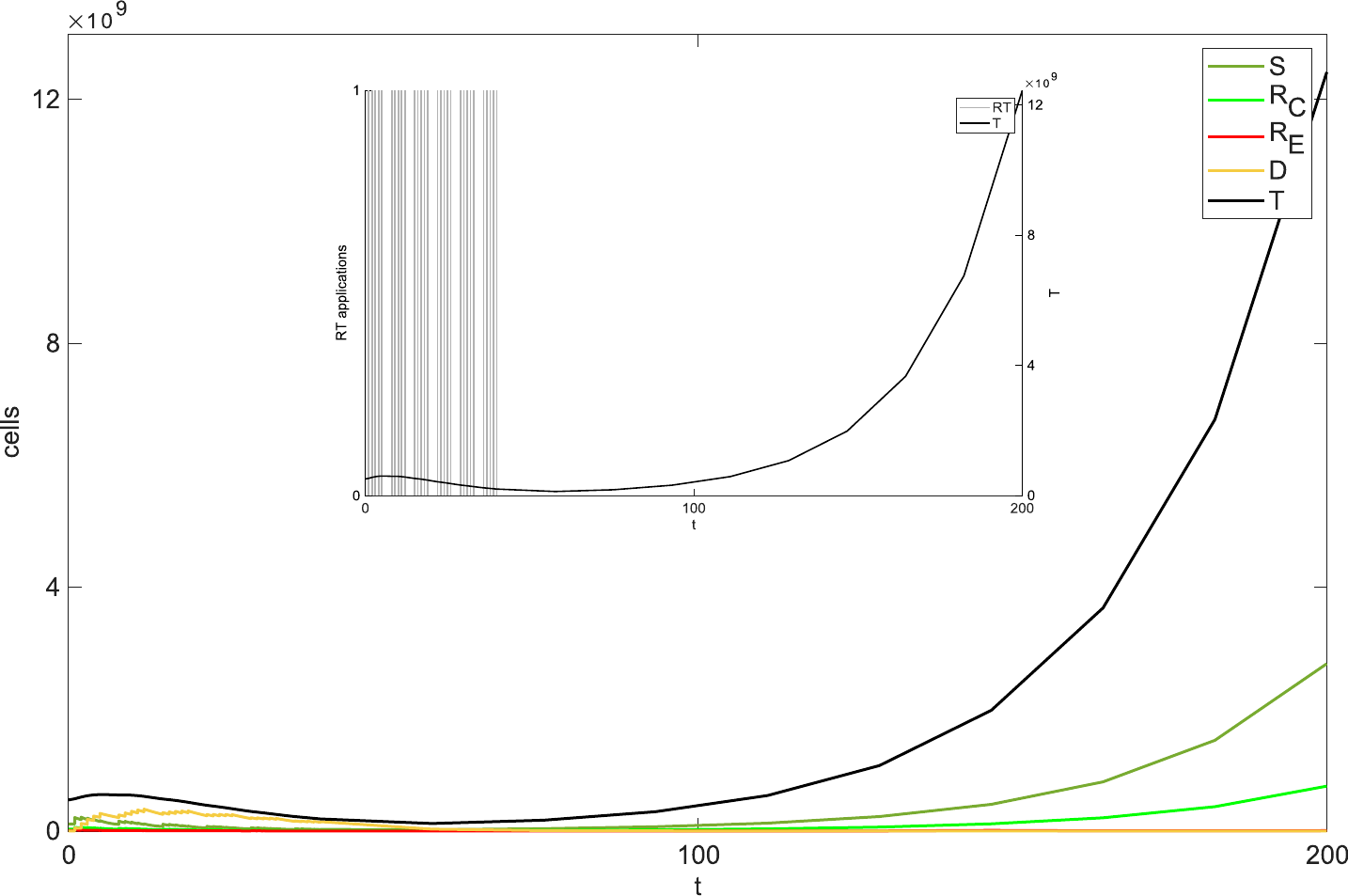} 
    \includegraphics[width=0.4\textwidth]{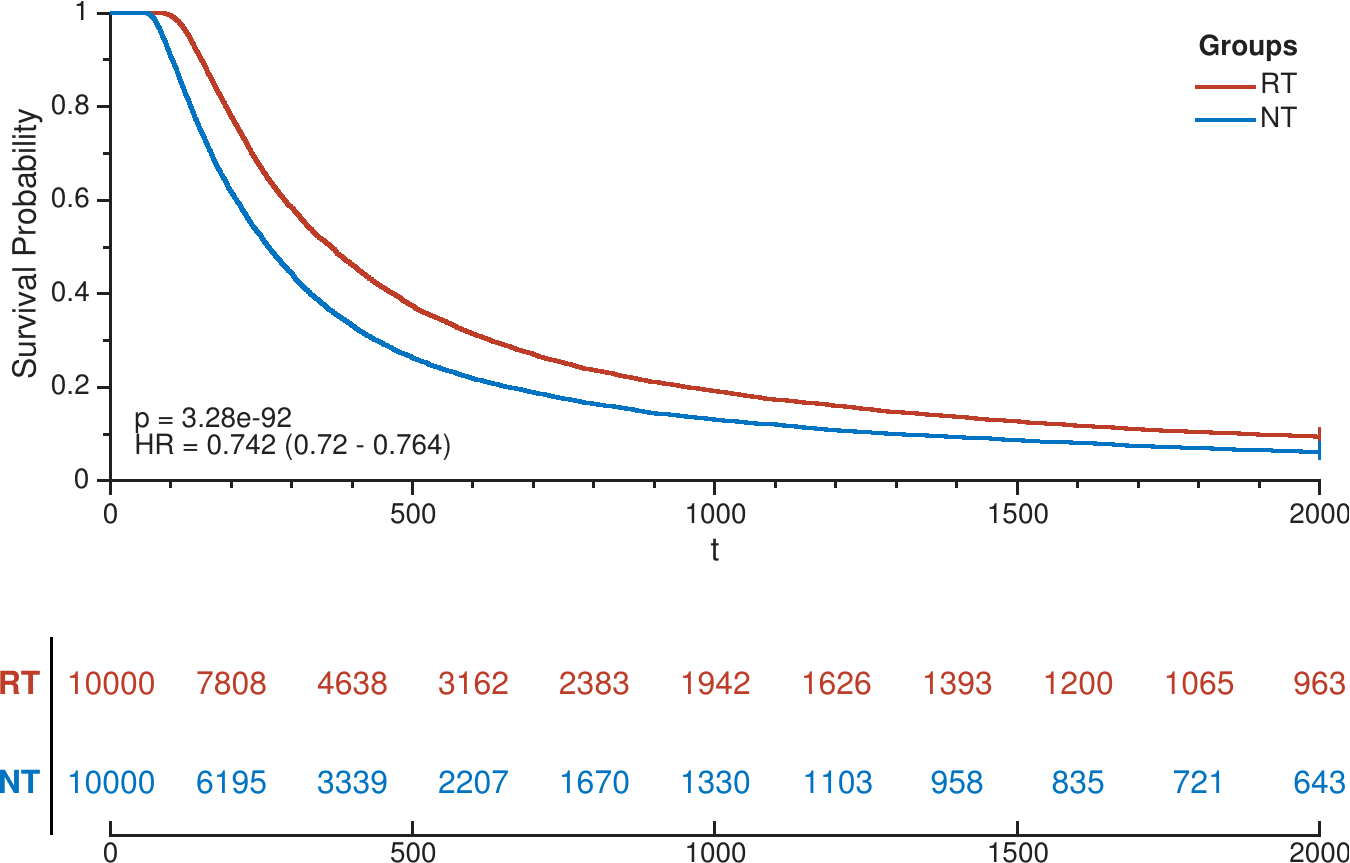}
\caption{RT monotherapy in the exploratory virtual cohort. Left:
dynamics of the MVP, showing the evolution of the
tumor-cell compartments and treatment-induced damage. Right:
KM survival curves for the postoperative untreated reference
and RT-treated cohorts ($N=10000$).}
    \label{fig:sim_dynamics_RT_only}
\end{figure}

\paragraph{Temozolomide.}

TMZ monotherapy produces a more modest survival benefit than RT. The median OS increases from $8.76$ months for the postoperative untreated reference to $10.52$ months (95\% CI: $10.28$-$10.76$ months) under TMZ, corresponding to an increase of approximately $1.8$ months. Consistently, the KM curves show a smaller separation than that observed for RT, with a HR of $0.875$ (95\% CI: $0.85$-$0.9$) for the untreated postoperative reference relative to TMZ.

The dynamics of the MVP are consistent with this more limited population-level effect. Repeated TMZ administration reduces the TMZ-sensitive proliferative compartments and generates a transient accumulation of lethally damaged cells. However, viable tumor cells remain after treatment, with TMZ-resistant populations persisting and being selected during TMZ therapy, ultimately allowing tumor regrowth. Consequently, within this heterogeneous cohort, TMZ monotherapy merely delays progression, offering a weaker survival benefit than RT.

\begin{figure}[h!]
    \centering
    \includegraphics[width=0.4\textwidth]{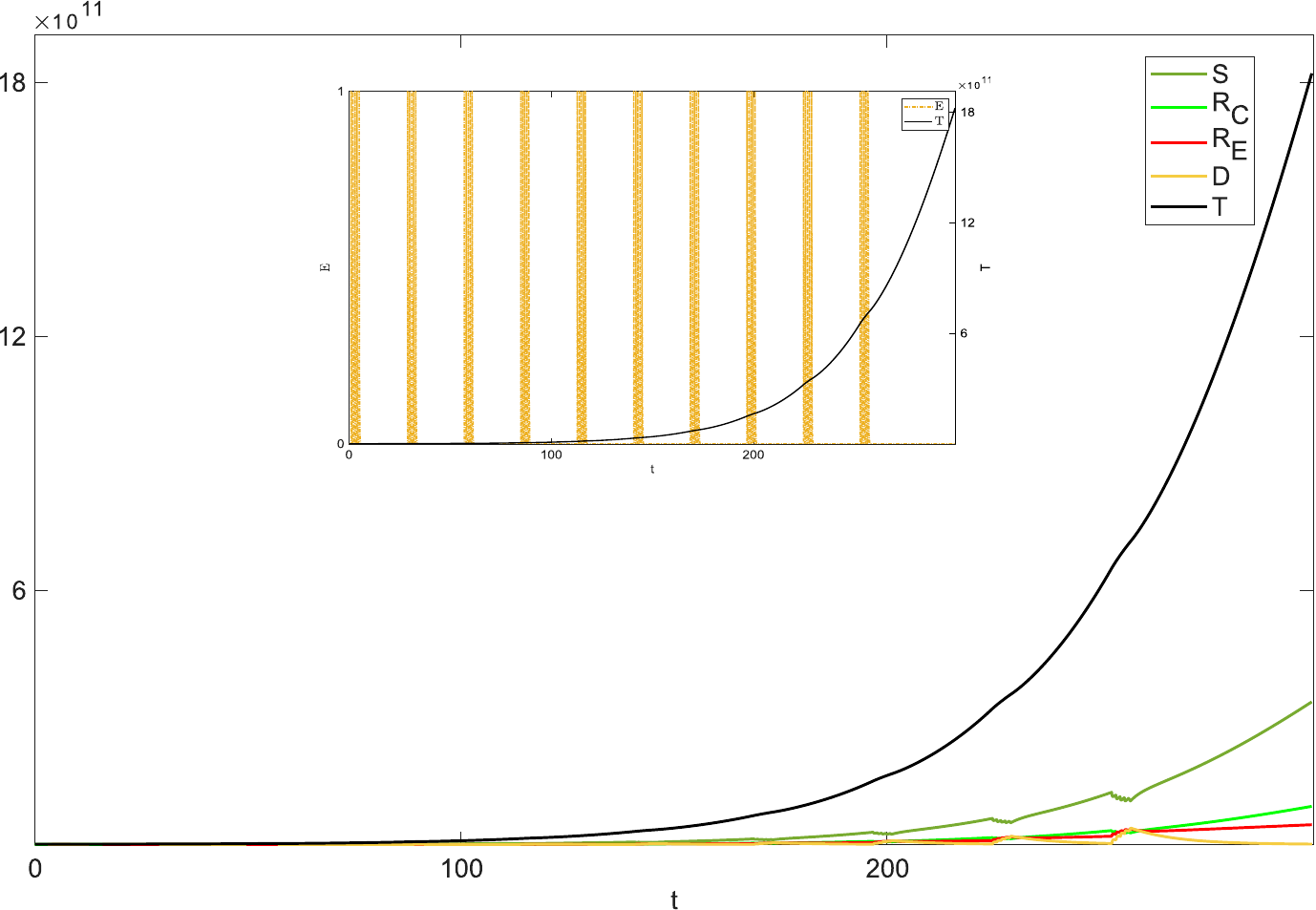} 
    \includegraphics[width=0.4\textwidth]{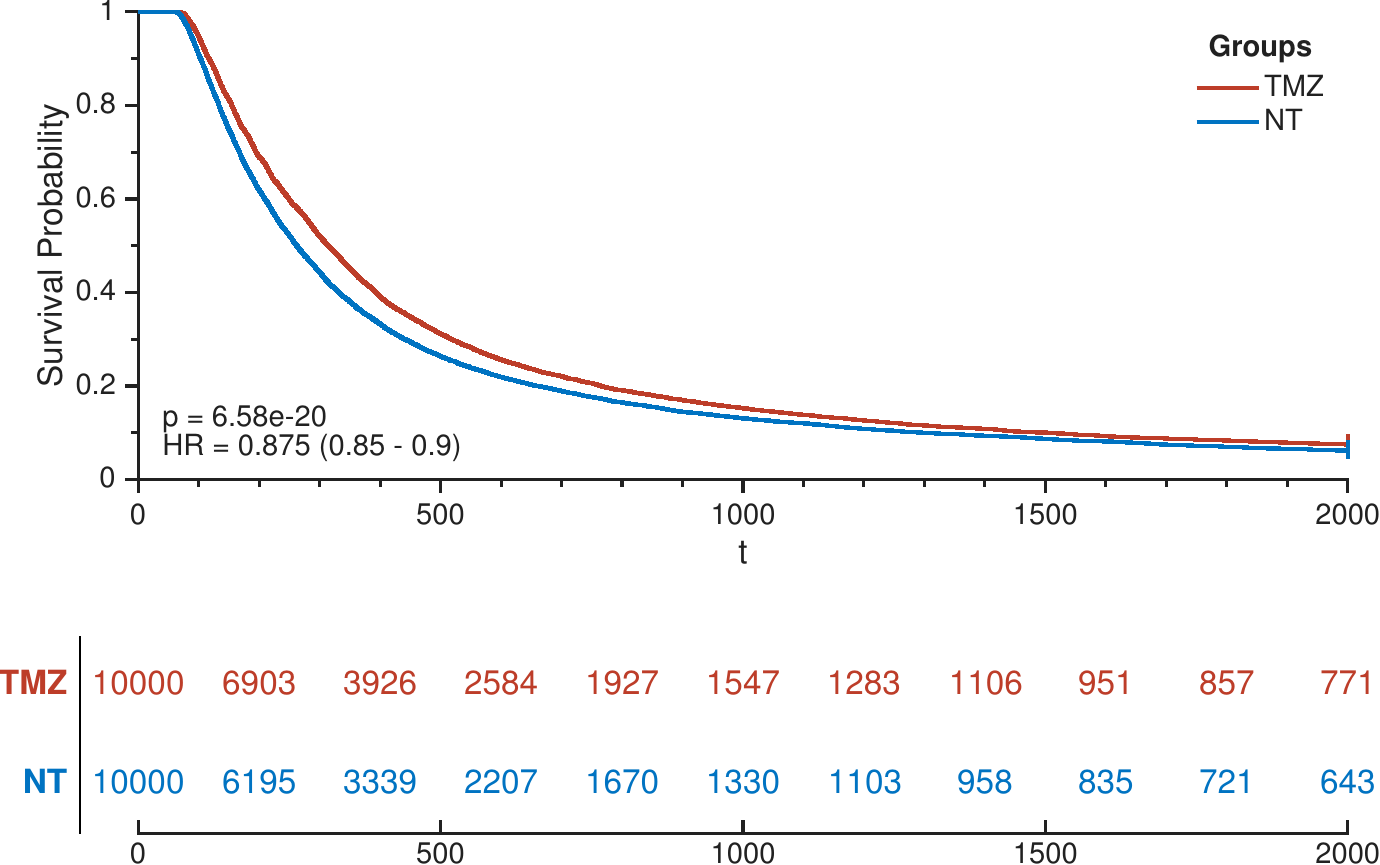}
\caption{TMZ monotherapy in the exploratory virtual cohort. Left:
dynamics of the median virtual patient, showing the evolution of the
tumor-cell compartments and treatment-induced damage. Right:
KM survival curves for the postoperative untreated reference
and TMZ-treated cohorts ($N=10000$).}
    \label{fig:TMZ_only_dynamics}
\end{figure}

\paragraph{Stupp protocol.}

Combining RT and TMZ via the Stupp protocol~\cite{Stupp2005} yields the highest survival advantage among the conventional regimens tested in this cohort. The median OS reaches $13.88$ months ($95\%$ CI: $13.57$--$14.21$), compared to 
$12.54$ months under RT monotherapy, $10.52$ months under TMZ monotherapy, and $8.76$ months for the untreated control. The KM curves show a clear separation, with the Stupp protocol achieving a HR of $0.667$ ($95\%$ CI: $0.648$--$0.687$) 
relative to the postoperative untreated reference.

The MVP dynamics highlight this complementary action. Simultaneous radiation fractions and TMZ cycles deeply deplete the proliferative compartments while triggering a transient buildup of lethally damaged cells, keeping the viable tumor burden suppressed significantly longer than either monotherapy. Yet, residual viable cells evade eradication and eventually drive repopulation. Thus, while the Stupp protocol substantially extends disease control, it falls short of permanent tumor eradication within the current model framework.

\begin{figure}[ht!]
    \centering
    \includegraphics[width=0.4\textwidth]{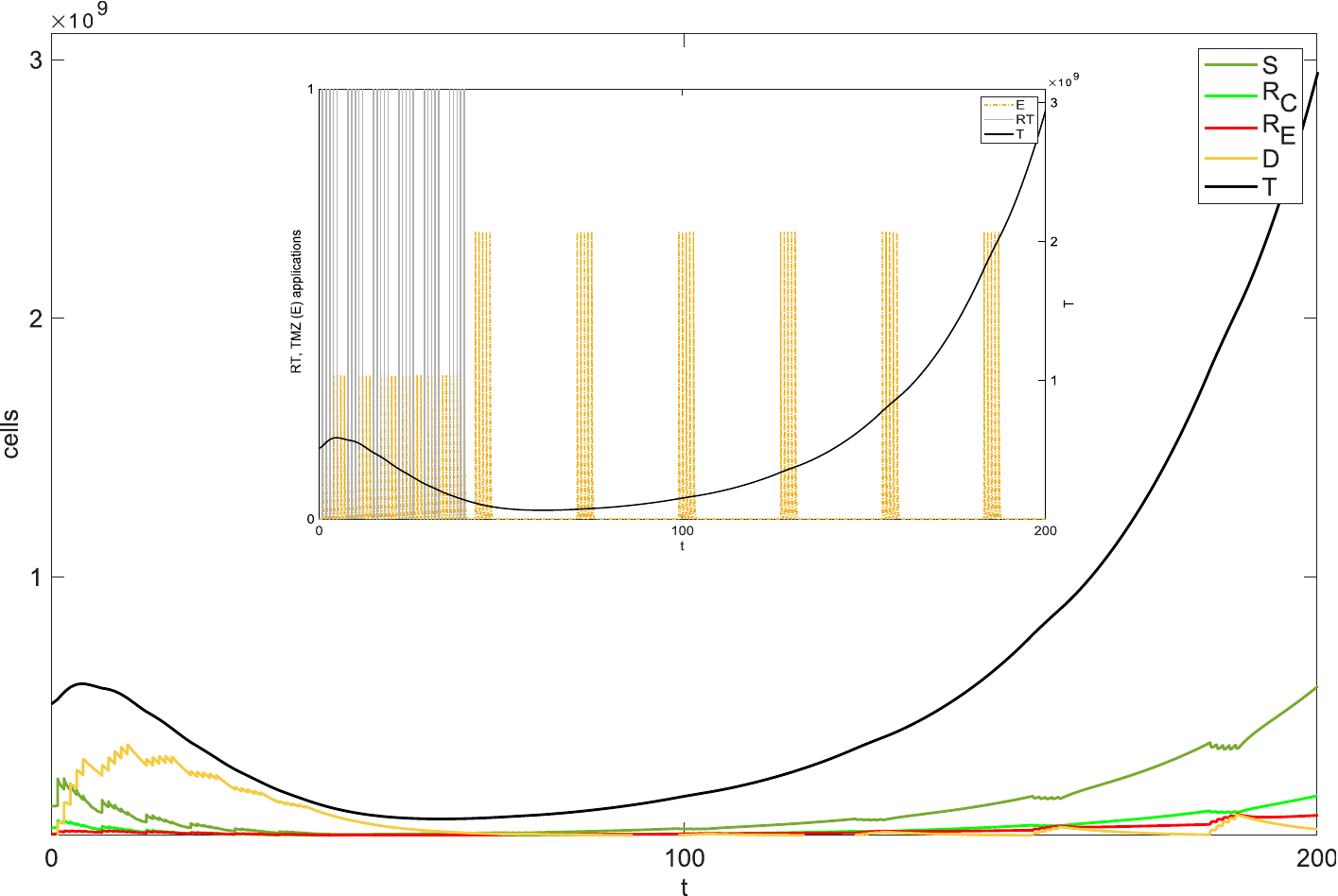} 
    \includegraphics[width=0.4\textwidth]{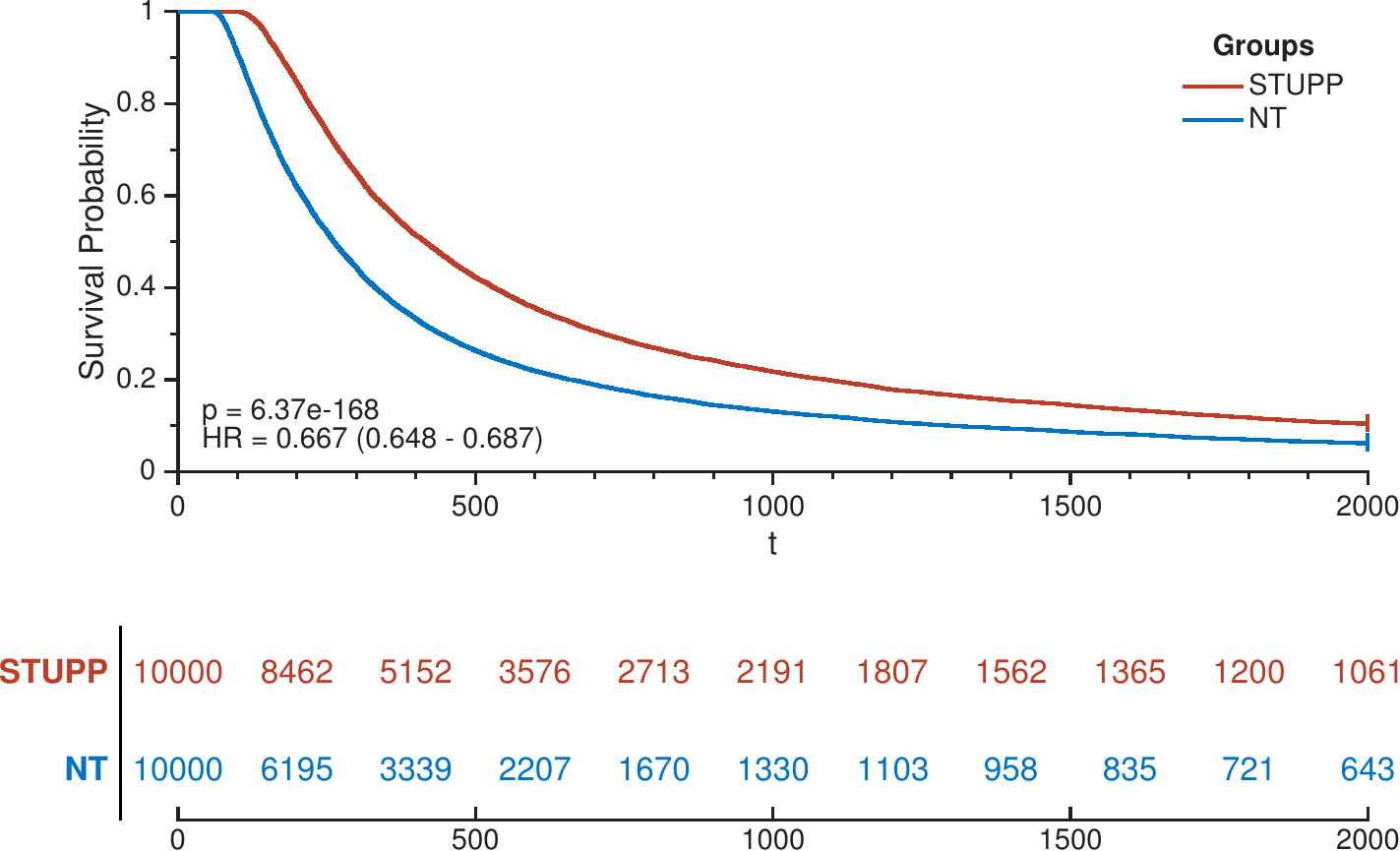}
    \caption{Stupp protocol in the exploratory virtual cohort. Left:
dynamics of the median virtual patient, showing the evolution of the
tumor-cell compartments and treatment-induced damage under combined
RT and TMZ treatment. Right: KM survival curves for the postoperative untreated reference and Stupp-treated cohorts
($N=10000$).}
    \label{fig:Stupp_dynamics}
\end{figure}

\paragraph{CAR-T cells.}

Finally, we evaluate CAR-T monotherapy, consisting of three infusions administered at 7-day intervals with $3\times10^8$, $3\times10^8$, and $4\times10^8$ cells, 
respectively, for a cumulative dose of $10^9$ cells. Under this regimen, median OS increases to $9.76$ months ($95\%$ CI: $9.50$--$9.98$), up from $8.76$ months for the untreated reference. Thus, CAR-T therapy extends median OS by approximately one month---an improvement substantially smaller than those yielded by TMZ, RT, or their combination into the Stupp protocol.

The KM curves reflect this modest population-level effect, with the CAR-T protocol achieving a HR of $0.91$ (95\% CI: $0.89-0.94$) relative to the postoperative untreated reference. The MVP dynamics provide a mechanistic basis for this limited long-term efficacy. Following infusion, the CAR-T cell population expands transiently to drive tumor-cell lysis, but subsequently contracts. Meanwhile, the resistant subpopulation $R_C$, which evades CAR-T targeting, persists and progressively fuels tumor regrowth.

Importantly, the constrained efficacy of CAR-T monotherapy does not preclude a synergistic benefit when combined with conventional treatments. In our model, CAR-T cells and the Stupp components act on partially complementary compartments: immunotherapy targets TMZ-resistant but CAR-T-sensitive cells, while TMZ and RT exert therapeutic pressure on populations escaping T-cell-mediated killing. This cooperative dynamic provides the mechanistic rationale for exploring whether integrating CAR-T therapy into the standard Stupp schedule can enhance tumor control.
\begin{figure}[ht!]
    \centering
    \includegraphics[width=0.4\textwidth]{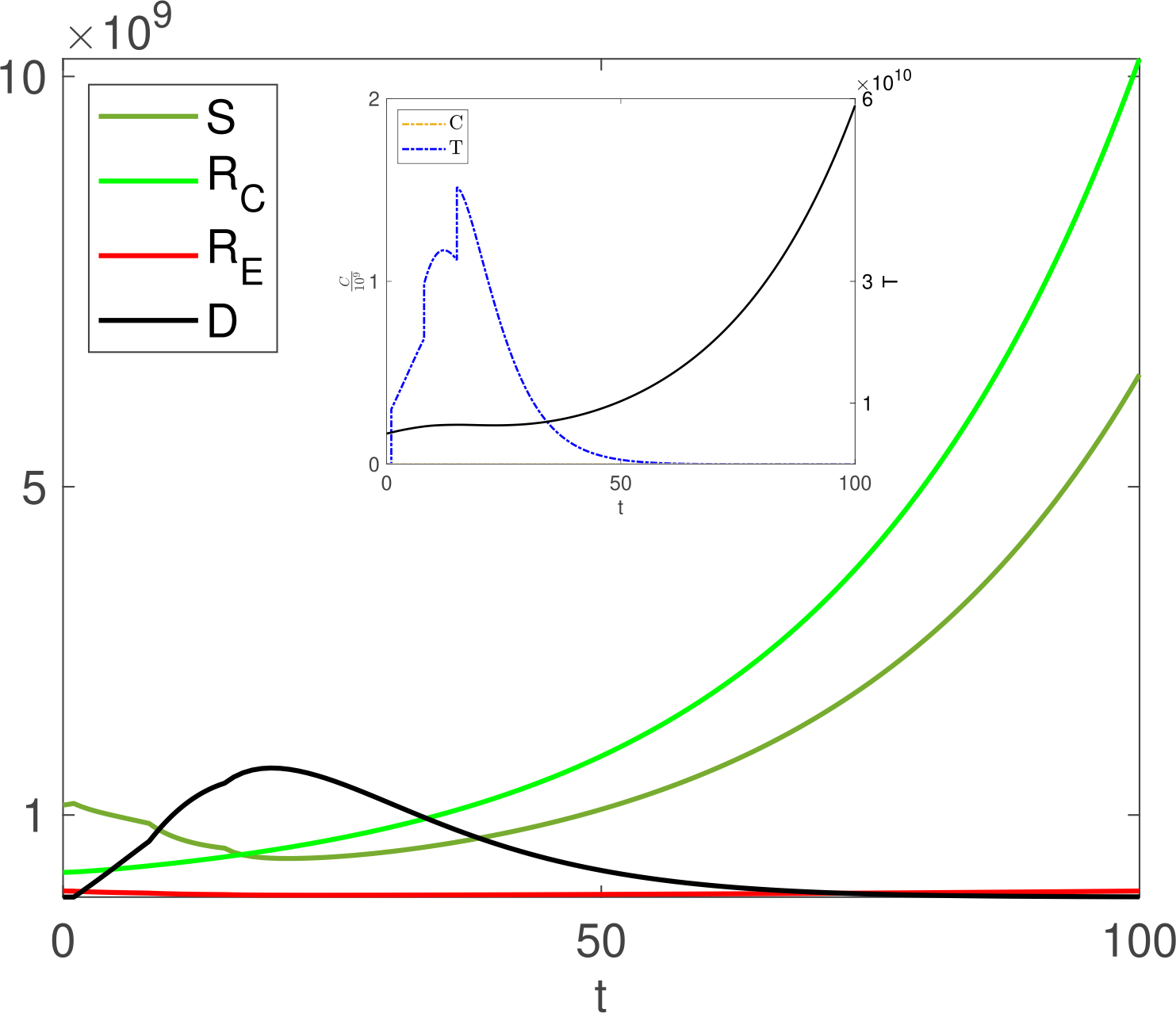} 
    \includegraphics[width=0.4\textwidth]{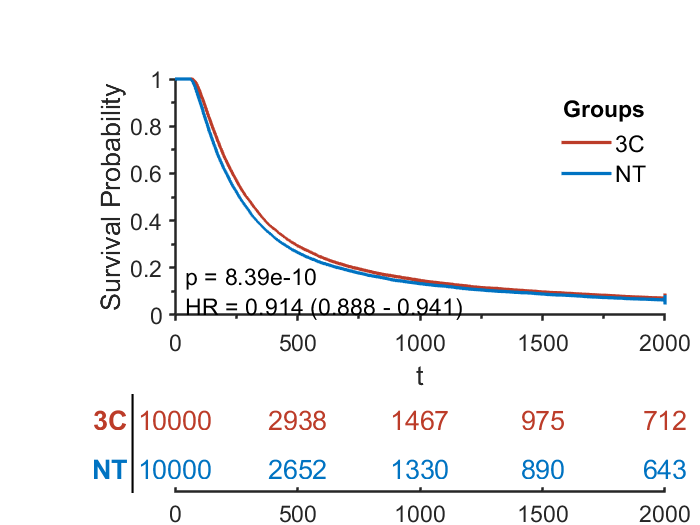}
\caption{CAR-T monotherapy in the exploratory virtual cohort. Left:
dynamics of the median virtual patient, showing the evolution of the
tumor-cell compartments and the CAR-T cell population following three
CAR-T infusions. Right: KM survival curves for the
postoperative untreated reference and CAR-T-treated cohorts
($N=10000$).  }
    \label{fig:CAR_T_only_tumor}
\end{figure}

The population-level outcomes of the reference treatment strategies are summarized in Table~\ref{tab:reference_treatments}. Among the treatments considered individually, RT produces the largest survival improvement, while CAR-T monotherapy provides a comparatively modest benefit. The
Stupp protocol yields the longest median survival among the reference strategies, in line with clinical observations, and therefore constitutes the natural comparator for assessing the additional benefit of incorporating CAR-T therapy.

\begin{table}[ht]
\centering
\caption{Median overall survival (OS), hazard ratio (HR) and their respective 95\% CI for the reference treatment strategies in the common exploratory cohort of $N=10000$ virtual patients. 
}
\label{tab:reference_treatments}
\begin{tabular}{lcc}
\hline
Protocol & Median OS (months) & HR \\
\hline
Surgery & 8.76 $[8.57,8.97]$ & $-$\\
TMZ & 10.52 $[10.28,10.76]$ & 0.875 $[0.85, 0.9]$\\
RT & 12.54 $[12.29,12.81]$ & 0.742 $[0.72,0.764]$\\
Stupp & 13.88 $[13.57,14.21]$ & 0.667 $[0.648,0.687]$\\
3C & 9.76 $[9.50,9.98]$ & 0.91 $[0.89,0.94] $\\
\hline
\end{tabular}
\end{table}

\subsubsection{Exploratory virtual cohort: combined treatment}

Having characterized the baseline regimens, we now investigate whether integrating CAR-T cell therapy into the Stupp protocol yields an additional population-level survival benefit, and to what extent this depends on scheduling. 
To ensure a rigorous comparison, all protocols are evaluated using the same cohort 
of $N=10000$ VPs. Retaining identical biological parameters and initial conditions 
for each VP creates a matched design, allowing us to isolate the effects of 
scheduling from population-level variability.

Hereafter, if 3C appears only once, it denotes the CAR-T schedule described above, consisting of three CAR-T cell infusions administered at 7-day intervals, with doses of $3\times10^8$, $3\times10^8$, and $4\times10^8$ cells, respectively, for a
total administered dose of $10^9$ CAR-T cells. Instead, if 3C appears two times in the protoocls, it denotes the CAR-T schedule described above but halves doses, i.e., considering each time three CAR-T cell infusions administered at 7-day intervals, with doses of $1.5\times10^8$, $1.5\times10^8$, and $2\times10^8$ cells, respectively, for a
total administered dose of $10^9$ CAR-T cells. For compactness, Stupp(C) and
Stupp(A) denote the concomitant and adjuvant phases of the Stupp protocol \cite{Stupp2005},
respectively.

We consider five treatment sequences designed to place CAR-T therapy in
different stages relative to standard chemoradiotherapy: 3C-Stupp,
Stupp-3C, 3C-Stupp-3C, Stupp(C)-3C-Stupp(A), and
Stupp(C)-3C-Stupp(A)-3C. These protocols compare CAR-T
administration before or after the Stupp protocol, on both sides of
it, or between its concomitant and adjuvant phases.
For each strategy, we compare the median OS, its corresponding
$95\%$ CI, and the KM survival distribution with
those obtained under the Stupp protocol alone, since the goal is to improve the standard Stupp protocol.

\paragraph{3C-Stupp protocol.}

We first consider the 3C-Stupp sequence, in which the three CAR-T cell
infusions are administered before the Stupp protocol. The median OS increases from $13.88$ months (95\% CI: $13.57-14.21$ months) under Stupp protocol to
$14.67$ months (95\% CI: $14.40-15.04$ months) under 3C-Stupp, corresponding to an increase of approximately $0.79$ months.

The KM curves (Fig.\ref{fig:3C_Stupp_dynamics} bottom panel) show a statistically significant, although modest,
survival improvement with the combined protocol ($p=0.00136$), with a
HR of $0.953$ (95\% CI: $0.925-0.981$), for Stupp relative
to 3C-Stupp, 
indicating that, at the population
level, the additional benefit produced by administering CAR-T cells before
Stupp is limited despite reaching statistical significance.

The dynamics of the MVP illustrate the sequential action
of the two therapeutic components (Fig.\ref{fig:3C_Stupp_dynamics} left panel). The initial CAR-T treatment transiently
modifies the tumor composition and reduces CAR-T-sensitive populations, even without expanding, after which the RT and TMZ components of the Stupp protocol exert additional
pressure on the remaining tumor cells. Nevertheless, residual
viable populations ultimately persist under the adjuvant phase of the Stupp protoocl and drive tumor regrowth. 

\begin{figure}[ht!]
    \centering
    \includegraphics[width=0.4\textwidth]{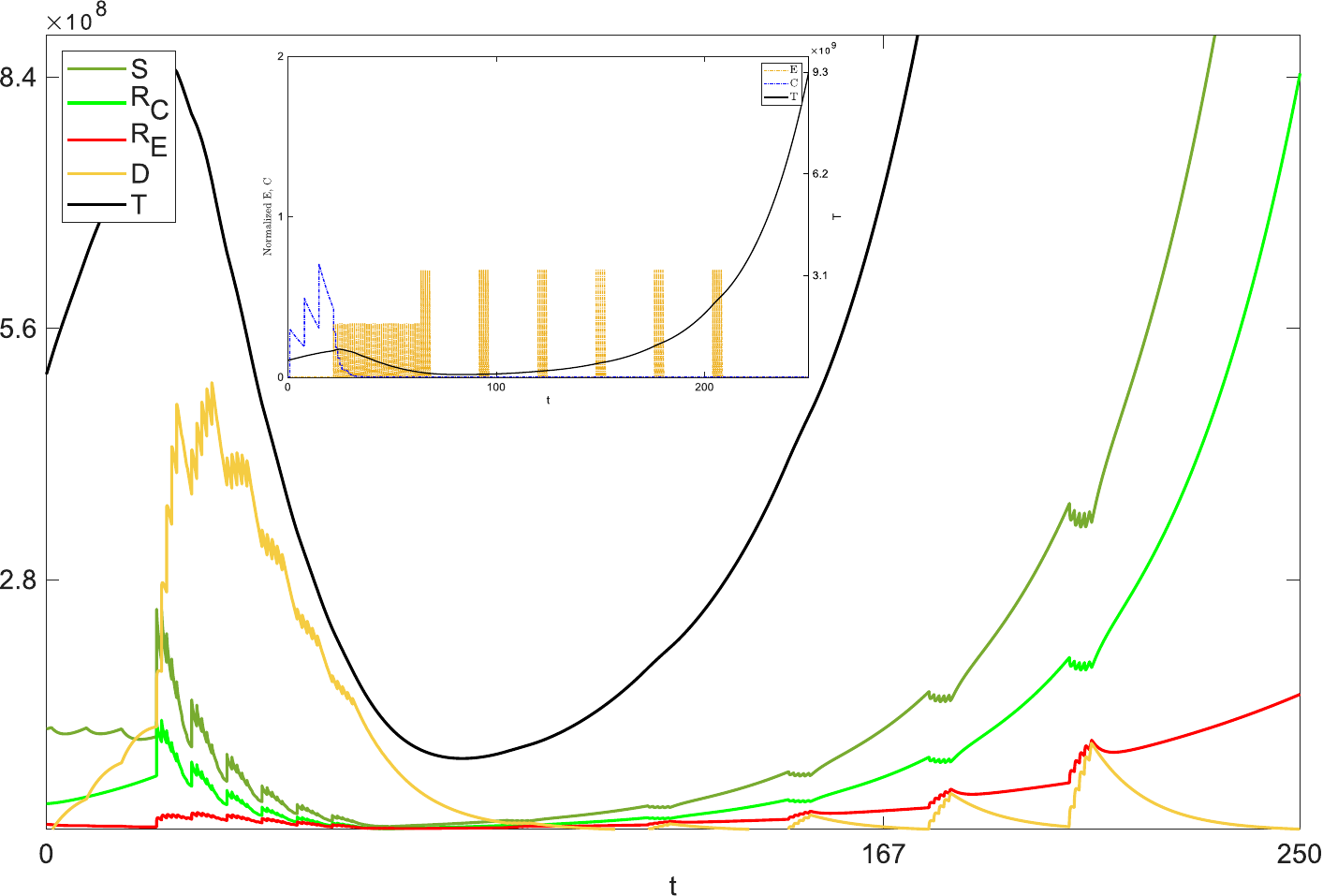} 
    \includegraphics[width=0.4\textwidth]{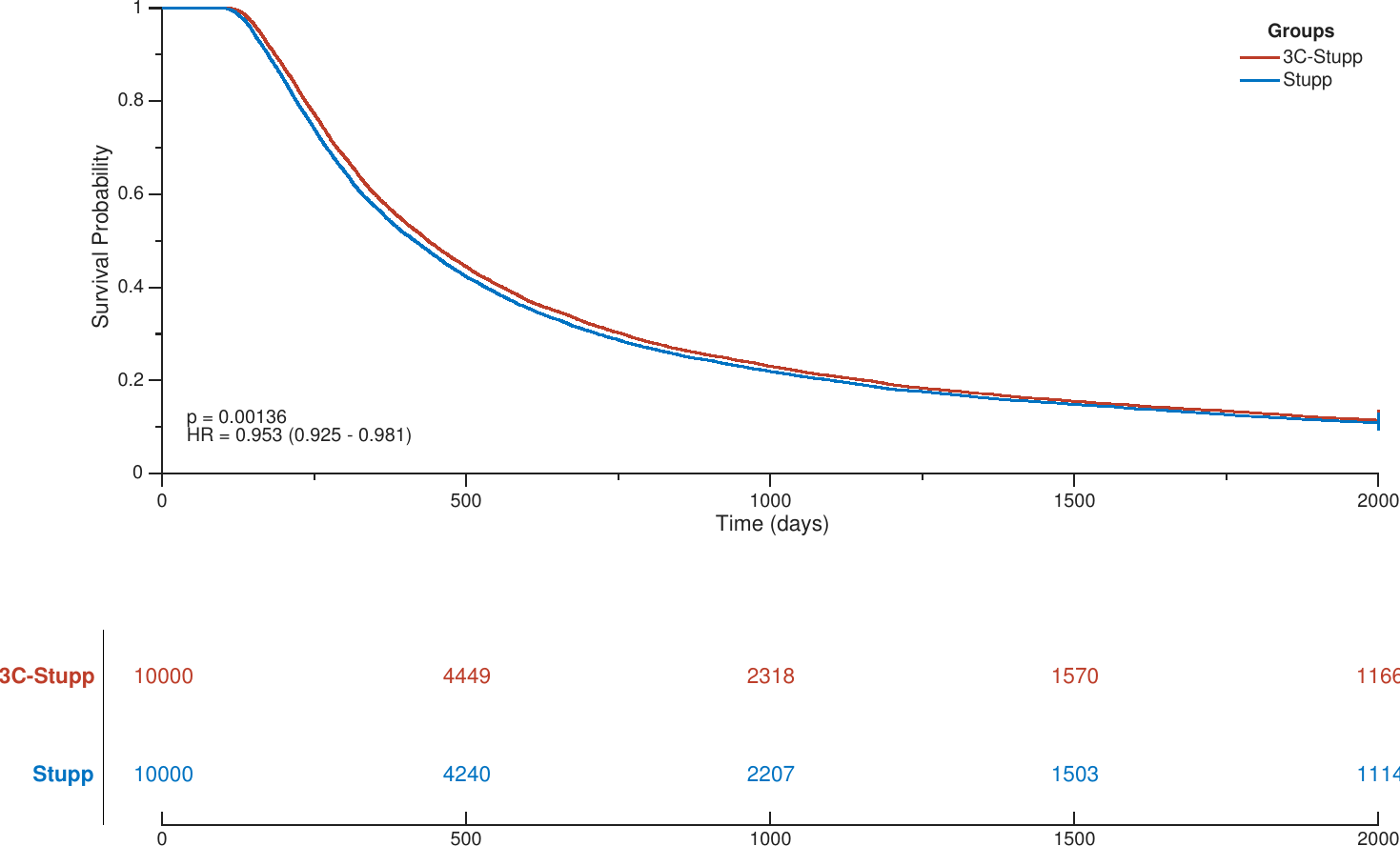}
\caption{3C-Stupp treatment in the exploratory virtual cohort. Left:
dynamics of the median virtual patient under three CAR-T cell infusions
followed by the Stupp protocol. Right: KM survival
curves for the 3C--Stupp and Stupp-treated cohorts ($N=10000$).}
    \label{fig:3C_Stupp_dynamics}
\end{figure}

\paragraph{Stupp-3C protocol.}

Next, we consider the Stupp-3C sequence, in which the Stupp
protocol is followed by three CAR-T cell infusions. The median OS increases from $13.88$ months under Stupp protocol to
$14.74$ months (95\% CI: $14.40-15.06$ months) under Stupp-3C, corresponding to an increase of approximately $0.9$ months.

The KM curves show a statistically significant but modest survival improvement relative to Stupp alone ($p=0.00611$), with a HR of $0.96$ (95\% CI: $0.932-0.988$) for Stupp relative to
Stupp-3C (Fig.\ref{fig:Stupp_3C_dynamics} right panel). As for the 3C-Stupp sequence, the relatively small reduction in HR indicates that the population-level effect of adding CAR-T therapy
remains limited despite the statistically significant difference.

The MVP dynamics illustrate the complementary action of the two treatments in the reverse temporal sequence (Fig.~\ref{fig:Stupp_3C_dynamics}, top panel). 
RT and TMZ initially deplete the tumor populations during the concomitant Stupp phase. However, during the subsequent adjuvant phase, the tumor regrows to nearly its initial volume, albeit with altered heterogeneity due to the selection of TMZ-resistant clones. Following this, CAR-T therapy targets the remaining sensitive cells, including the TMZ-resistant fractions that remain susceptible to immune-mediated killing. 
Nevertheless, CAR-T-resistant sub-population and other viable residues persist, ultimately driving a second wave of tumor regrowth. Notably, CAR-T cells undergo a more pronounced expansion when administered after the Stupp protocol rather than before and are not subject to RT- and TMZ-induced mortality (compare the dotted blue $C$ lines in the top panels of Figs.~\ref{fig:3C_Stupp_dynamics} and~\ref{fig:Stupp_3C_dynamics}). This enhanced expansion could provide a mechanistic explanation for the slightly higher survival benefit observed at the population level.

\begin{figure}[ht!]
    \centering
    \includegraphics[width=0.4\textwidth]{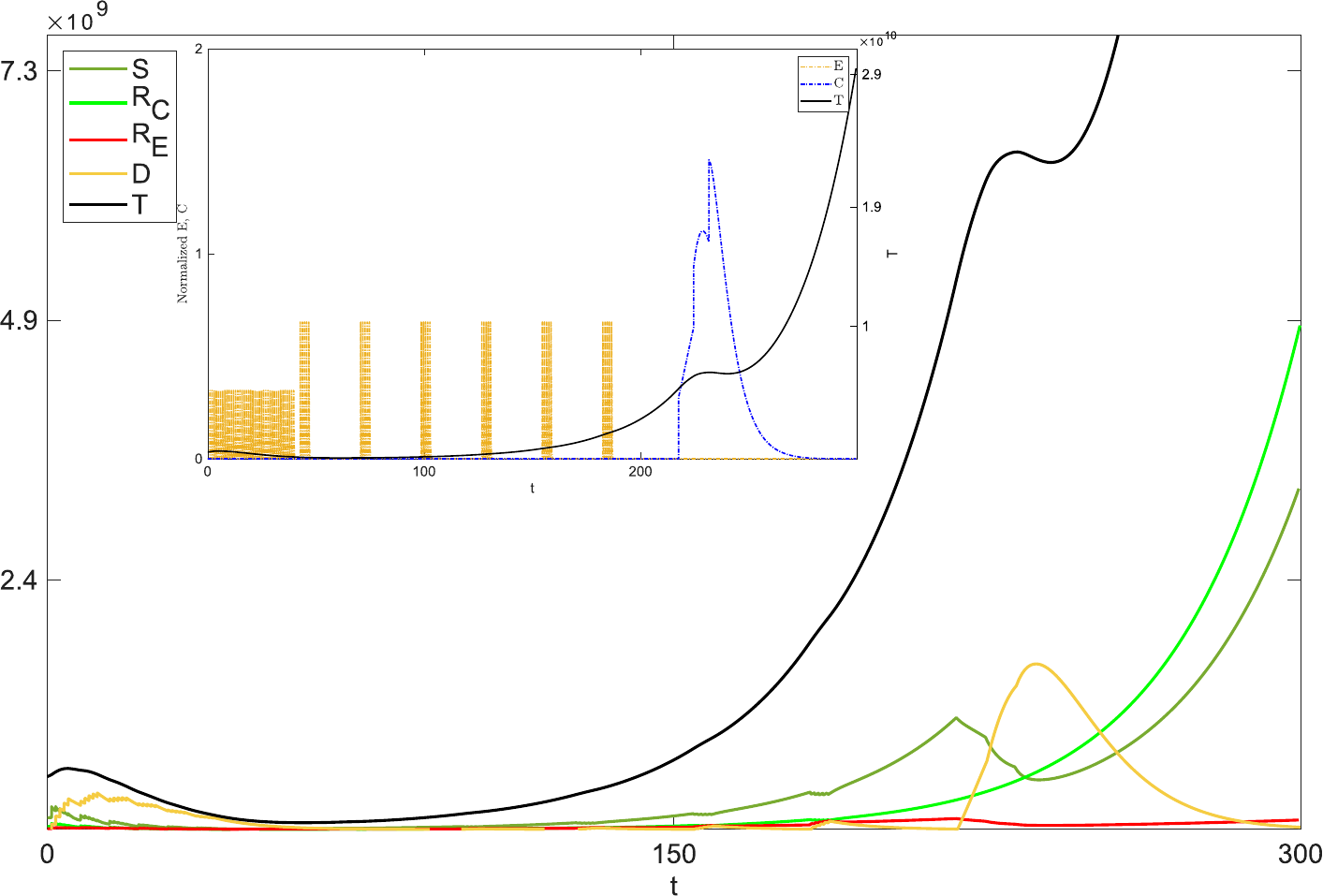} 
    \includegraphics[width=0.4\textwidth]{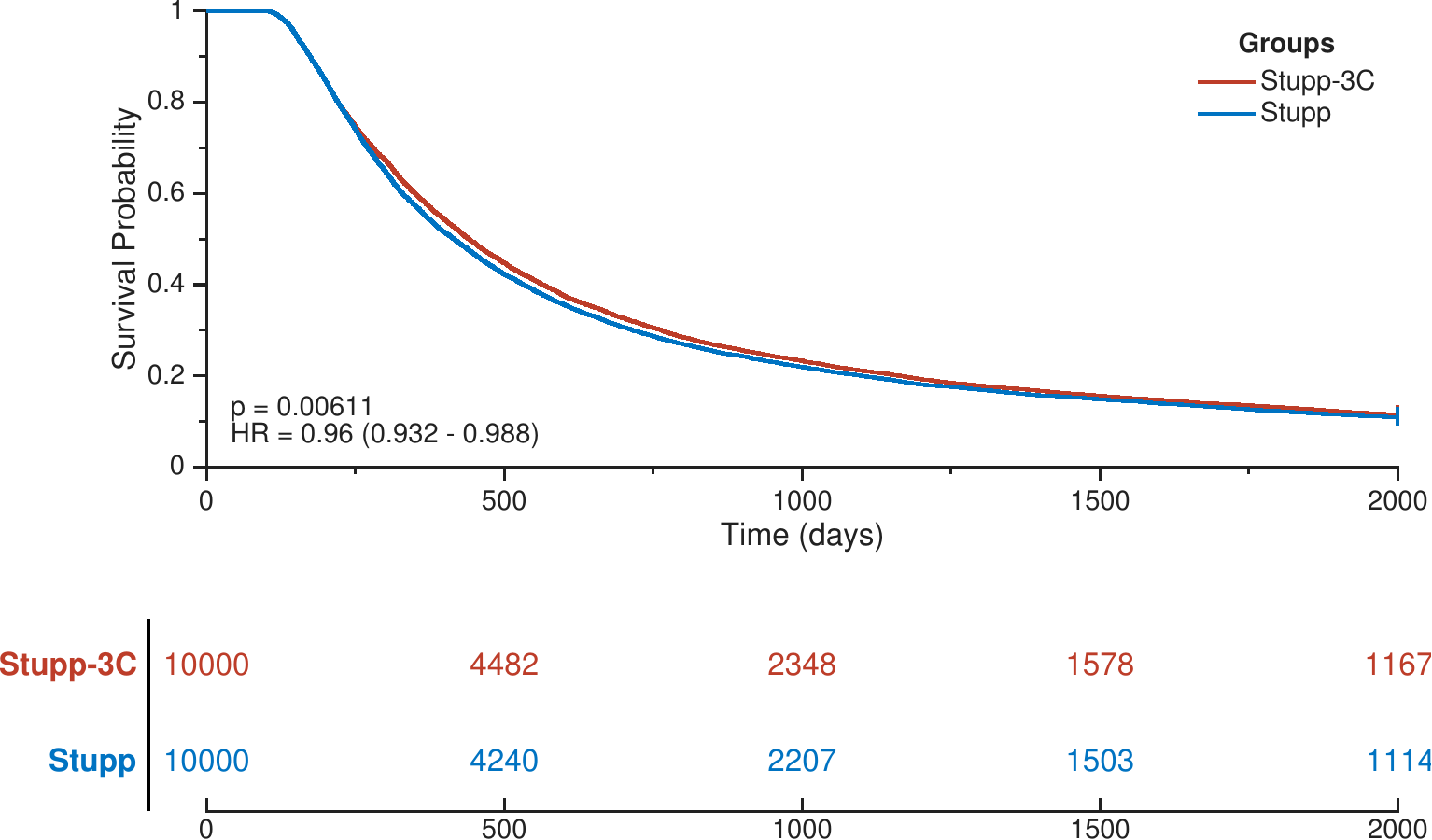}
\caption{Stupp-3C treatment in the exploratory virtual cohort. Left: dynamics of the median virtual patient under the Stupp protocol followed by three CAR-T cell infusions. Right: KM survival curves for the Stupp-3C and Stupp-treated cohorts ($N=10000$).}
    \label{fig:Stupp_3C_dynamics}
\end{figure}

\paragraph{3C-Stupp-3C protocol.}

Here, we evaluate the 3C-Stupp-3C sequence, which schedules three CAR T-cell infusions prior to the Stupp protocol and three subsequent infusions upon its completion. This strategy subjects the tumor to CAR T-cell-mediated pressure both before and after standard chemoradiotherapy, albeit with a lower individual dose intensity to maintain a constant cumulative dose of $10^9$ cells. The median OS increases from $13.88$ months under Stupp protocol to $14.78$ months
(95\% CI: $14.47-15.11$ months) under 3C-Stupp-3C, corresponding to an increase of approximately $0.9$ months.

The KM curves show a statistically significant survival
improvement relative to Stupp alone ($p=0.000541$), with an HR of
$0.949$ (95\% CI: $0.922-0.978$) for Stupp relative to 3C-Stupp-3C (Fig.\ref{fig:3C_Stupp_3C_dynamics} bottom panel).
As in the previous combined schedules, the magnitude of the population-level
effect remains modest despite statistical significance.

The MVP dynamics illustrate how the residual tumor is subjected to repeated, complementary therapeutic pressures (Fig.~\ref{fig:3C_Stupp_3C_dynamics}, top panel). The initial CAR-T infusions, despite their lack of expansion, manage to constrain the immune-sensitive compartments. Nonetheless, the total tumor burden exceeds its  baseline volume by the time the concomitant Stupp phase begins. While RT and TMZ initially deplete these populations, the tumor recovers during the subsequent  adjuvant phase, reaching an even greater size. The final CAR-T courses then target the remaining sensitive cells, including the chemo-resistant fractions that retain susceptibility to immune-mediated killing. As observed previously, CAR-T cells expand more robustly when positioned after the Stupp protocol rather than before. However, compared to the Stupp-3C regimen, this second expansion is smaller, but it is counterbalanced by the early antitumor activity at the very beginning of the 3C-Stupp-3C sequence. This dynamic trade-off provides a mechanistic explanation for the remarkably similar population-level survival benefits yielded by the Stupp-3C and 3C-Stupp-3C protocols.

\begin{figure}[ht!]
    \centering
    \includegraphics[width=0.4\textwidth]{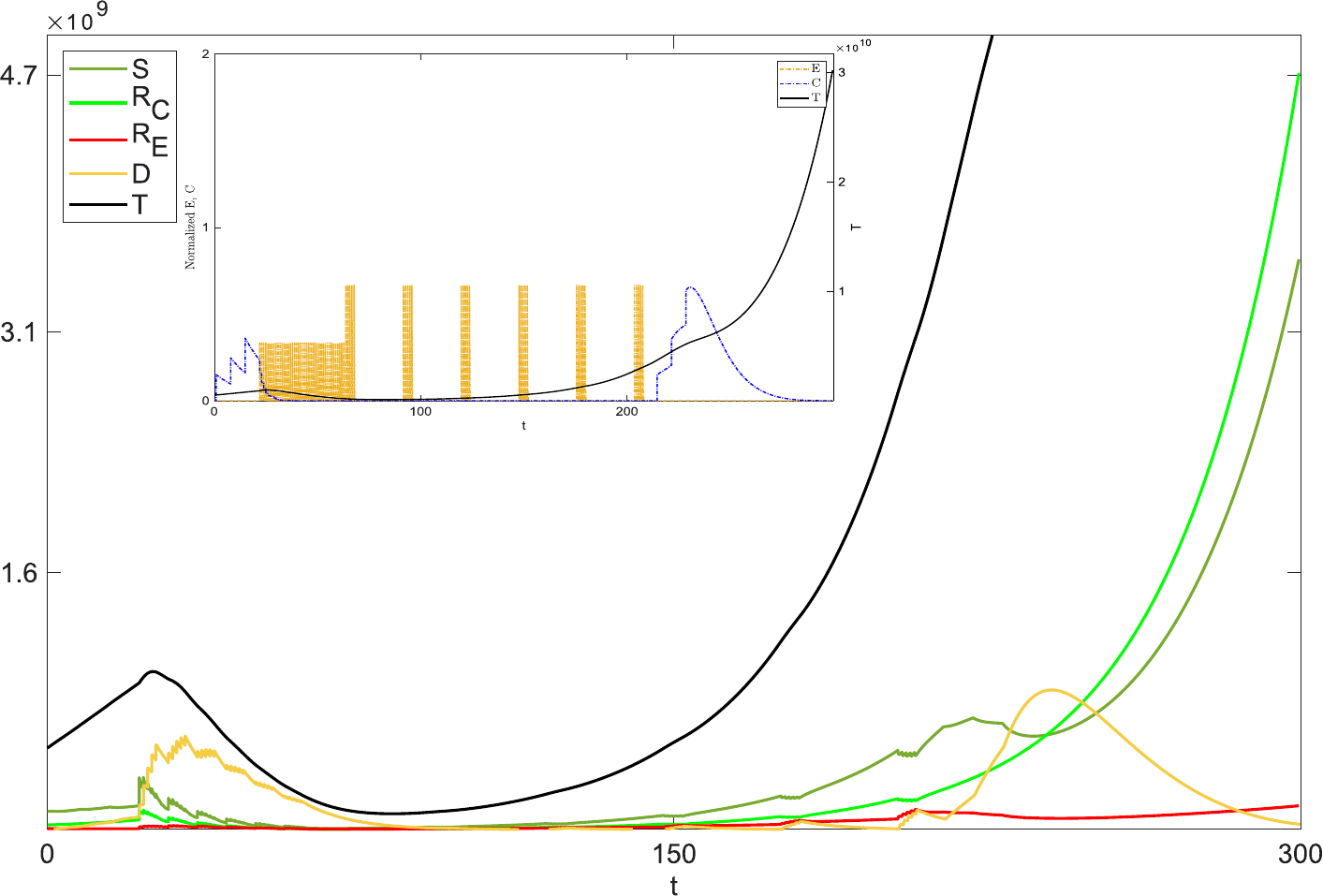}
    \includegraphics[width=0.4\textwidth]{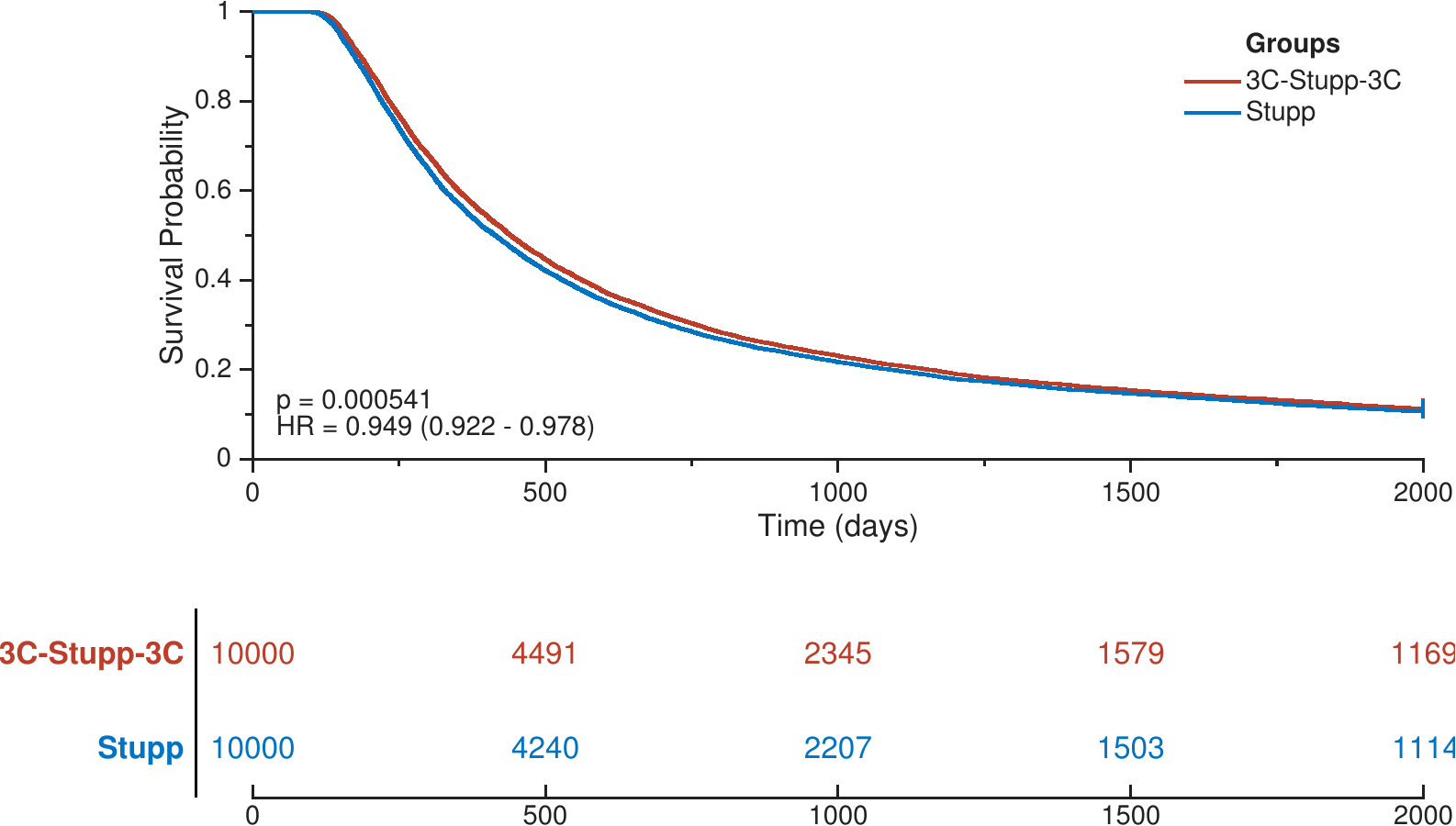}
\caption{3C-Stupp-3C treatment in the exploratory virtual cohort. Left:
dynamics of the median virtual patient under three CAR-T cell infusions,
the Stupp protocol, and three additional CAR-T infusions. Right: KM survival curves for the 3C-Stupp-3C and Stupp-treated cohorts ($N=10000$).}
    \label{fig:3C_Stupp_3C_dynamics}
\end{figure}

\paragraph{Stupp(C)-3C-Stupp(A) protocol.}

Next, we consider the Stupp(C)-3C-Stupp(A) sequence, in which the
concomitant phase of the Stupp protocol is followed by three CAR-T cell infusions before the Stupp adjuvant phase. This strategy therefore introduces CAR-T therapy between the two main phases of standard chemoradiotherapy \cite{Stupp2005}. The
median OS increases from $13.88$ months under Stupp protocol to $14.74$ months (95\% CI: $14.39-15.07$ months) under Stupp(C)-3C-Stupp(A), corresponding to an
increase of approximately $0.9$ months.

The KM curves show a statistically significant but modest
survival improvement relative to Stupp alone ($p=0.00045$), with a hazard
ratio of $0.949$ (95\% CI: $0.921-0.977$) for Stupp relative to Stupp(C)-3C-Stupp(A) (Fig.\ref{fig:StuppC_3C_Stupp_A_dynamics} bottom panel). As for the previous combined protocols, the small reduction in hazard indicates that the additional
population-level benefit remains limited in magnitude.

The MVP dynamics illustrate the sequential action of the three therapeutic components (Fig.~\ref{fig:StuppC_3C_Stupp_A_dynamics}, top panel). The concomitant RT-TMZ phase initially depletes the proliferative tumor populations, allowing the subsequent CAR-T cells to target the residual immune-sensitive disease before the adjuvant TMZ phase is introduced. By placing CAR-T therapy immediately after the radiation regimen, this temporal scheduling prevents the subsequent irradiation of the infused lymphocytes, even though the ensuing adjuvant TMZ phase still compromises their viability. This cooperative dynamic provides a mechanistic explanation for the slightly higher population-level survival benefit observed compared to the 3C-Stupp protocol.

\begin{figure}[ht!]
    \centering
    \includegraphics[width=0.4\textwidth]{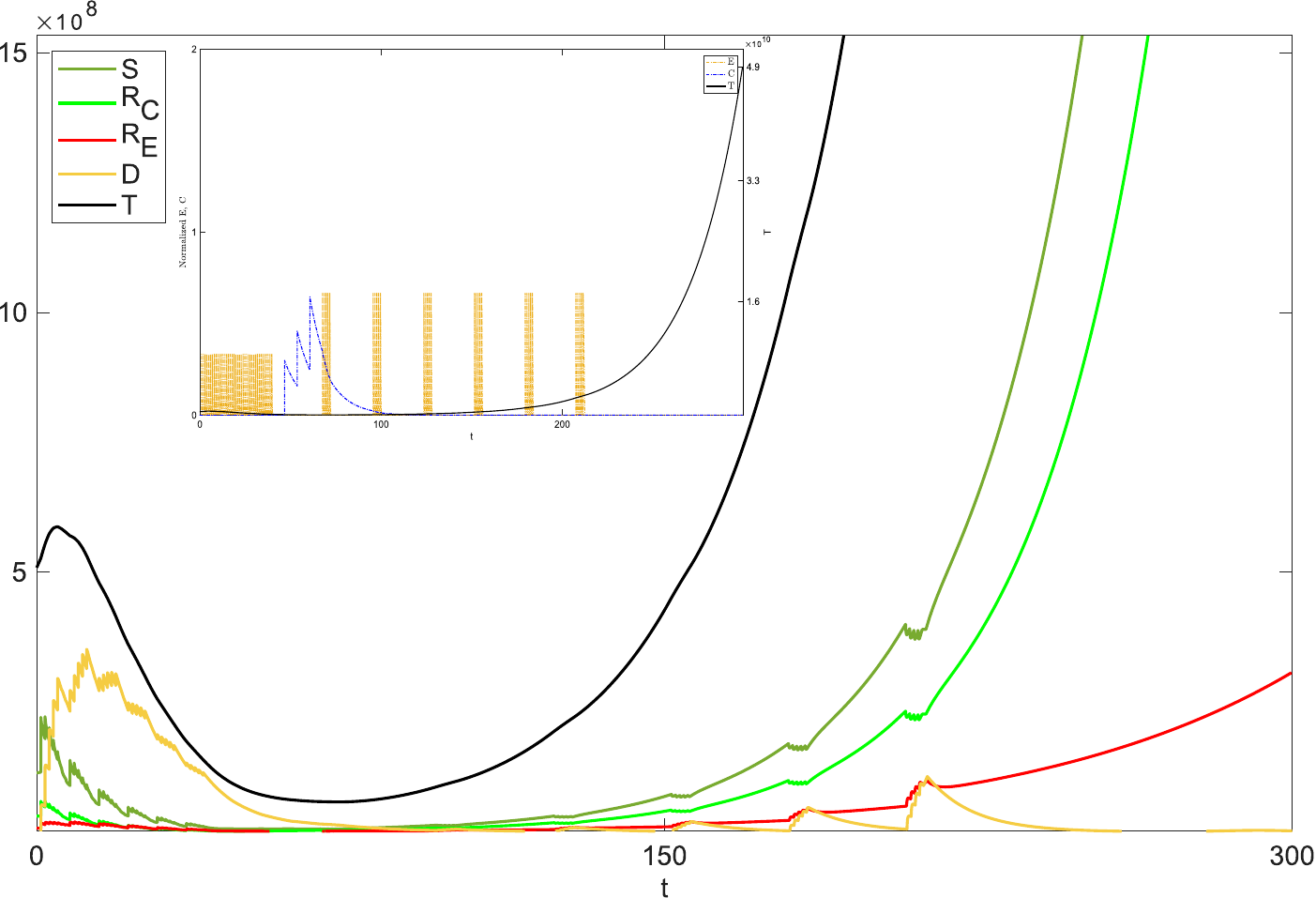} 
    \includegraphics[width=0.4\textwidth]{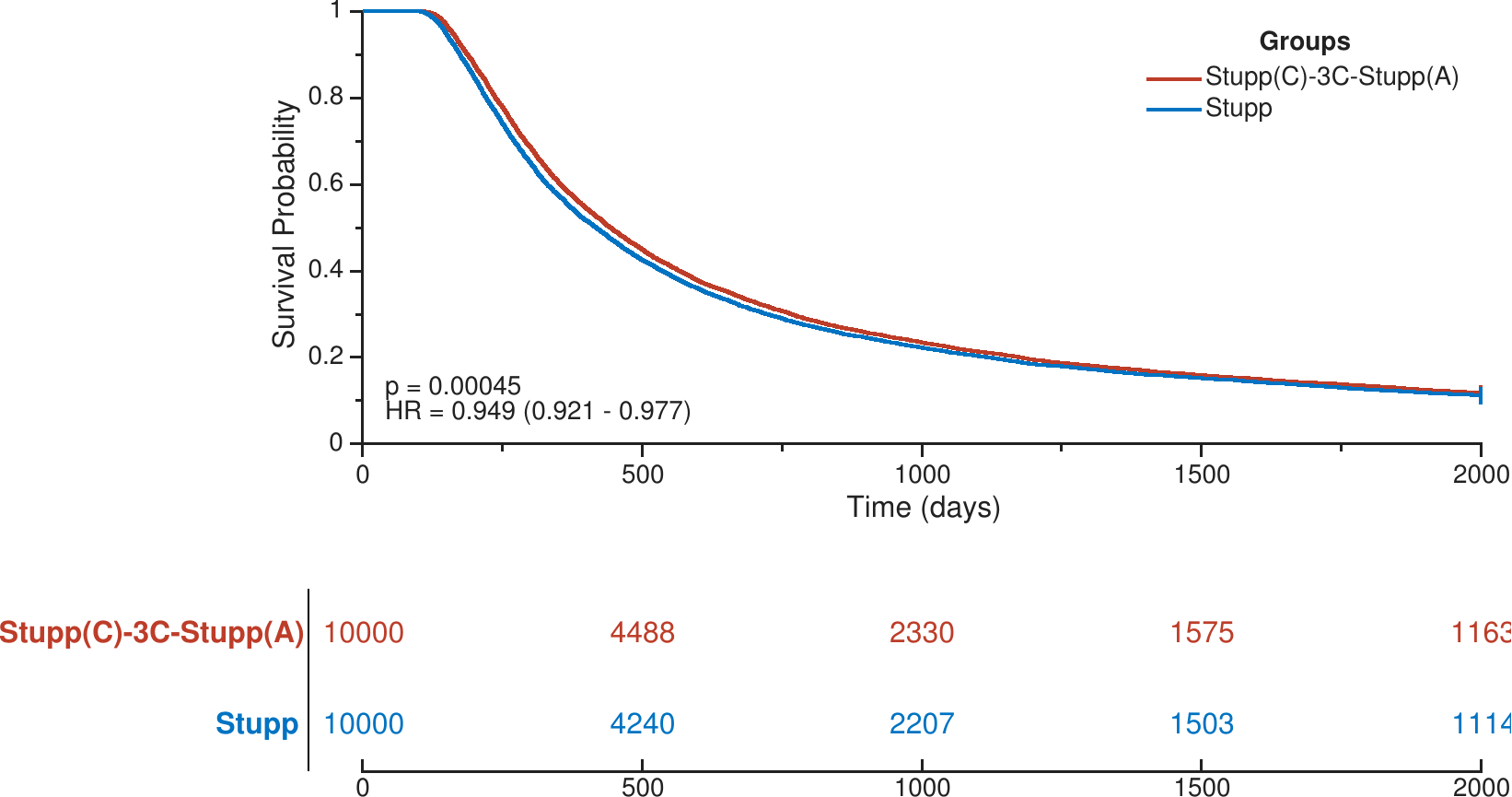}
\caption{Stupp(C)-3C-Stupp(A) treatment in the exploratory virtual
cohort. Left: dynamics of the median virtual patient under the concomitant phase of the Stupp protocol, three CAR-T cell infusions, and the subsequent
adjuvant TMZ phase. Right: KM survival curves for the Stupp(C)-3C-Stupp(A) and Stupp-treated cohorts ($N=10000$).}
    \label{fig:StuppC_3C_Stupp_A_dynamics}
\end{figure}

\paragraph{Stupp(C)-3C-Stupp(A)-3C protocol.}

Finally, we consider the Stupp(C)-3C-Stupp(A)-3C sequence. In this protocol, the concomitant phase of Stupp is followed by three CAR-T cell
infusions, after which the adjuvant TMZ phase is administered and a second block of three CAR-T infusions is given after completion of the Stupp regimen. Thus, CAR-T therapy is applied both between the concomitant and
adjuvant phases and again after completion of standard chemoradiotherapy.
The median OS increases from $13.88$ months
under Stupp protocol to $14.74$ months (95\% CI: $14.40-15.08$ months) under Stupp(C)-3C- Stupp(A)-3C, corresponding to an increase of approximately $0.9$ months.

The KM curves show a statistically significant but modest
survival improvement relative to Stupp alone ($p=0.00053$), with a hazard
ratio of $0.949$ (95\% CI: $0.922-0.978$) for Stupp relative to 
Stupp(C)-3C Stupp(A)-3C (Fig.\ref{fig:StuppC_3C_Stupp_A_3C_dynamics} bottom panel). As observed for the other combined protocols, the magnitude of the population-level effect remains
limited despite the statistically significant difference.

The MVP dynamics illustrate the impact of multi-dose CAR-T therapy targeting the residual disease at different stages of chemoradiotherapy (Fig.~\ref{fig:StuppC_3C_Stupp_A_3C_dynamics}, top panel). The first CAR-T block is introduced immediately after the concomitant RT-TMZ phase, whereas the second targets the viable burden remaining after adjuvant TMZ. Although treatment-resistant and residual viable clones ultimately persist to drive tumor regrowth, CAR-T cells undergo significant expansion only when administered after the adjuvant chemotherapy. Interestingly, delivering the initial CAR-T course before the Stupp protocol rather than between its concomitant and adjuvant phases leaves the tumor slightly more susceptible to immunotherapy prior to the final three infusions (compare the green $R_C$ lines in the top panels of Figs.~\ref{fig:3C_Stupp_3C_dynamics} and~\ref{fig:StuppC_3C_Stupp_A_3C_dynamics}). This small divergence in tumor composition provides a mechanistic explanation for the slightly higher population-level survival benefit observed under the 3C-Stupp-3C protocol compared to the Stupp(C)-3C-Stupp(A)-3C regimen.

\begin{figure}[ht!]
    \centering
    \includegraphics[width=0.4\textwidth]{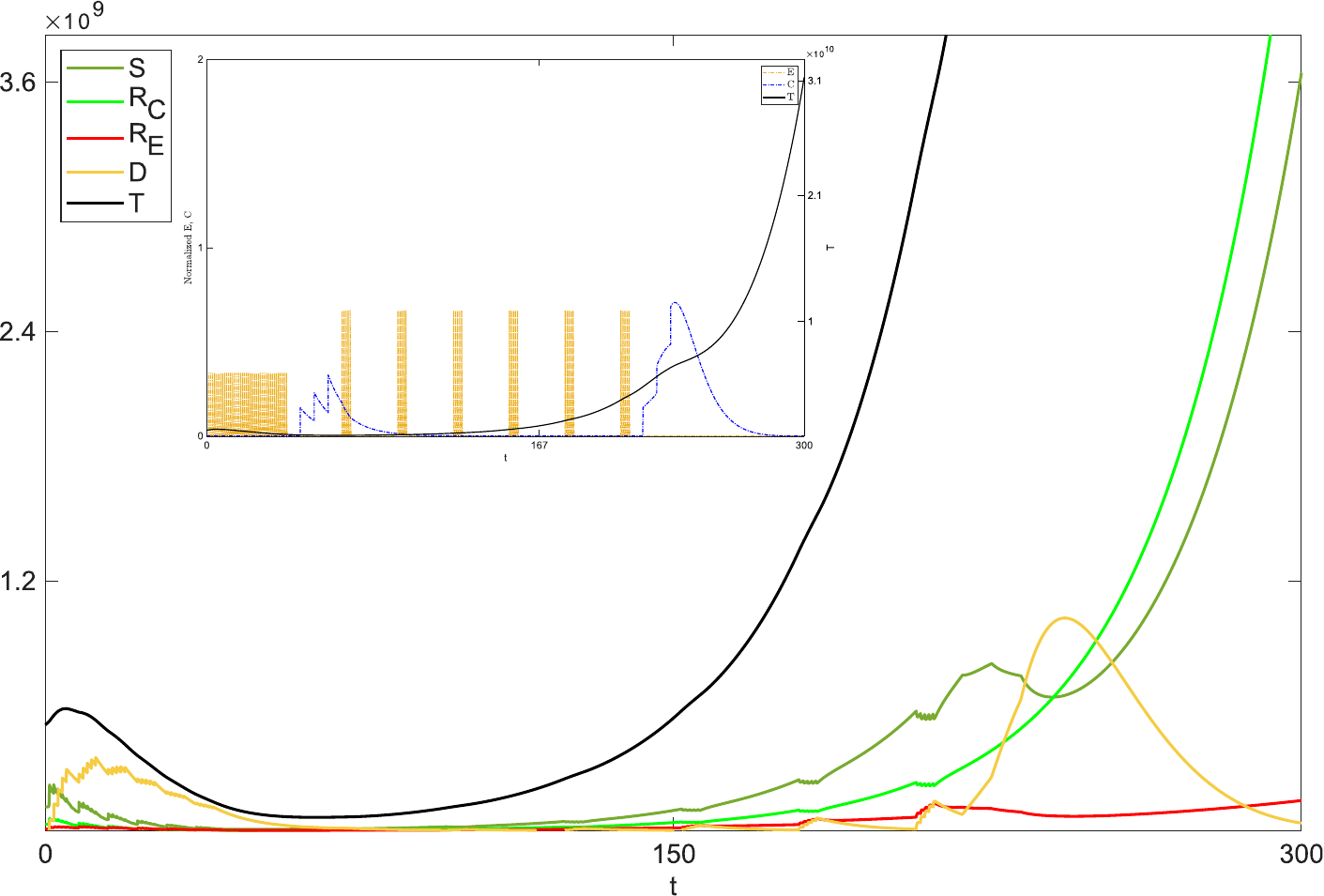} 
    \includegraphics[width=0.4\textwidth]{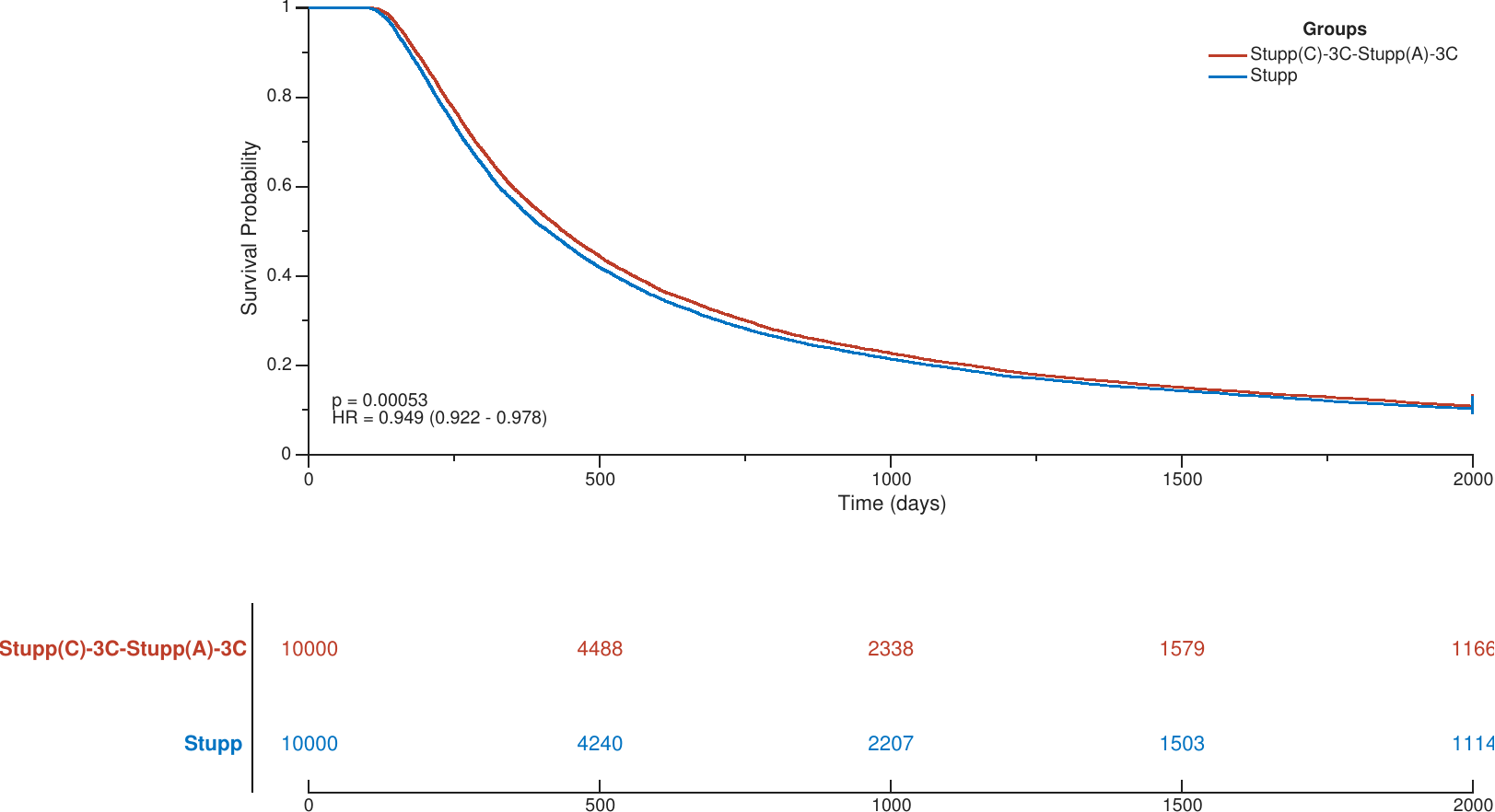}
  \caption{Stupp(C)-3C-Stupp(A)-3C treatment in the exploratory virtual
cohort. Lefy: dynamics of the median virtual patient under the concomitant
phase of the Stupp protocol, three CAR-T cell infusions, the adjuvant TMZ
phase, and three additional CAR-T infusions. Right: KM survival curves for the Stupp(C)-3C-Stupp(A)-3C and Stupp-treated cohorts ($N=10000$).}
    \label{fig:StuppC_3C_Stupp_A_3C_dynamics}
\end{figure}
Table~\ref{tab:combined_protocols} summarizes the population-level outcomes of the five combined treatment sequences. All combined CAR-T-Stupp protocols investigated here improve median OS
relative to Stupp alone, with median gains ranging from $0.72$ to $0.90$ months and HRs between approximately $1.04$ and $1.06$ for Stupp relative
to the corresponding combined protocol. Thus, the incorporation of CAR-T therapy produces a consistent, although modest, population-level survival benefit over standard
Stupp protocol.

In contrast, the differences among the five combined treatment sequences
are very small. Median OS varies by only $0.11$ months across the investigated schedules. The 3C--Stupp-3C strategy yields the largest
median OS numerically ($14.78$ months), but its advantage over the
other CAR-T-Stupp combined sequences is very modest. The KM analyses similarly yield closely comparable HRs
across protocols. Therefore, although CAR-T administration improves population-level outcome relative to Stupp, the present results do not
identify a clearly superior temporal ordering of the treatments at the
population level.

These findings suggest that the principal population-level effect arises
from the incorporation of CAR-T therapy itself rather than from the precise
ordering of the five treatment sequences considered here. However, population-level similarity does not imply that the protocols are equivalent
for every VP. The possibility that different treatment sequences may benefit different patient phenotypes will therefore be investigated below through within-patient analyses of individual therapeutic benefit.

\begin{table*}[ht!]
\centering
\small
\setlength{\tabcolsep}{4pt}
\caption{Population-level survival outcomes for the Stupp protocol
and the five CAR-T--Stupp treatment sequences in the common cohort of
$N=10000$ VPs. The survival gain is computed relative to Stupp alone.
Hazard ratios (HRs) are reported for each combined protocol relative to the Stupp protocol.}

\label{tab:combined_protocols}
\begin{tabular}{lcccc}
\hline
Protocol &
Median OS [95\% CI] (months) &
Gain vs Stupp (months) &
HR [95\% CI] &
$p$-value \\
\hline
Stupp
& 13.88 [13.57, 14.21]
& --
& --
& -- \\

3C-Stupp
& 14.60 [14.33, 14.98]
& 0.72
& 0.953 [0.925 - 0.981]
& 0.00136 \\

Stupp-3C
& 14.74 [14.40, 15.06]
& 0.86
& 0.96 [0.932 - 0.988]
& 0.00611 \\

3C-Stupp-3C
& 14.78 [14.45, 15.11]
& 0.90
& 0.949 [0.922 - 0.978]
& 0.000541 \\

Stupp(C)-3C-Stupp(A)
& 14.74 [14.39, 15.07]
& 0.86
& 0.949 [0.921 - 0.977]
& 0.000450 \\

Stupp(C)-3C-Stupp(A)-3C
& 14.74 [14.40, 15.08]
& 0.86
& 0.949 [0.922 - 0.978]
& 0.00053 \\
\hline
\end{tabular}
\end{table*}

\subsubsection{Exploratory virtual cohort: CAR-T sensitivity}

Having established that the five primary combined sequences yield similar population-level survival outcomes, we next investigate the robustness of these findings to variations in the administration schedule. Specifically, we examine the effects of modifying the infusion frequency, the total cellular dose, the temporal gaps between immunotherapy and the Stupp protocol, the spacing between consecutive injections, and the dose allocation across multiple treatment blocks. The purpose of this analysis is not to propose an entirely new protocol, but rather to determine whether population-level outcomes are highly sensitive to scheduling nuances, or if the modest benefits observed are robust across clinically plausible 
patterns. 

Our sensitivity analysis reveals that the survival benefit remains remarkably robust to moderate scheduling variations. Fine-tuning the number of infusions, dose spacing, or temporal alignment with chemoradiotherapy produces only limited changes in median OS, with the most noticeable fluctuations occurring when modifying the pre-Stupp schedules. In contrast, expanding the total pool of administered CAR-T cells provides a more consistent survival gain, though its magnitude remains modest. Overall, these results suggest that, within the explored parameter ranges, the primary population-level effect stems from the integration of CAR-T cells into the regimen itself, rather than from the fine optimization of their timing or fractionation. Detailed results for the different scheduling are provided in Appendix \ref{app:additional_sensitivity}

\paragraph{Number of CAR-T infusions and total dose.}
We first examine whether population-level outcomes depend on the scheduling and total number of infused CAR-T cells. For this purpose, the cumulative dose is distributed across one to ten administrations, considering immunotherapy either before or after the Stupp protocol, with total doses of $10^9$ and 
$2\times10^9$ cells.

When CAR-T therapy follows the Stupp protocol, increasing the number of infusions produces essentially no change in survival for a fixed total dose: median OS 
remains tightly constrained between $14.72$ and $14.74$ months for $10^9$ cells, 
and between $14.86$ and $14.88$ months for $2\times10^9$ cells.

Conversely, when CAR-T therapy precedes the Stupp protocol, a slightly stronger dependence on the number of injections is observed, though the overall magnitude remains limited. Here, median OS ranges from $14.57$ to $14.75$ months for $10^9$ cells, and from $14.69$ to $14.97$ months for $2\times10^9$ cells. In this pre-Stupp setting, median survival increases with fractionated dosing due to 
the cytotoxic effects of subsequent TMZ-RT during the concomitant phase. Splitting 
the total dose into multiple injections allows the CAR-T cells to exert therapeutic 
pressure more progressively and over a longer window.

Table \ref{tab:sensitivity_to_n_inj_tot_inj} summarizes the results (the complete numerical outcomes are reported in Appendix~\ref{app:number_cart_infusions}). In conclusion, while increasing the total CAR-T cell count yields an additional survival gain, fractionating a fixed total dose into more infusions has a 
population-level impact only when administered pre-Stupp.

\begin{table}[ht]
\centering
\caption{Summary of the sensitivity of median overall survival to the number of CAR-T infusions and the total administered CAR-T dose. CAR-T therapy is administered either before or after the Stupp protocol, considering between one and ten infusions ($nC$ with $n=1,\dots,10$) and total doses of $10^9$ or $2\times10^9$ cells.}\label{tab:sensitivity_to_n_inj_tot_inj}
\begin{tabular}{lcc}
\hline
Sequence & Total CAR-T & Median OS range (months) \\
\hline
Stupp--$n$C & $10^9$        & $14.72$--$14.74$ \\
$n$C--Stupp & $10^9$        & $14.49$--$14.69$ \\
Stupp--$n$C & $2\times10^9$ & $14.86$--$14.88$ \\
$n$C--Stupp & $2\times10^9$ & $14.59$--$14.92$ \\
\hline
\end{tabular}
\end{table}

\paragraph{Treatment spacing.}


We next examine the sensitivity of the combined protocols to two temporal parameters: the spacing between CAR-T therapy and the different phases of the Stupp protocol (denoted by $Gap_{SC}$), and the interval between consecutive CAR-T infusions (denoted by $Gap_C$). Specifically, we evaluate the 3C-Stupp, Stupp-3C, and Stupp(C)-3C-Stupp(A) sequences, varying $Gap_C$ across $7$, $14$, $21$, and $28$ days, while exploring $Gap_{SC}$ over a broader range from $7$ to $90$ days. The complete numerical outcomes are reported in Appendix~\ref{app:treatment_spacing}. 

The impact of treatment spacing is highly dependent on the scheduling of the CAR-T block. When immunotherapy precedes the Stupp protocol (3C-Stupp), wider spacing yields a modest increase in median OS across the tested intervals. Conversely, when CAR-T therapy follows the completion of the Stupp protocol (Stupp-3C), longer delays are progressively associated with lower median survival, with the reduction becoming more pronounced at higher $Gap_{SC}$ values. When CAR-T infusions are intercalated between the concomitant and adjuvant phases, median OS remains comparatively insensitive to moderate adjustments in both $Gap_C$ and $Gap_{SC}$. Thus, the influence of treatment spacing is protocol-dependent rather than uniform. For short-to-moderate separations, fluctuations in median OS remain generally bounded. The most favorable outcomes occur when extending both $Gap_C$ and $Gap_{SC}$ prior to the Stupp regimen, which serves to prolong the therapeutic window of the CAR-T cells. On the other hand, substantially delaying CAR-T administration post-Stupp triggers a more appreciable drop in survival, likely because highly vulnerable VPs succumb to the disease before receiving the full course of infusions. Overall, these findings underscore the robustness of the population-level benefit against moderate variations in CAR-T scheduling.

\paragraph{Distribution of CAR-T cell doses.}

Finally, we investigate split-course protocols where the CAR-T therapy is divided 
into two distinct blocks of three infusions each. Specifically, we evaluate the 
3C-Stupp-3C and Stupp(C)-3C-Stupp(A)-3C sequences, which position these blocks 
either around the entire Stupp protocol or within its internal phases. 
The complete results for the different dose distributions and treatment intervals 
are reported in Appendix~\ref{app:prepost_distribution}.

For the 3C-Stupp-3C protocol, varying the pre/post distribution of a total
CAR-T dose of $10^9$ cells and the corresponding treatment spacing yields median OS values between $14.59$ and $14.80$ months. Increasing the total CAR-T dose to $2\times10^9$ cells shifts this range modestly upwards, to approximately $14.70-15.00$ months. A similar pattern is
observed for Stupp(C)-3C-Stupp(A)-3C, for which median survival ranges
from approximately $14.50$ to $14.75$ months with $10^9$ CAR-T cells and
from $14.59$ to $14.95$ months with $2\times10^9$ cells.

Although the precise distribution of CAR-T therapy between the two blocks modifies the median outcome, these differences remain small at the
population level. Increasing the total CAR-T dose produces a somewhat more consistent improvement than redistributing a fixed dose between treatment blocks. Moreover, longer temporal separations generally do not improve median survival and, in several configurations, are associated with a slightly lower outcome. Overall, these results indicate that the population-level benefit of the combined treatment is relatively robust to
moderate changes in the distribution and timing of CAR-T administration.

\subsubsection{Exploratory virtual cohort: heterogeneity of CAR-T benefit}
\label{sec:individual_benefit}

The population-level analyses above show that the five CAR-T-Stupp
protocols produce similar median OS outcomes, with a modest overall
improvement relative to Stupp alone. However, similarity at the population
level does not imply that the therapeutic benefit is equally distributed
across VPs. In particular, a modest change in median OS
may coexist with substantial benefit in a subset of patients and little or
even no benefit in others.

The matched design of the virtual population allows for this heterogeneity to be investigated directly. Since each VP is simulated under Stupp and under each of the five combined protocols while retaining exactly the
same biological parameters and initial conditions, treatment outcomes can be
compared individually. For each VP $i$ and combined protocol $P$, we define the absolute survival benefit as
\begin{equation}
\Delta_i^P=T_{s,i}^{P}-T_{s,i}^{\mathrm{Stupp}},
\label{eq:absolute_benefit}
\end{equation}
where $T_{s,i}^{P}$ and $T_{s,i}^{\mathrm{Stupp}}$ denote the survival times
of VP $i$ under protocol $P$ and under the Stupp
protocol, respectively. Thus, $\Delta_i^P>0$ indicates an individual
survival benefit from adding CAR-T therapy, whereas $\Delta_i^P<0$ indicates
that the corresponding combined protocol performs worse than Stupp for that
VP.

Since a given absolute survival extension carries different weight depending on a patient's baseline prognosis, we also consider the
multiplicative survival benefit
\begin{equation}
R_i^P=
\frac{T_{s,i}^{P}}{T_{s,i}^{\mathrm{Stupp}}},
\label{eq:relative_benefit}
\end{equation}
for which $R_i^P>1$ indicates improved survival relative to Stupp,
$R_i^P=1$ indicates no change, and $R_i^P<1$ indicates a poorer outcome.
The two measures provide complementary information: $\Delta_i^P$ quantifies
the absolute survival extension, whereas $R_i^P$ expresses the same change
relative to the survival achieved under standard chemoradiotherapy.

\paragraph{Distribution of individual benefit.}

Figure \ref{fig:benefit_distributions} shows the distributions of the
absolute survival benefit $\Delta_i^P$ for the five combined protocols.
Although the population-level median survival gains reported above are
smaller than one month, the individual-level distributions are markedly
heterogeneous. A large proportion of VPs experience relatively
small survival extensions, whereas a substantially smaller subgroup exhibits
much larger benefits, producing a pronounced right tail in the
distributions. Thus, the modest population-level effect of adding CAR-T
therapy does not arise from an identical response across the virtual
population, but instead masks considerable inter-patient variability in the
magnitude of therapeutic benefit.

\begin{figure}[ht!]
    \centering
    \includegraphics[width=\textwidth]{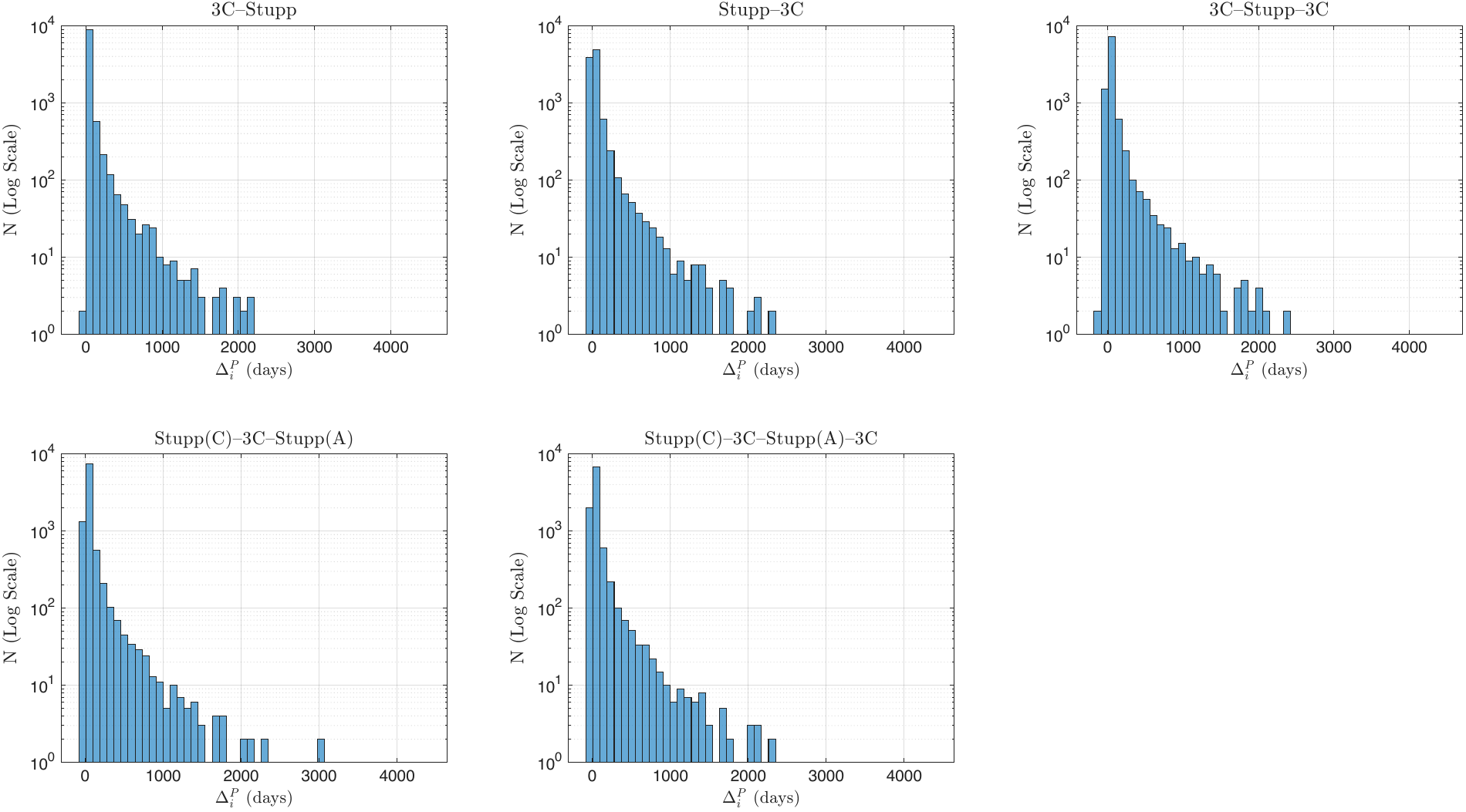}
\caption{Distribution of the absolute individual survival benefit
$\Delta_i^P=T_{s,i}^{P}-T_{s,i}^{\mathrm{Stupp}}$ across the five
CAR-T--Stupp protocols in the common cohort of $N=10000$ VPs.} 
\label{fig:benefit_distributions}

\end{figure}

\begin{figure}[ht!]
    \centering
    \includegraphics[width=\textwidth]{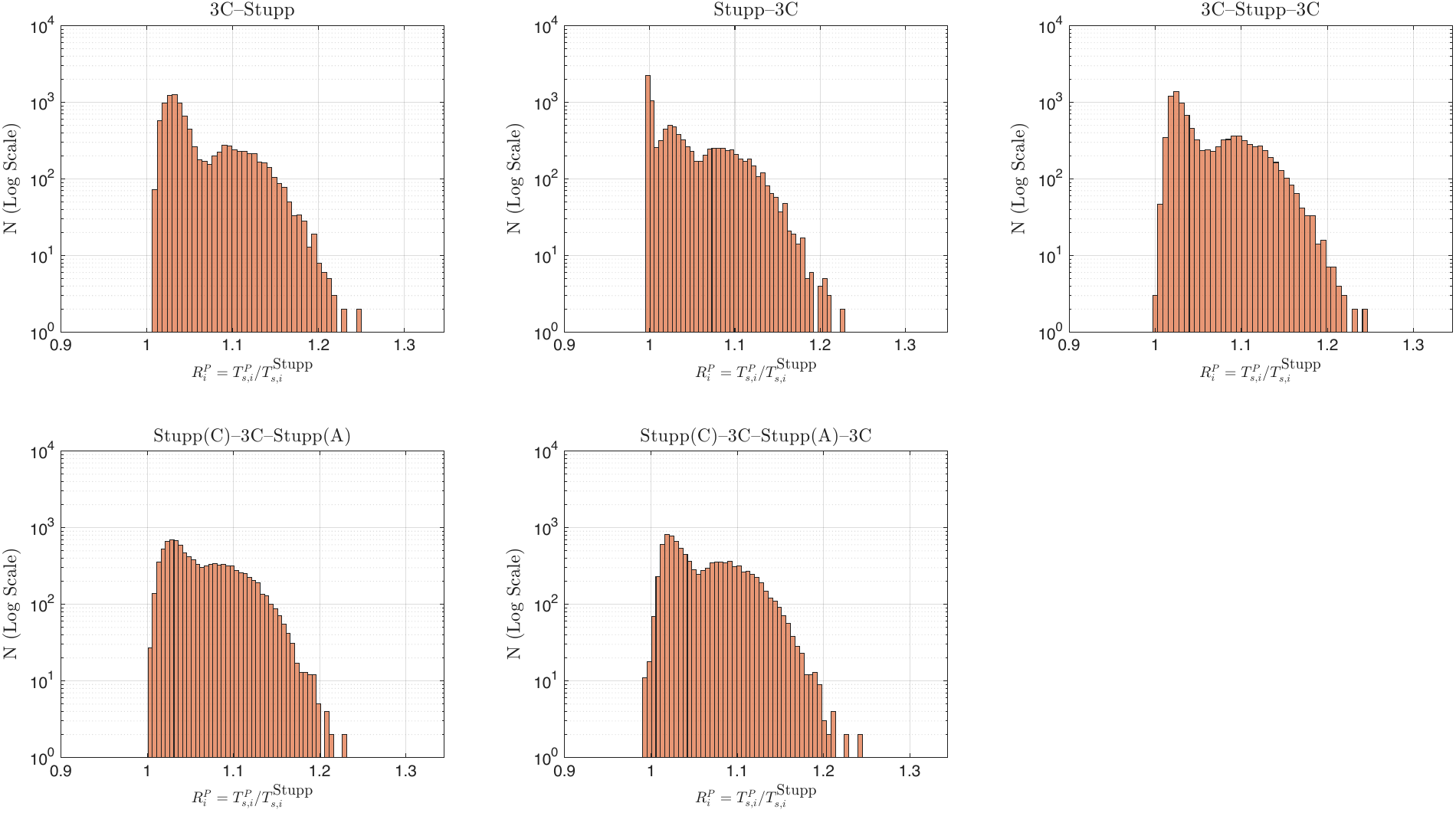}
\caption{Distribution of the multiplicative survival benefit
$R_i^P=T_{s,i}^{P}/T_{s,i}^{\mathrm{Stupp}}$ across the five
CAR-T--Stupp protocols in the common cohort of $N=10000$ VPs.}
    \label{fig:multiplicative_benefit_distributions}
\end{figure}

To further quantify the heterogeneity observed in the benefit distributions,
we calculated the proportion of VPs reaching two descriptive
thresholds of absolute survival extension. As shown in
Table~\ref{tab:absolute_benefit_thresholds}, approximately $19$--$22\%$ of VPs achieve an absolute survival extension greater than $60$ days, whereas only about $3\%$ achieve an extension greater than one year. These thresholds are used here as descriptive measures of the upper tail of the benefit distribution rather than as clinically validated response
criteria.

\begin{table}[htpb]
    \caption{Number and percentage of VPs ($N=10000$) exhibiting absolute survival extensions larger than $60$ days and $365$ days under each combined protocol relative to the Stupp protocol.}
    \label{tab:absolute_benefit_thresholds}
    \centering
    \small
    \begin{ruledtabular}
    \begin{tabular}{lcccc}
        \multirow{2}{*}{\textbf{Protocol ($P$)}} & \multicolumn{2}{c}{$\Delta_i^P > 60\text{ days}$} & \multicolumn{2}{c}{$\Delta_i^P > 365\text{ days}$} \\
        \cline{2-3} \cline{4-5}
        & $N_{\text{resp}}$ & $\%$ & $N_{\text{resp}}$ & $\%$ \\
        \hline
        3C--Stupp                   & 1828 & 18.28 & 290 & 2.90 \\
        Stupp--3C                   & 2127 & 21.27 & 307 & 3.07 \\
        3C--Stupp--3C               & 2126 & 21.26 & 311 & 3.11 \\
        Stupp(C)--3C--Stupp(A)      & 1940 & 19.40 & 285 & 2.85 \\
        Stupp(C)--3C--Stupp(A)--3C  & 2069 & 20.69 & 299 & 2.99 \\
    \end{tabular}
    \end{ruledtabular}
\end{table}

Absolute survival extensions provide an intuitive measure of treatment
benefit, but their magnitude may partly depend on less aggressive baseline tumor kinetics. We therefore also quantify individual
benefit in relative terms using the multiplicative survival ratio $R_i^P$
defined above. Table~\ref{tab:relative_benefit_thresholds} reports the proportion of VPs whose survival increases by more than $5\%$, $10\%$, $20\%$, and $30\%$ under each combined protocol relative to Stupp.

\begin{table*}[t!]
    \caption{Number and percentage of VPs ($N=10000$) achieving relative survival improvements greater than $5\%$, $10\%$, $20\%$, and $30\%$
under each combined protocol with respect to the Stupp protocol.}
    \label{tab:relative_benefit_thresholds}
    \centering
    \small
    \begin{ruledtabular}
    \begin{tabular}{lcccccccc}
        \multirow{2}{*}{\textbf{Protocol ($P$)}} & \multicolumn{2}{c}{$R_i^P > 1.05$} & \multicolumn{2}{c}{$R_i^P > 1.10$} & \multicolumn{2}{c}{$R_i^P > 1.20$} & \multicolumn{2}{c}{$R_i^P > 1.30$} \\
        \cline{2-3} \cline{4-5} \cline{6-7} \cline{8-9}
        & $N_{\text{resp}}$ & \% & $N_{\text{resp}}$ & \% & $N_{\text{resp}}$ & \% & $N_{\text{resp}}$ & \% \\
        \hline
        3C--Stupp                      & 3995 & 39.95 & 2033 & 20.33 &   24 &  0.24 &    0 &  0.00 \\
        Stupp--3C                      & 3616 & 36.16 & 1472 & 14.72 &   15 &  0.15 &    0 &  0.00 \\
        3C--Stupp--3C                  & 4690 & 46.90 & 2280 & 22.80 &   24 &  0.24 &    0 &  0.00 \\
        Stupp(C)--3C--Stupp(A)         & 5508 & 55.08 & 2199 & 21.99 &   10 &  0.10 &    0 &  0.00 \\
        Stupp(C)--3C--Stupp(A)--3C     & 5301 & 53.01 & 2225 & 22.25 &   12 &  0.12 &    0 &  0.00 \\
    \end{tabular}
    \end{ruledtabular}
\end{table*}

Approximately $36-55\%$ of VPs achieve a relative survival improvement greater than $5\%$, whereas about $15-24\%$ exceed a $10\%$ improvement. By contrast, relative gains above $20\%$ are very uncommon, and no VP achieves an improvement greater than $30\%$ in any of the investigated protocols.

\paragraph{Model parameters associated with CAR-T benefit.}

Having established that individual survival gains are highly heterogeneous 
across the virtual population, we next investigate which patient-specific 
parameters drive a larger benefit from CAR-T therapy. Because the exact 
same VPs are tracked across all regimens, these associations reflect 
divergences in therapeutic return within matched individuals, rather than 
variability between independently generated cohorts.

We first examine the absolute survival benefit $\Delta_i^P$. 
Figures~\ref{fig:scatter_3C_Stupp}--\ref{fig:scatter_StuppC_3C_StuppA_3C} 
plot $\Delta_i^P$ against the tumor proliferation rate $r_1$, the initial 
CAR-T-resistant fraction $\delta_2$, and the tumor-mediated CAR-T inactivation 
parameter $\rho_4$ for each combination protocol. Rather than simply picking the 
three strongest univariate correlations, we selected these specific parameters to 
illustrate three complementary determinants of response: intrinsic tumor 
growth ($r_1$), pre-existing immune resistance ($\delta_2$), and tumor-mediated CAR-T suppression ($\rho_4$).

\begin{figure}[ht!]
        \centering
        \includegraphics[width=\textwidth]{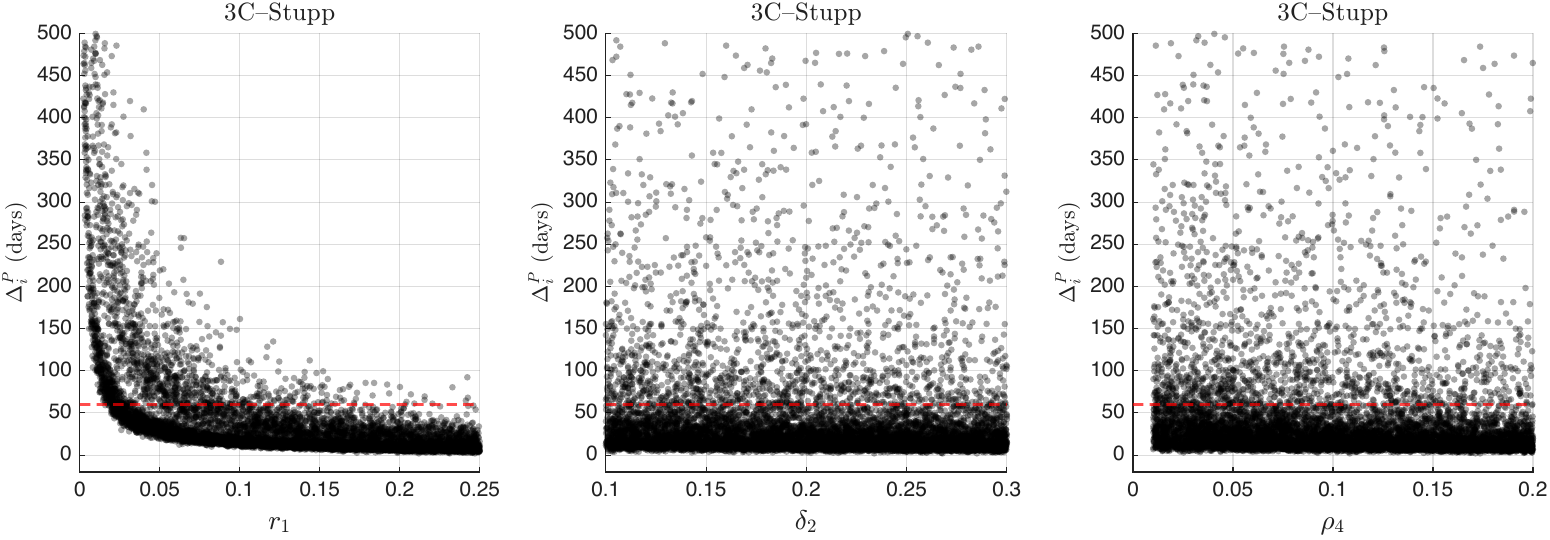}
        \caption{Individual benefit $(\Delta_i^P)$ versus $r_1$, $\delta_2$ and $\rho_4$ for the 3C--Stupp protocol.}
        \label{fig:scatter_3C_Stupp}
    \end{figure}
    
    \begin{figure}[ht!]
        \centering
        \includegraphics[width=\textwidth]{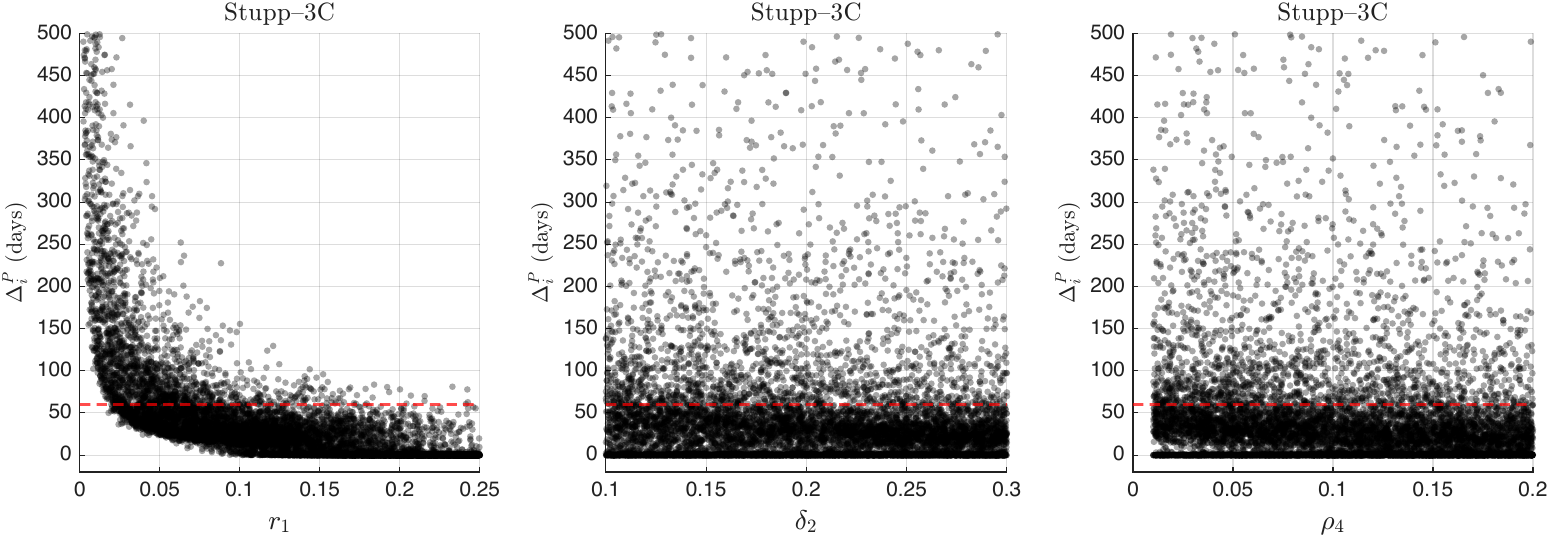}
        \caption{Individual benefit $(\Delta_i^P)$ versus $r_1$, $\delta_2$ and $\rho_4$ for the Stupp--3C protocol.}
        \label{fig:scatter_Stupp_3C}
    \end{figure}
    
    \begin{figure}[ht!]
        \centering
        \includegraphics[width=\textwidth]{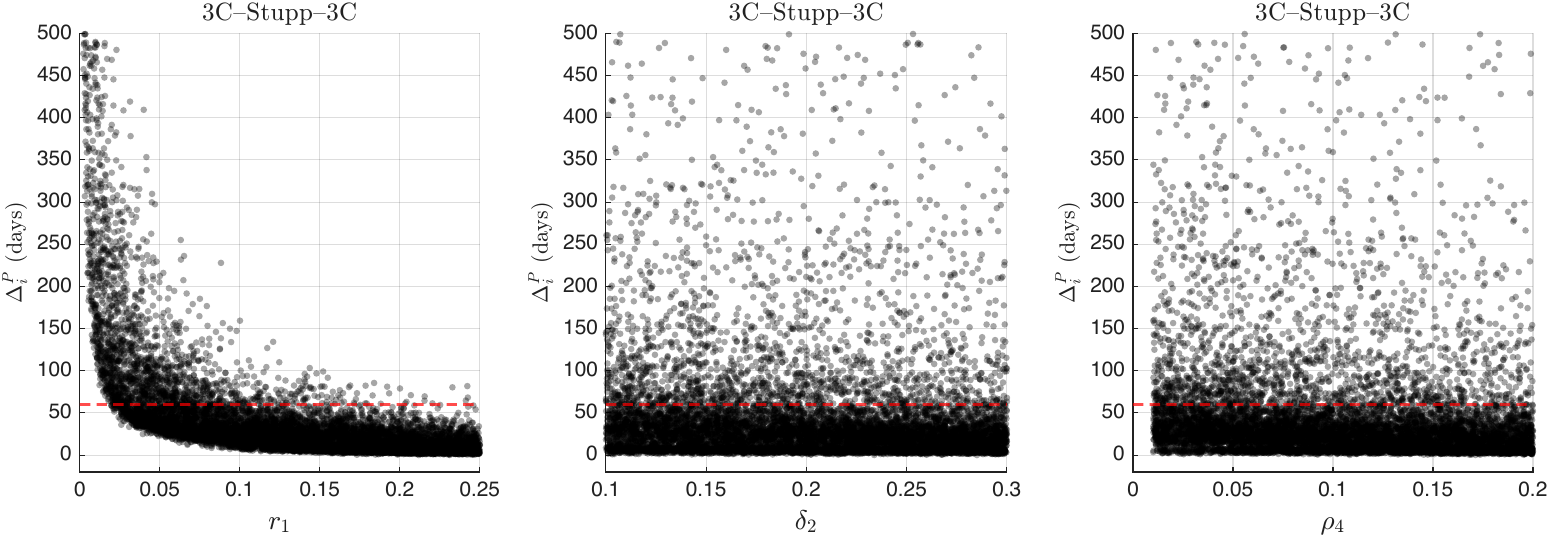}
        \caption{Individual benefit $(\Delta_i^P)$ versus $r_1$, $\delta_2$ and $\rho_4$ for the 3C--Stupp--3C protocol.}
        \label{fig:scatter_3C_Stupp_3C}
    \end{figure}
    
    \begin{figure}[ht!]
        \centering
        \includegraphics[width=\textwidth]{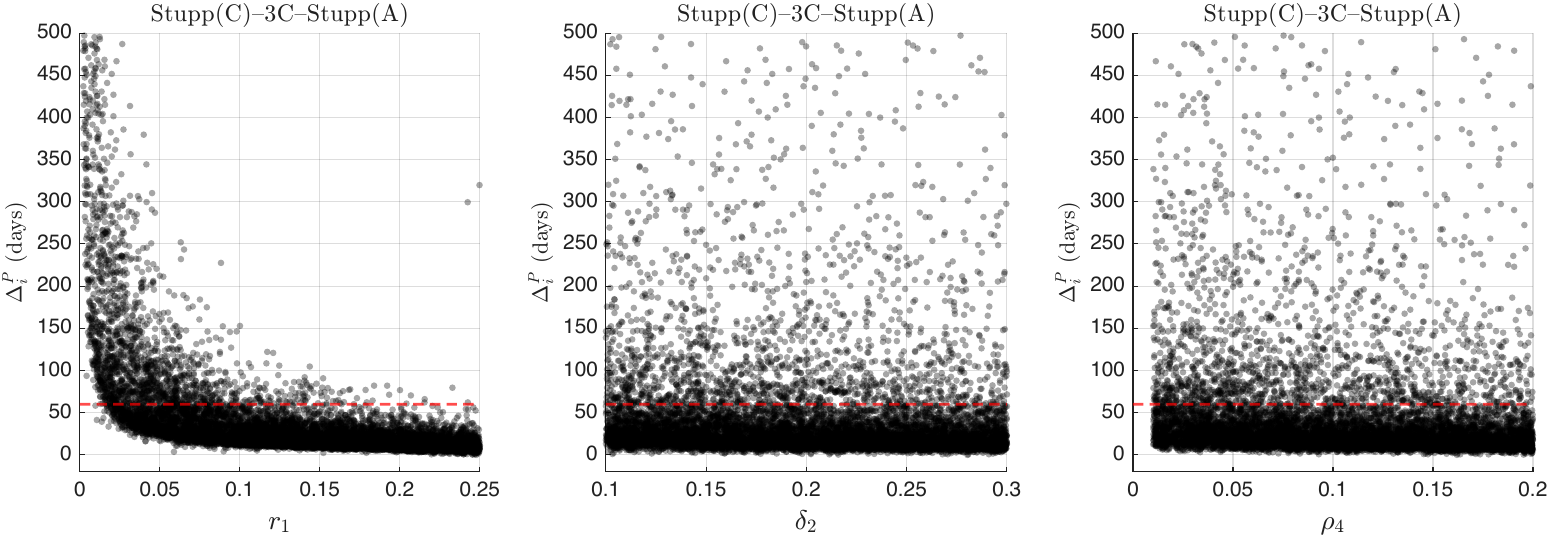}
        \caption{Individual benefit $(\Delta_i^P)$ versus $r_1$, $\delta_2$ and $\rho_4$ for the Stupp(C)--3C--Stupp(A) protocol.}
        \label{fig:scatter_StuppC_3C_StuppA}
    \end{figure}
    
    \begin{figure}[ht!]
        \centering
        \includegraphics[width=\textwidth]{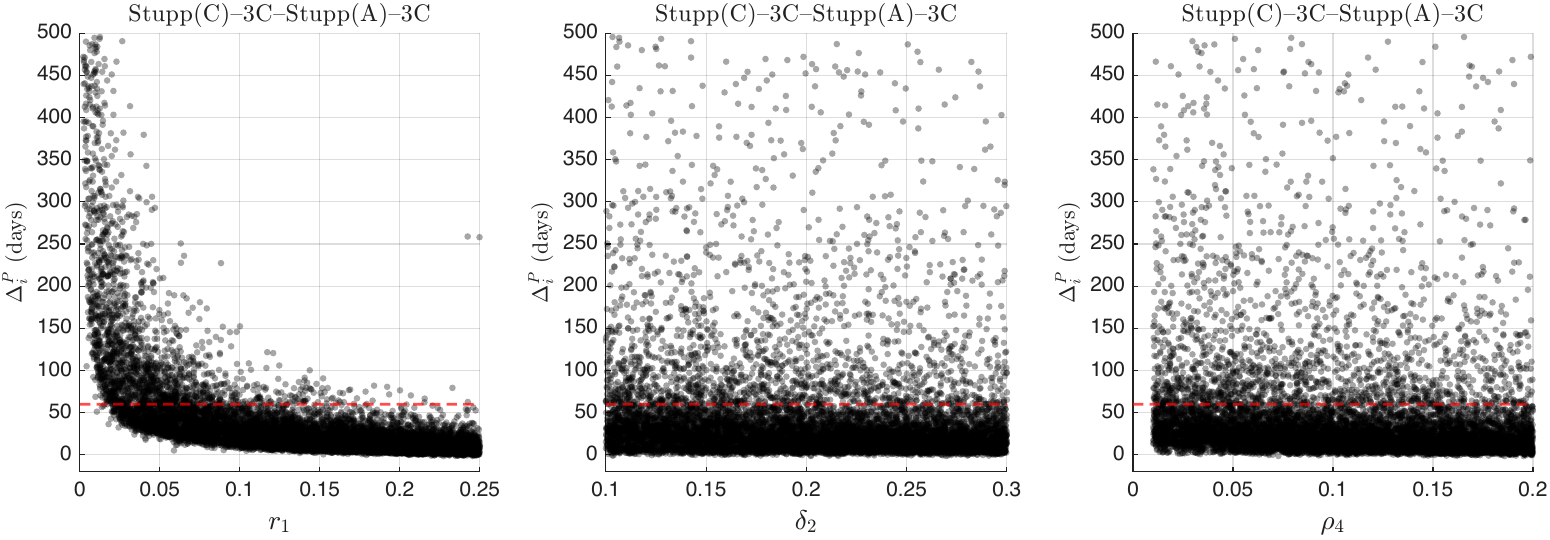}
        \caption{Individual benefit $(\Delta_i^P)$ versus $r_1$, $\delta_2$ and $\rho_4$ for the Stupp(C)--3C--Stupp(A)--3C protocol.}
        \label{fig:scatter_StuppC_3C_StuppA_3C}
    \end{figure}

A consistent qualitative pattern is observed across treatment sequences.
Larger absolute survival gains are preferentially obtained for VPs with lower tumor proliferation rates $r_1$, lower initial CAR-T-resistant fractions $\delta_2$, and lower values of $\rho_4$.
The association with $r_1$ is particularly pronounced, suggesting that
the incremental benefit of CAR-T therapy over Stupp is greatest in tumors
with slower intrinsic proliferative dynamics.
\begin{table}[ht!]
    \caption{Top five parameter groups showing the strongest Spearman rank correlations ($\rho$) with the additive individual benefit $\Delta_i^P$ for each combined protocol (excluding $\tau$). Parameters are ordered by absolute correlation magnitude.}
    \label{tab:spearman_benefit}
    \centering
    \small
    \begin{ruledtabular}
        \begin{tabular}{llrr}
            \textbf{Protocol} & \textbf{Parameter} & \textbf{$\rho$} & \textbf{$p$-value} \\
            \hline
            \multirow{5}{*}{3C--Stupp}
            & $r_1, r_2$      & $-0.7606$ & $<0.001$ \\
            & $\rho_2,\rho_3$ & $0.3606$  & $<0.001$ \\
            & $\rho_4$        & $-0.2913$ & $<0.001$ \\
            & $\delta_2$      & $-0.1802$ & $<0.001$ \\
            & $\beta_2$       & $0.1286$  & $<0.001$ \\
            \hline
            \multirow{5}{*}{Stupp--3C}
            & $r_1,r_2$       & $-0.8450$ & $<0.001$ \\
            & $\text{Ki-67}$  & $-0.2209$ & $<0.001$ \\
            & $\rho_4$        & $-0.1838$ & $<0.001$ \\
            & $\rho_2,\rho_3$ & $0.1744$  & $<0.001$ \\
            & $\delta_2$      & $-0.0993$ & $<0.001$ \\
            \hline
            \multirow{5}{*}{3C--Stupp--3C}
            & $r_1,r_2$       & $-0.7860$ & $<0.001$ \\
            & $\rho_2,\rho_3$ & $0.3387$  & $<0.001$ \\
            & $\rho_4$        & $-0.2800$ & $<0.001$ \\
            & $\delta_2$      & $-0.1647$ & $<0.001$ \\
            & $\text{Ki-67}$  & $-0.1140$ & $<0.001$ \\
            \hline
            \multirow{5}{*}{Stupp(C)--3C--Stupp(A)}
            & $r_1, r_2$      & $-0.7467$ & $<0.001$ \\
            & $\rho_2,\rho_3$ & $0.3098$  & $<0.001$ \\
            & $\rho_4$        & $-0.2756$ & $<0.001$ \\
            & $\delta_2$      & $-0.2159$ & $<0.001$ \\
            & $\rho_1$        & $-0.1757$ & $<0.001$ \\
            \hline
            \multirow{5}{*}{Stupp(C)--3C--Stupp(A)--3C}
            & $r_1,r_2$       & $-0.7693$ & $<0.001$ \\
            & $\rho_2,\rho_3$ & $0.3334$  & $<0.001$ \\
            & $\rho_4$        & $-0.2815$ & $<0.001$ \\
            & $\delta_2$      & $-0.1920$ & $<0.001$ \\
            & $\rho_1$        & $-0.1267$ & $<0.001$ \\
        \end{tabular}
    \end{ruledtabular}
\end{table}Spearman rank-correlation analysis confirms these trends
(Table~\ref{tab:spearman_benefit}). The tumor proliferation rate $r_1$ (and hence $r_2=r_1/2$ in the reference VP cohort) shows the strongest association with individual benefit in all
five combined protocols, with correlation coefficients ranging from
approximately $-0.75$ to $-0.85$. Thus, slower-growing tumors consistently exhibit larger
absolute survival gains from the addition of CAR-T therapy.

CAR-T-related parameters also contribute substantially to the variability
in benefit. Higher values of the CAR-T expansion parameter $\rho_2$ are
generally associated with larger survival gains, whereas larger values of
the tumor-mediated CAR-T inactivation parameter $\rho_4$ and a larger
initial CAR-T-resistant fraction $\delta_2$ are associated with smaller
benefits. Although the precise ranking of the remaining parameters varies
among treatment sequences, these results identify a consistent responder
profile characterized by slower tumor proliferation, stronger CAR-T
expansion, reduced tumor-mediated CAR-T loss, and a smaller CAR-T-resistant
tumor fraction.

Because the magnitude of an absolute survival extension may depend on the
survival time already achieved under Stupp, we next examine whether the
same parameter-response relationships are observed when treatment benefit
is expressed in relative terms using $R_i^P$.

\begin{figure}[ht!]
    \centering
    \includegraphics[width=\textwidth]{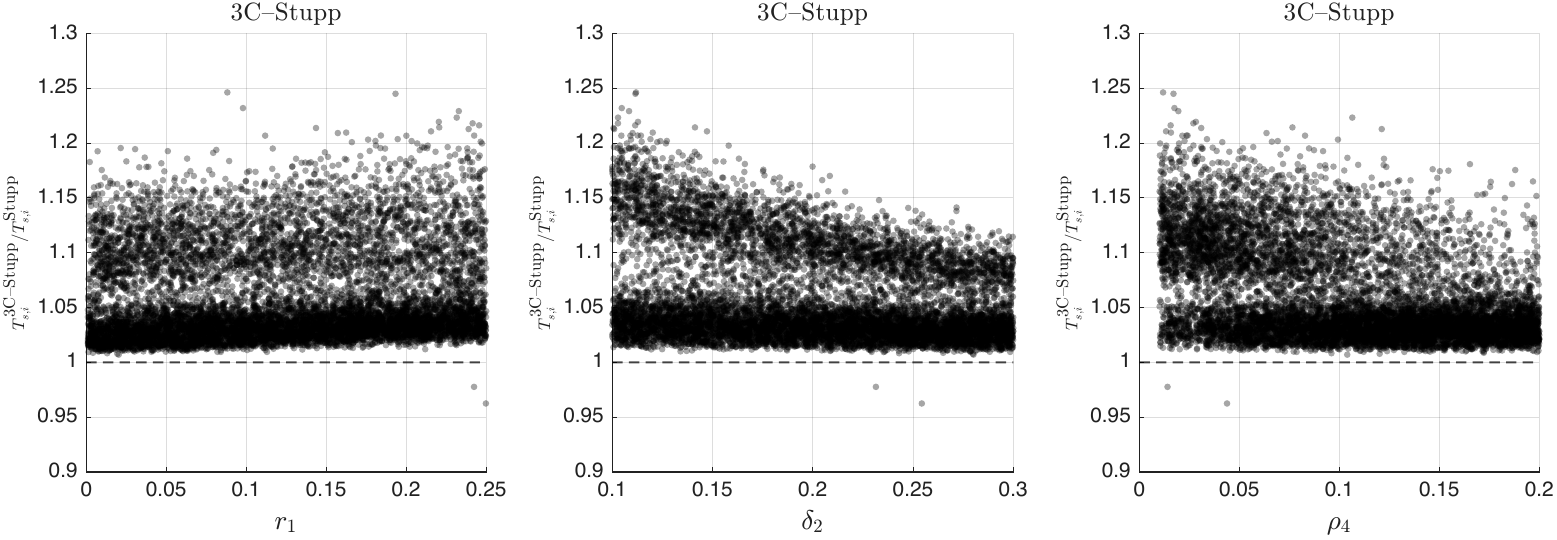}
    \caption{Multiplicative survival benefit ($T_{s,i}^{\text{3C--Stupp}} / T_{s,i}^{\text{Stupp}}$) versus tumor proliferation rate ($r_1$), CAR-T resistant fraction ($\delta_2$), and tumor-mediated CAR-T inactivation rate ($\rho_4$). The dashed line indicates a ratio of 1 (no relative benefit).}
    \label{fig:multiplicative_scatter_3C_Stupp}
\end{figure}

\begin{figure}[ht!]
    \centering
    \includegraphics[width=\textwidth]{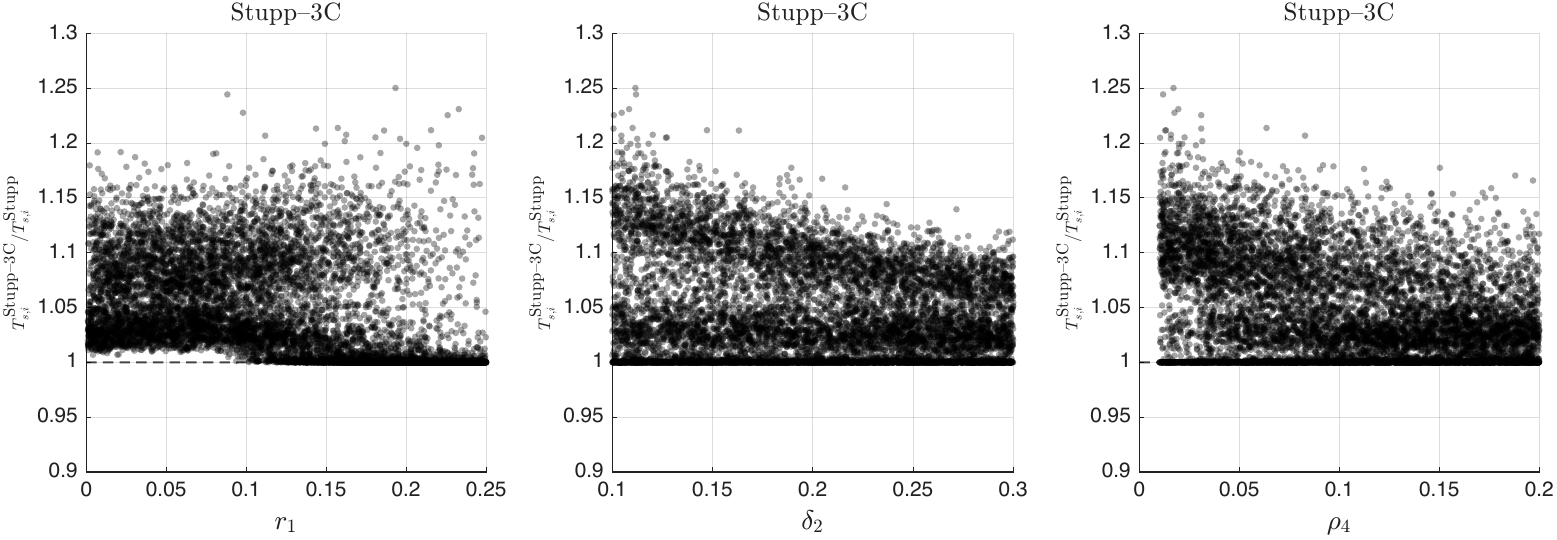}
    \caption{Multiplicative survival benefit ($T_{s,i}^{\text{Stupp--3C}} / T_{s,i}^{\text{Stupp}}$) versus $r_1$ (left), $\delta_2$ (center), and $\rho_4$ (right).}
    \label{fig:multiplicative_scatter_Stupp_3C}
\end{figure}

\begin{figure}[ht!]
    \centering
    \includegraphics[width=\textwidth]{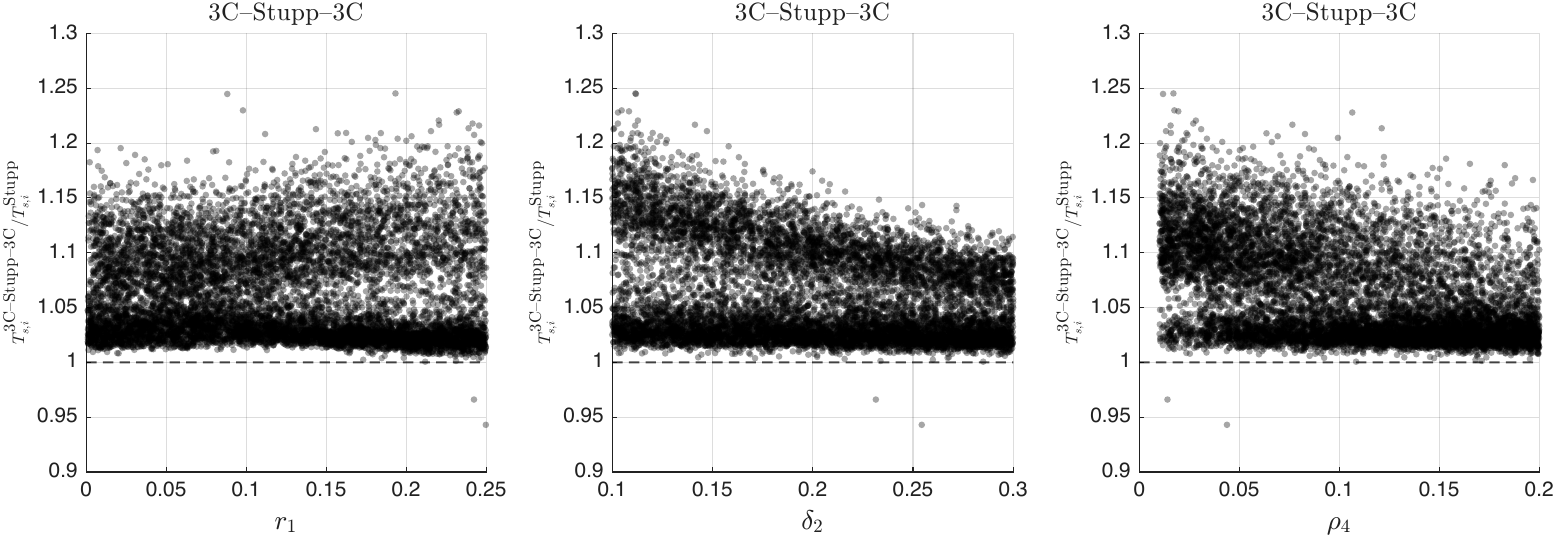}
    \caption{Multiplicative survival benefit ($T_{s,i}^{\text{3C--Stupp--3C}} / T_{s,i}^{\text{Stupp}}$) versus $r_1$ (left), $\delta_2$ (center), and $\rho_4$ (right).}
    \label{fig:multiplicative_scatter_3C_Stupp_3C}
\end{figure}

\begin{figure}[ht!]
    \centering
    \includegraphics[width=\textwidth]{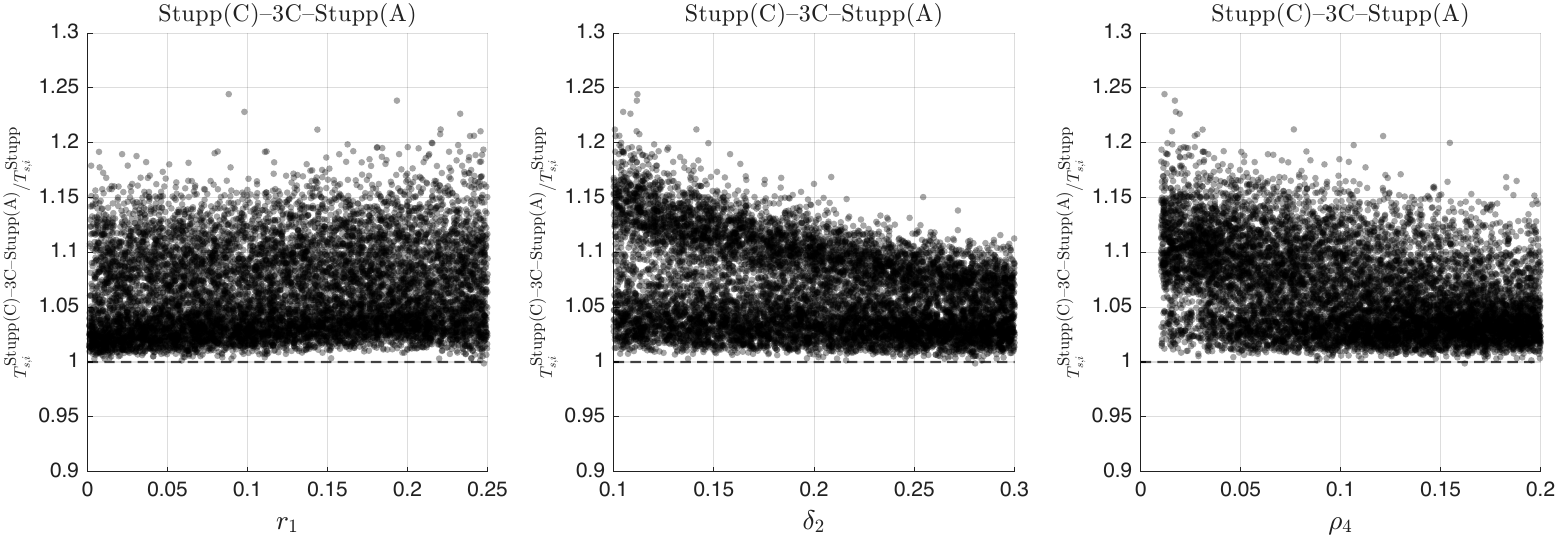}
    \caption{Multiplicative survival benefit ($T_{s,i}^{\text{Stupp(C)--3C--Stupp(A)}} / T_{s,i}^{\text{Stupp}}$) versus $r_1$ (left), $\delta_2$ (center), and $\rho_4$ (right).}
    \label{fig:multiplicative_scatter_StuppC_3C_StuppA}
\end{figure}

\begin{figure}[ht!]
    \centering
    \includegraphics[width=\textwidth]{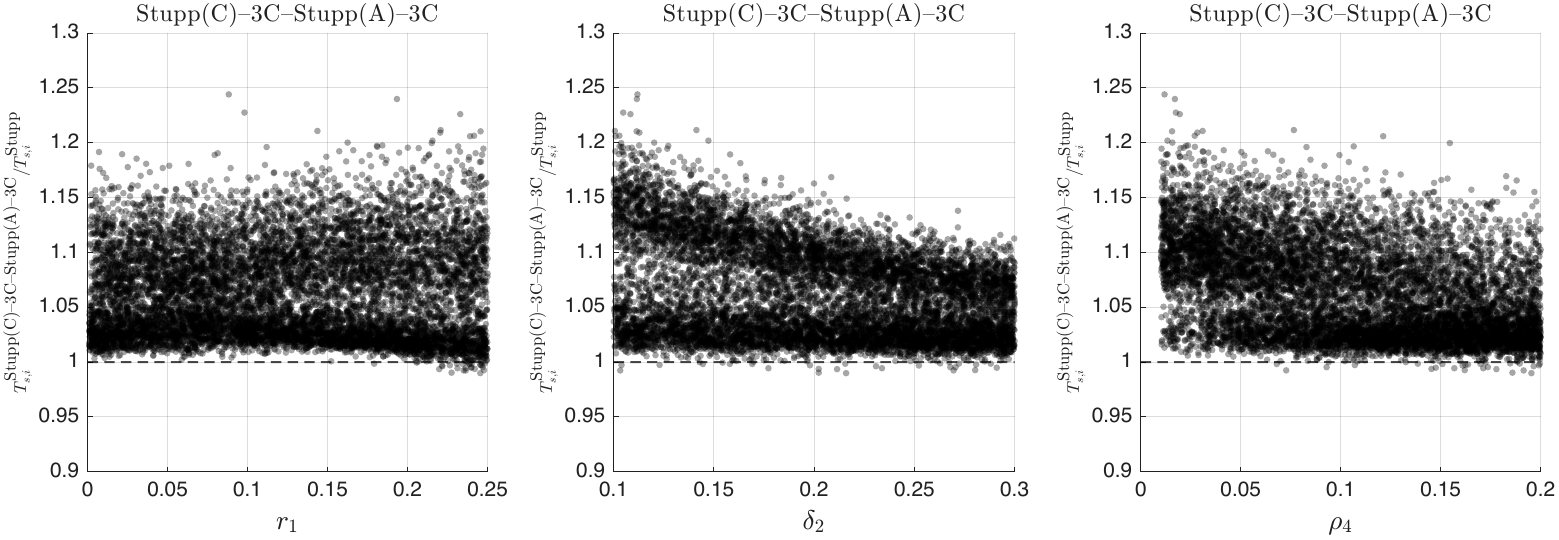}
    \caption{Multiplicative survival benefit ($T_{s,i}^{\text{Stupp(C)--3C--Stupp(A)--3C}} / T_{s,i}^{\text{Stupp}}$) versus $r_1$ (left), $\delta_2$ (center), and $\rho_4$ (right).}
    \label{fig:multiplicative_scatter_StuppC_3C_StuppA_3C}
\end{figure}

Figure \ref{fig:multiplicative_scatter_3C_Stupp} shows the relative survival benefit for the 3C-Stupp
protocol. After normalization by survival under Stupp, the dependence
on $r_1$ becomes considerably less pronounced, whereas increasing $\delta_2$ or $\rho_4$ is associated with a smaller relative benefit of treatment benefit.

Figure \ref{fig:multiplicative_scatter_Stupp_3C} shows the corresponding relative benefit for the Stupp-3C protocol. Again, the dependence on $r_1$ is substantially weaker than in the absolute-benefit analysis, whereas increasing $\delta_2$ or $\rho_4$ is associated with a progressive reduction in the relative survival gain.

A similar pattern is observed for the 3C-Stupp-3C protocol (Fig. \ref{fig:multiplicative_scatter_3C_Stupp_3C}). The relative benefit remains only weakly dependent on
$r_1$, while lower CAR-T resistance and weaker tumor-mediated CAR-T
inactivation are associated with larger improvements over Stupp.

Figure \ref{fig:multiplicative_scatter_StuppC_3C_StuppA} shows the results for the
Stupp(C)-3C-Stupp(A) sequence. The normalization again attenuates the dependence on tumor proliferation, whereas the decrease in relative benefit with increasing $\delta_2$ and $\rho_4$ remains clearly
visible.

Finally, the Stupp(C)-3C-Stupp(A)-3C protocol exhibits the same qualitative behavior (Fig. \ref{fig:multiplicative_scatter_StuppC_3C_StuppA_3C}), with a comparatively weak dependence
on $r_1$ and decreasing relative benefit as either $\delta_2$ or
$\rho_4$ increases.

Taken together, Figs. \ref{fig:multiplicative_scatter_3C_Stupp} -\ref{fig:multiplicative_scatter_StuppC_3C_StuppA_3C} show that the strong association between
$r_1$ and the absolute survival gain is substantially reduced when
benefit is normalized by survival under Stupp. In contrast, the
associations with $\delta_2$ and $\rho_4$ persist across all treatment
sequences, supporting the role of pre-existing CAR-T resistance and
tumor-mediated CAR-T inactivation as determinants of the relative
benefit obtained from adding CAR-T therapy.

Finally, to distinguish associations with differential CAR-T benefit from
general prognostic effects, we examine the correlations between model
parameters and absolute survival time $T_s$ under each treatment protocol (see Table \ref{tab:spearman_correlation_raw_tsurv}).

\begin{table}[ht!]
    \caption{Top five parameter groups showing the strongest Spearman rank correlations ($\rho$) with the survival times $T_{s}$ for each combined protocol. Parameters are ordered by absolute correlation magnitude.}
    \label{tab:spearman_correlation_raw_tsurv}
    \centering
    \small
    \begin{ruledtabular}
        \begin{tabular}{llrr}
            \textbf{Protocol} & \textbf{Parameter} & \textbf{$\rho$} & \textbf{$p$-value} \\
            \hline
            \multirow{5}{*}{Stupp}
            & $r_1,r_2$       & $-0.8891$ & $<0.001$ \\
            & \text{Ki-67}    & $-0.3887$ & $<0.001$ \\
            & $\beta_2$       & $-0.1488$ & $<0.001$ \\
            & $\alpha$        &  $0.0788$ & $<0.001$ \\
            & $\beta$         &  $0.0723$ & $<0.001$ \\
            \hline
            \multirow{5}{*}{3C--Stupp}
            & $r_1,r_2$       & $-0.8932$ & $<0.001$ \\
            & \text{Ki-67}    & $-0.3806$ & $<0.001$ \\
            & $\beta_2$       & $-0.1400$ & $<0.001$ \\
            & $\alpha$        &  $0.0742$ & $<0.001$ \\
            & $\beta$         &  $0.0681$ & $<0.001$ \\
            \hline
            \multirow{5}{*}{Stupp--3C}
            & $r_1,r_2$       & $-0.8937$ & $<0.001$ \\
            & \text{Ki-67}    & $-0.3792$ & $<0.001$ \\
            & $\beta_2$       & $-0.1387$ & $<0.001$ \\
            & $\alpha$        &  $0.0777$ & $<0.001$ \\
            & $\beta$         &  $0.0708$ & $<0.001$ \\
            \hline
            \multirow{5}{*}{3C--Stupp--3C}
            & $r_1,r_2$       & $-0.8935$ & $<0.001$ \\
            & \text{Ki-67}    & $-0.3797$ & $<0.001$ \\
            & $\beta_2$       & $-0.1390$ & $<0.001$ \\
            & $\alpha$        &  $0.0747$ & $<0.001$ \\
            & $\beta$         &  $0.0686$ & $<0.001$ \\
            \hline
            \multirow{5}{*}{Stupp(C)--3C--Stupp(A)}
            & $r_1,r_2$       & $-0.8937$ & $<0.001$ \\
            & \text{Ki-67}    & $-0.3804$ & $<0.001$ \\
            & $\beta_2$       & $-0.1395$ & $<0.001$ \\
            & $\alpha$        &  $0.0753$ & $<0.001$ \\
            & $\beta$         &  $0.0692$ & $<0.001$ \\
            \hline
            \multirow{5}{*}{Stupp(C)--3C--Stupp(A)--3C}
            & $r_1,r_2$       & $-0.8941$ & $<0.001$ \\
            & \text{Ki-67}    & $-0.3792$ & $<0.001$ \\
            & $\beta_2$       & $-0.1382$ & $<0.001$ \\
            & $\alpha$        &  $0.0758$ & $<0.001$ \\
            & $\beta$         &  $0.0695$ & $<0.001$ \\
        \end{tabular}
    \end{ruledtabular}
\end{table}

\subsubsection{Exploratory virtual cohort: patient-specific protocol}
\label{sec:personalization}

The substantial heterogeneity in individual treatment benefit raises a
natural question: whether the population-level outcome could be improved by
selecting the treatment sequence separately for each VP rather
than applying the same protocol uniformly to the entire cohort. To explore
the theoretical upper bound of such a strategy, we retrospectively identify,
for each VP, the protocol yielding the longest survival time.

For each VP $i$, we define
\begin{equation}
P_i^\ast=\arg\max_{P} T_{s,i}^{P},
\end{equation}
and the corresponding optimal survival time as
\begin{equation}
T_{s,i}^\ast=\max_{P} T_{s,i}^{P}.
\end{equation}

\begin{table*}[t!]
    \caption{Distribution of patient-specific optimal protocol selection across the virtual population ($N=10000$). The \textit{Pop. Median OS} column reports the median OS obtained when each protocol is applied uniformly to the entire cohort, whereas the \textit{Subpop. Median OS} column reports the median survival within the subset of VPs for whom that protocol is individually optimal. Furthermore, the observed median and 95\% CI for $T_s^\ast$ is $14.81$ months ($95$\% CI: $14.51-15.16$ months).}
    \label{tab:protocol_selection}
    \centering
    \small
    \begin{ruledtabular}
        \begin{tabular}{lccc}
            \textbf{Protocol ($P$)} & \textbf{Optimal for VPs (\%)} & \textbf{Pop. Median OS (months)} & \textbf{Subpop. Median OS ($T_s^\ast$,months)} \\
            \hline
            Stupp                          & 0.00          & 13.88          & -- \\
            3C--Stupp                      & 22.79          & 14.60          & 17.06 \\
            Stupp--3C                      & 15.17          & 14.74          & 30.56 \\
            3C--Stupp--3C                  & \textbf{30.95} & \textbf{14.78} & 18.02 \\
            Stupp(C)--3C--Stupp(A)         & 23.04          & 14.74          & 9.15 \\
            Stupp(C)--3C--Stupp(A)--3C     & 8.05          & 14.74          & 9.71 \\
        \end{tabular}
    \end{ruledtabular}
\end{table*}


\begin{table*}[t!]
    \caption{Median parameter values across the global virtual population and within the optimal patient subpopulations for each clinical protocol.}
    \label{tab:biological_profiles_subpopulations}
    \centering
    \small
    \begin{ruledtabular}
        \begin{tabular}{lcccccccc}
            \textbf{Optimal Subpopulation} & \textbf{$r_1$} & \textbf{$\tau$} & \textbf{Ki-67} & \textbf{$\delta_1$} & \textbf{$\delta_2$} & \textbf{$\rho_1$} & \textbf{$\rho_2,\rho_3$} & \textbf{$\rho_4$} \\
            \hline
            \textbf{Global Population ($N=10000$)} & 0.12 & 0.09 & 0.30 & 0.20 & 0.05 & 0.09 & 4.57 & 0.11 \\
            \hline
            \textbf{3C--Stupp                   } & 0.10 & 0.07 & 0.33 & 0.20 & 0.05 & 0.09 & 6.12 & 0.07 \\
            \textbf{Stupp--3C                   } & 0.06 & 0.05 & 0.28 & 0.20 & 0.05 & 0.07 & 3.40 & 0.12 \\
            \textbf{3C--Stupp--3C               } & 0.12 & 0.09 & 0.26 & 0.20 & 0.05 & 0.10 & 5.61 & 0.09 \\
            \textbf{Stupp(C)--3C--Stupp(A)      } & 0.18 & 0.16 & 0.32 & 0.20 & 0.05 & 0.08 & 2.55 & 0.14 \\
            \textbf{Stupp(C)--3C--Stupp(A)--3C  } & 0.16 & 0.12 & 0.32 & 0.19 & 0.05 & 0.09 & 4.48 & 0.12 \\
        \end{tabular}
    \end{ruledtabular}
\end{table*}

Although different protocols were individually optimal for substantial subsets of the virtual population, retrospective patient specific protocol selection produced only a marginal improvement in population level median
OS, from $14.78$ months for the best fixed strategy to $14.81$ months ($95$\% CI: $14.51-15.16$ months). This suggests that VPs who benefit from the combination of CAR-T and Stupp protocols are a subset of the virtual cohort which responds well with each combinations investigated here, in line with the correlations presented in Table~\ref{tab:spearman_benefit}, where the correlation is independent from the combined protocols. To see a higher benefit for all VPs, the combination should be further personalized for each VP.

\section{Discussion}

In this work, we have developed a unified mathematical framework to investigate the interaction between RT, TMZ, and CAR-T cell therapy in MGs. The model integrates and extends our previous formulations for TMZ--CAR-T therapy and RT--TMZ treatment \cite{Sinelshchikov2025,PeralesPaton2026}, simultaneously accounting for treatment-specific resistance, proliferative and quiescent tumor states, treatment-induced damage, CAR-T cell dynamics, and the detrimental effects that conventional therapies may exert on therapeutic T cells. To our knowledge, these mechanisms have not previously been incorporated within a single mathematical framework for the combined study of RT, TMZ, and CAR-T therapy in MGs, and in mathematical oncology in general.

A central feature of our model is that the three modalities target partially overlapping tumor populations while dynamically interfering with one another. TMZ and CAR-T cells affect complementary, treatment-sensitive proliferative phenotypes, whereas RT triggers accelerated repopulation in quiescent cells and acts more broadly across proliferative states. 
Concurrently, RT and TMZ deplete CAR-T cell counts, while the latter drives chemoresistance. 
Consequently, therapeutic outcomes emerge from highly non-linear, complex dynamics governed not just by individual treatment efficacy, but also by clonal heterogeneity and temporal scheduling. This significantly advances our previous TMZ--CAR-T model, where sequencing substantially shifted outcomes~\cite{Sinelshchikov2025}, by embedding the proliferative-quiescent structure and treatment-induced damage of the RT--TMZ framework~\cite{PeralesPaton2026}. The resulting model offers a setting to explore treatment resistance, cellular heterogeneity, and multimodal scheduling simultaneously.

The mathematical analysis provides a complementary interpretation of the treatment mechanisms. In the absence of sustained therapy, the tumor-free equilibrium is unstable, reflecting the ability of a residual viable tumor population to repopulate the system. By contrast, the continuous-treatment formulation identifies parameter regimes in which a tumor-free state becomes locally asymptotically stable. Similar threshold-type behavior was obtained in our previous TMZ--CAR-T model \cite{Sinelshchikov2025} and in other mathematical formulations of CAR-T therapy for glioblastoma \cite{Bodnar2023amcs}, emphasizing that durable tumor control depends on the balance between tumor proliferation and therapeutic pressure.

In the present model, the stability conditions
$\bar{E}_0>\mu(r_1-\alpha_4)/(\alpha_1+\epsilon_1)$
and
$\bar{C}_0>\left(\rho_1+\alpha_6+\alpha_3\bar{E}_0/\mu\right)(r_2-\alpha_5)/\alpha_2$
also reveal the competing effects generated by multimodal treatment. RT-mediated tumor killing, represented through $\alpha_4$ and $\alpha_5$, reduces the TMZ and CAR-T pressure required for tumor control. Conversely, RT-induced CAR-T loss through $\alpha_6$ and TMZ-induced CAR-T depletion through $\alpha_3\bar{E}_0/\mu$ increase the CAR-T input required for stabilization. Thus, the analytical results already capture the competition between complementary antitumor effects and treatment-induced interference that motivates the subsequent numerical study of treatment sequencing. These thresholds should be regarded as theoretical controls rather than clinical dosing rules, since the clinically relevant protocols consist of finite impulsive administrations rather than sustained treatment.

The explicit distinction between proliferative and quiescent tumor cells provides an additional mechanism for treatment persistence and recurrence. Quiescent MG populations have been associated with protective microenvironmental niches, drug-tolerant states, and subsequent tumor regrowth \cite{Tejero2019,Liau2017,Xie2022}. In the present model, quiescent cells therefore act as a protected reservoir that must return to the proliferative state before becoming susceptible to the modeled therapeutic effects. This should be interpreted as a phenomenological simplification rather than as implying complete biological resistance of every quiescent tumor cell, but it captures the general observation that cellular state can substantially modify treatment vulnerability. Crucially, this compartmental structure decouples microscopic cell division $r$ from macroscopic tumor growth, which scales as $r_{\mathrm{eff}} \simeq \text{Ki-67} \cdot r$ during early phases. This theoretical distinction aligns with the vast clinical heterogeneity of untreated glioblastomas, where Stensj{\o}en \textit{et al.} documented a median volume-doubling time of $50$ days alongside pronounced, size-dependent 
growth deceleration~\cite{Stenjoen}. Importantly, the proliferative fraction is a dynamical variable. While initialized by the $\text{Ki-67}$ index ($\Delta$), the system predicts that as the tumor approaches carrying capacity, $U \to \beta_2 / (\beta_1 + \beta_2)$. Therapy disrupts this state through selective clonal depletion, though the natural quiescent-to-proliferative flux tends to restabilize the $\text{Ki-67}$ balance post-treatment. Thus, $\text{Ki-67}$ operates as a baseline patient marker rather than an invariant feature, capturing how growth, therapy, and internal cellular fluxes progressively reshape the proliferative architecture of the residual disease.

To validate the model, we systematically simulated all core clinical scenarios: 
postoperative untreated disease \cite{Feucht}, radiotherapy (RT) monotherapy ~\cite{Stupp2005}, temozolomide (TMZ) monotherapy \cite{Wick2012,Glantz2003}, 
their combination (the standard Stupp protocol)~\cite{Stupp2005}, and CAR-T monotherapy~\cite{Brown2024}. 
The framework closely replicates established clinical survival timescales and incremental drug benefits; 
for instance, matched simulations for the Stupp trial yield median OS times of $12.21$ months for RT 
and $14.60$ months for RT+TMZ, mirroring the clinically reported $12.1$ and $14.6$ months. 
While early-phase glioblastoma CAR-T trials exhibit significant manufacturing and logistical 
heterogeneity~\cite{Bagley2025,Choi2024,Monje2025}, our model consistently aligns with the survival 
clocks of recurrent disease benchmarks~\cite{Brown2024}. Collectively, this broad alignment confirms the biological realism of our parameterized system and underscores the robustness 
of the subsequent comparative analysis.

At the population level, adding CAR-T therapy to the Stupp protocol produced
a statistically significant but modest survival benefit across all five
treatment sequences. Median OS increased by approximately $0.8$--$0.9$ months
relative to Stupp alone, whereas the differences among the five CAR-T--Stupp
sequences were relativelly small ($<0.1$ month). Thus, within the parameter
ranges and treatment doses considered here, the principal population-level
effect arises from adding CAR-T therapy itself rather than from a uniquely
favourable temporal ordering.

From a biological perspective, this relative robustness to treatment order
may reflect the existence of competing interactions between conventional
therapy and CAR-T cells. TMZ-induced lymphodepletion has been reported to
enhance CAR-T expansion and antitumor activity in glioblastoma
\cite{Suryadevara2018}, while preclinical studies have also shown synergistic
interactions between radiotherapy and CAR-T therapy \cite{Weiss2018}.
Conversely, both cytotoxic chemotherapy and irradiation can impair therapeutic
T-cell persistence. The resulting balance between cooperation and treatment
interference may therefore reduce the existence of a single strongly preferred
sequence. This result contrasts with the stronger sequencing effects observed
in our previous TMZ--CAR-T framework \cite{Sinelshchikov2025}. The additional
proliferative--quiescent structure and treatment-induced damage inherited from
our RT--TMZ framework \cite{PeralesPaton2026}, together with the broader
multi-modal treatment context considered here, may therefore buffer the effect
of fine schedule rearrangements.

Our scheduling sensitivity analyses reinforce these insights. At a fixed cumulative dose, infusion fractionation minimally shifts median survival, whereas scaling up the total cell count offers a predictable benefit. Crucially, temporal spacing effects are sequence-dependent: widening the gap pre-Stupp provides a minor survival advantage by extending early T-cell activity, whereas delaying immunotherapy post-Stupp steadily erodes outcomes. Thus, preventing excessive post-treatment gaps appears far more vital than meticulously optimizing fractionation schedules. This conclusion accords with clinical evidence that solid-tumor CAR-T performance depends on a multifactorial interplay of cell persistence, trafficking, tumor burden, and the overall therapeutic environment~\cite{Li2025}.

Population averages, however, conceal substantial heterogeneity in individual benefit. 
While the median population-level gain remains under one month, approximately $19\%$--$22\%$ 
of VPs extend survival by more than $60$ days, and about $3\%$ gain over a year under 
the combined arms. Similarly, relative survival increases exceed $5\%$ in $36\%$--$55\%$ 
of VPs, whereas gains above $20\%$ remain rare. Thus, a modest shift in population-level 
median OS does not imply a uniformly negligible effect across individual virtual patients.

The intrinsic proliferation rate $r_1$ emerges as the primary determinant of additive CAR-T benefit, with slower-growing tumors consistently yielding greater therapeutic returns. T-cell kinetics, microenvironmental inactivation, and initial clonal resistance further shape this efficacy, mirroring well-documented clinical barriers such as poor persistence, immunosuppression, and antigen heterogeneity~\cite{Li2025}. Specifically, clinical observations of antigen escape following EGFRvIII-targeted CAR-T trials support our focus on pre-existing and therapy-selected resistant fractions~\cite{ORourke2017}. Crucially, these associations represent model-derived predictive hypotheses rather than established clinical biomarkers, and their validation will require independent empirical datasets.

The protocol-selection analysis reinforces the contrast between individual heterogeneity and population baseline shifts. Although different sequences are individually optimal for distinct subsets of the cohort, retrospective, patient-specific matching raises the median OS by a mere $0.04$ months (from $14.78$ to $14.82$). Consequently, schedule optimization yields minimal returns when restricted to these candidate configurations. This finding underscores that the subpopulation medians in Table\ref{tab:protocol_selection} reflect selected patient phenotypes rather than a direct head-to-head measure of protocol efficacy.

This distinction is vital for model-based virtual clinical trials, where population-level rankings and individually optimal strategies often diverge, and conclusions remain highly dependent on cohort generation and constraints~\cite{Craig2023VCT,Gevertz2024VCT}. Substantial personalization gains may therefore require adapting not only sequence order, but also doses, fractionation, and intervals. Ultimately, shifting from retrospective stratification to predictive personalization will require integrating patient-specific longitudinal data, as envisioned in mechanistic digital-twin frameworks in oncology~\cite{Wang2024,Wu2022DigitalTwin}.

Several limitations temper our findings. First, the model is non-spatial,  capturing restricted CAR-T cell infiltration and niche-mediated quiescent protection  phenomenologically rather than through explicit tumor geometries, vascular architectures,  or immune trafficking. Second, treating quiescent cells as completely refractory  to cytotoxicity is an idealized representation of state-dependent treatment susceptibility  designed to mirror dormancy-driven persistence~\cite{dormantstateandback, baldominos2022quiescent}.  Third, the resistance architecture remains parsimonious, omitting intermediate drug-tolerant  persister states or cross-resistant clones that may emerge during repeated multimodal therapy.  Fourth, the virtual cohort lacks molecular markers like MGMT promoter methylation, and is  generated via bounded uniform parameter ranges rather than calibrated multivariate patient  distributions, meaning quantitative boundaries must be interpreted within this virtual clinical  trial setup~\cite{Craig2023VCT, Gevertz2024VCT}. Finally, our patient-specific sequencing  optimization is retrospective, leveraging full simulated outcomes to define a theoretical  upper bound of efficacy. Moving toward prospective treatment adaptation will require  identifying measurable pretreatment markers capable of anticipating these response profiles.

Taken together, our findings suggest that the primary value of incorporating CAR-T  therapy lies not in finding a universally optimal schedule, but in exploiting  complementary mechanisms within biologically selected patient cohorts. The minimal  differences among population-level timelines coexist with substantial individual  heterogeneity, underscoring that population- and individual-level optimization are  fundamentally distinct problems. This framework shifts the paradigm from searching for  a single population-wide regimen toward identifying the specific tumor characteristics  that dictate if, when, and at what dose immunotherapy provides a meaningful advantage  over standard care. Ultimately, combining mechanistic modeling, clinical benchmarking,  and virtual trials establishes a robust pipeline to generate testable hypotheses  and advance model-informed multimodal personalization.

\section{Conclusions} 

We have developed a unified mathematical framework to investigate the combined action of  radiotherapy, temozolomide, and CAR-T cell therapy in malignant gliomas. Modeled as a single  impulsive dynamical system, the framework integrates treatment-specific resistance,  proliferative-quiescent transitions, tissue damage, and cellular interference. Analytical stability conditions successfully formalize the trade-off between synergistic cytolysis and treatment-induced T-cell antagonism under continuous treatment. Rigorous clinical benchmarking  across all core regimens---including untreated, monotherapeutic, and combined Stupp protocol---validates the temporal realism of our parameterized virtual population.  

At the population level, incorporating CAR-T therapy into standard chemoradiotherapy yields  a consistent but modest survival extension, shifting median OS by $0.8$--$0.9$ months with negligible divergence among the five candidate sequences. Sensitivity analyses confirm that median outcomes are robust to scheduling variations, indicating that the baseline therapeutic gain stems from the inclusion of the immunotherapeutic axis rather than meticulous timeline  optimization. Crucially, these tight population medians mask profound individual heterogeneity:  $19\%$--$22\%$ of virtual patients extend survival by more than $60$ days, and $\sim 3\%$  exceed one year.   

Our parameter analysis identifies the intrinsic tumor proliferation rate $r_1$ as the primary candidate biomarker of response, with slower-growing tumors deriving the largest incremental gains. While distinct treatment sequences are individually optimal for specific  patient subsets, retrospective personalized scheduling raises the overall population median  by a mere $0.04$ months. This suggests that the promise of personalization lies not in shifting  broad epidemiological medians, but in identifying responsive phenotypes. Ultimately, this  work shifts the paradigm toward mapping patient-specific vulnerability windows.

\section*{Acknowledgments}
\noindent
Dmitry Sinelshchikov is supported by PID2025-170691NB-I00 funded by Ministerio de Ciencia e Innovación/Agencia Estatal de Investigación. Spain (doi:10.13039/501100011033) and European Regional Development Fund (ERDF A way of making Europe), H2020-MSCACOFUND-2020-101034228-
WOLFRAM2 and PIBA-2024-1-0016.\\
Nikols Amaru Mora Mill\'an  is supported by PIBA-2024-1-0016.\\
Juan Belmonte-Beitia was partially supported by project
PID2024-155384OB-C21, funded by Ministerio de Ciencia e 
Innovación/Agencia Estatal de Investigación, 
Spain (doi:10.13039\-/501100011033) and European Regional 
Development Fund (ERDF A way of making Europe). \\
This research is part of the research project 
SBPLY/23/180225/000041, funded by the European Union through 
the European Regional Development Fund (ERDF) and 
by the Regional Government of Castilla-La Mancha (JCCM) 
through the INNOCAM programme.\\
Matteo Italia was partially supported by CONVOCATORIA INTRAMURAL – IDISCAM PhD, funded by Instituto de Investigación Sanitaria de Castilla la Mancha (IDISCAM).\\
This work was partially supported by University of Castilla-La Mancha / ERDF, A way of making Europe (Applied Research Projects) under grant  2025-GRIN-38309.\\
This work was partially supported by project PID2025-174928OB-I00, funded by Ministerio de Ciencia e Innovación/Agencia Estatal de Investigación. Spain (doi:10.13039/501100011033) and European Regional Development Fund (ERDF A way of making Europe).

\section*{Data Availability}

Data sharing is not applicable to this article as no new data were created or analyzed in this study.

\section*{Code Availability}

The simulation code is available at https://github.com/disinel/Combining-standard-chemoradiotherapy-with-CAR-T-cell-therapy-in-malignant-gliomas.

\appendix

\section{Continuous-treatment formulation for tumor-free stability analysis}
\label{sec:app_1}

The treatment protocols considered in the numerical simulations are modeled through impulsive interventions, as described in Section \ref{sec:modeling_treatments}. For the analytical study of tumor eradication, however, it is useful to introduce an auxiliary continuous-treatment formulation in which the effects of RT, TMZ, and CAR-T therapy are represented by constant effective treatment rates. This formulation is not intended to reproduce a clinically implementable treatment schedule; rather, it provides a mathematically tractable limiting problem that allows us to determine conditions under which sustained therapeutic pressure can stabilize a tumor-free state.

Specifically, $\bar{C}_0$ and $\bar{E}_0$ represent constant input rates of CAR-T cells and TMZ, respectively. The parameters $\alpha_4$ and $\alpha_5$ describe effective RT-induced killing of the corresponding proliferative tumor populations, $\gamma_1$ represents continuous RT-induced recruitment from quiescence into proliferation, and $\alpha_6$ accounts for the effective loss of CAR-T cells due to irradiation. Under these assumptions, system \eqref{eq:new_model4} is replaced by

\begin{eqnarray}
\label{eq:new_model_const_treatment}
\dot{S}&=&r_1 S \left(1-\frac{S+Q+R_{C}+Q_{C}+R_{E}+Q_{E}+D}{K}\right)-\beta_{1}S+\beta_{2}Q-\alpha_{4}S + \gamma_{1} Q- (\alpha_1+\epsilon_{1}) E S -\alpha_2 CS, \label{eqS_const} \nonumber \\
\dot{Q}&=&\beta_{1} S - \beta_{2} Q-\gamma_{1} Q ,  \label{eqSQ_const} \nonumber \\
\dot{R}_{C}&=&r_1 R_C \left(1-\frac{S+Q+R_{C}+Q_{C}+R_{E}+Q_{E}+D}{K}\right) -\beta_{1}R_{C}+\beta_{2}Q_{C}-\alpha_{4}R_{C} + \gamma_{1} Q_{C}-(\alpha_1+\epsilon_{1})E R_C, \label{eqRC_const} \nonumber \\
\dot{Q}_{C}&=&\beta_{1} R_C - \beta_{2} Q_{C} -\gamma_{1} Q_{C},  \label{eqRCQ_const} \\
\dot{R}_{E}&=&r_2R_E  \left(1-\frac{S+Q+R_{C}+Q_{C}+R_{E}+Q_{E}+D}{K}\right) -\beta_{1}R_{E}+\beta_{2}Q_{E}-\alpha_{5}R_{E} + \gamma_{1} Q_{E}-\alpha_2CR_E+\epsilon_1 (S+R_C)E, \label{eqRE_const} \nonumber \\
\dot{Q}_{E}&=&\beta_{1} R_E - \beta_{2} Q_{E} -\gamma_{1} Q_{E},  \label{eqREQ_const} \nonumber \\
\dot{D}&=&\alpha_{1}(S+R_{C})E+\alpha_{2}(S+R_{E})C-\tau D+\alpha_{4}(S+R_{C})+\alpha_{5}R_{E}, \label{eqD_const} \nonumber \\
\dot{C}&=&C_{0}-\rho_1C+\frac{\rho_2SC}{g_1+S} +\frac{\rho_3R_EC}{g_2+R_E}\rho_{4}\frac{(S+Q+R_{C}+Q_C+R_{E}+Q_E+D)C}{g_{3}+C}-\alpha_3EC -\alpha_6C \label{eqC_const}, \nonumber\\
\dot{E}&=&\bar{E}_0-\mu E \label{eqE_const} \nonumber
\end{eqnarray}

The continuous treatment system preserves the biological structure of the impulsive model while replacing the discrete treatment events by effective sustained actions. In particular, the tumor-growth, proliferative-quiescent switching, TMZ-resistance, CAR-T expansion, and tumor-mediated CAR-T inactivation mechanisms remain unchanged. The additional terms $\alpha_4$, $\alpha_5$, $\gamma_1$, and $\alpha_6$ provide a continuous representation of the RT effects that are modelled through instantaneous state changes in the clinically motivated formulation.

This auxiliary system is used in the mathematical analysis of the main text to characterize the tumor-free equilibrium and to derive explicit conditions relating tumor proliferation and therapeutic pressure under which this equilibrium becomes locally asymptotically stable. The resulting thresholds should therefore be interpreted as theoretical benchmarks for tumor eradication rather than as clinical dosing recommendations.

\section{Proliferative-fraction dynamics and effective tumor growth rates}\label{sec:app_2}

The introduction of quiescent compartments distinguishes the intrinsic proliferation rate of actively dividing tumor cells from the effective growth rate of the corresponding total tumor population. In this appendix, we characterize the dynamics of the proliferative fraction, describe its relation with the $\mbox{Ki-67}$ index, and establish the connection between the proliferation rates of the present model and the effective growth rates used in our previous reduced formulation \cite{Sinelshchikov2025}.

For clarity, consider first a generic proliferative--quiescent pair, denoted by $X_P$ and $X_Q$, with intrinsic proliferation rate $r$. The variable $X_{P}$ represents one of the proliferative tumor components $S$, $R_{C}$ or $R_{E}$ and the variable $X_{Q}$ represents the corresponding quiescent compartment $Q$, $Q_{C}$ or $Q_{E}$,

In the absence of treatment, its dynamics can be written as
\begin{equation}
\begin{aligned}
\dot X_P &=
rX_P\left(1-\frac{T}{K}\right)
-\beta_1X_P+\beta_2X_Q,\\
\dot X_Q &=\beta_1X_P-\beta_2X_Q,
\end{aligned}
\label{eq:appendix_PQ}
\end{equation}
where $T$ denotes the total tumor burden entering the logistic growth term, as defined in the main model \eqref{eq:new_model4}.

Let
\begin{equation}
X=X_P+X_Q,
\qquad
U=\frac{X_P}{X},
\label{eq:appendix_U}
\end{equation}
where $U$ represents the proliferative fraction. Summing the two equations in \eqref{eq:appendix_PQ} gives
\begin{equation}
\dot X=
rX_P\left(1-\frac{T}{K}\right)
=
rU X\left(1-\frac{T}{K}\right).
\label{eq:appendix_total_growth}
\end{equation}
Using \eqref{eq:appendix_U}, the evolution of the proliferative fraction is therefore governed by
\begin{equation}
\dot U
=
\beta_2+
\left[
r\left(1-\frac{T}{K}\right)-\beta_1-\beta_2
\right]U
-r\left(1-\frac{T}{K}\right)U^2.
\label{eq:appendix_U_general}
\end{equation}
Thus, even in the absence of treatment, the proliferative fraction is in general a dynamical quantity coupled to the tumor burden.

\subsection*{Early-growth regime and \texorpdfstring{\mbox{Ki-67}} calibration}

During the early growth phase, $T/K\ll1$, and the logistic factor is approximately equal to one. Equation~\eqref{eq:appendix_U_general} then reduces to
\begin{equation}
\dot U
=
\beta_2+(r-\beta_1-\beta_2)U-rU^2.
\label{eq:appendix_U_exponential}
\end{equation}

Let $\Delta\in(0,1)$ denote the reference proliferative fraction identified with the $\mbox{Ki-67}$ index. The switching parameters are calibrated using the tumor populations with intrinsic proliferation rate $r_1$, namely the $S/Q$ and $R_C/Q_C$ pairs. Requiring $U=\Delta$ to be an equilibrium of their early-growth proliferative-fraction dynamics gives
\begin{equation}
\beta_2=\Delta\left(
\frac{\beta_1}{1-\Delta}-r_1
\right).
\label{eq:appendix_beta2}
\end{equation}
Consequently, the condition $\beta_2\geq0$ requires
\begin{equation}
\beta_1\geq r_1(1-\Delta).
\label{eq:appendix_admissibility}
\end{equation}

Moreover, the derivative of the right-hand side of
\eqref{eq:appendix_U_exponential}, evaluated at $U=\Delta$ and $r=r_1$, is
\begin{equation}
\frac{(1-\Delta)^2r_1-\beta_1}{1-\Delta},
\end{equation}
which is negative under condition~\eqref{eq:appendix_admissibility}. Hence, $\Delta$ is a locally asymptotically stable equilibrium of the early-growth proliferative-fraction dynamics for the $S/Q$ and $R_C/Q_C$ populations. In particular, if these populations are initialized with $U(0)=\Delta$, then $U(t)=\Delta$ within the exponential-growth approximation.

The same switching rates $\beta_1$ and $\beta_2$ are used for the TMZ-resistant pair $R_E/Q_E$, whose intrinsic proliferation rate $r_2$
is retained as a distinct parameter from $r_1$ in the general mathematical
formulation. This general formulation allows the proliferative-fraction dynamics to be characterized for arbitrary relations between $r_1$ and $r_2$, although the reference virtual cohort used in the numerical simulations assumes $r_2=r_1/2$. Consequently, unless $r_2=r_1$,
$U=\Delta$ is not an equilibrium for this population. Indeed, in the early-growth regime,
\begin{equation}
\left.
\dot U_E
\right|_{U_E=\Delta}
=
\Delta(1-\Delta)(r_2-r_1),
\label{eq:appendix_UE_initial}
\end{equation}
where $U_E=R_E/(R_E+Q_E)$. Hence, if $r_2>r_1$, the proliferative fraction of the TMZ-resistant population initially increases relative to $\Delta$, whereas it decreases if $r_2<r_1$. When $r_2=r_1$, the same equilibrium proliferative fraction is recovered. Thus, although all tumor phenotypes can be initialized using the same reference $\mbox{Ki-67}$ value, their subsequent proliferative fractions may differ because of phenotype-dependent proliferation rates.

\subsection*{Effect of logistic growth}

Even for the populations growing with rate $r_1$, the value $U=\Delta$ does not remain an exact equilibrium once logistic inhibition becomes relevant. Evaluating \eqref{eq:appendix_U_general} at $U=\Delta$, with $r=r_1$ and using \eqref{eq:appendix_beta2}, yields
\begin{equation}
\left.
\dot U
\right|_{U=\Delta}
=
-r_1\frac{T}{K}\Delta(1-\Delta).
\label{eq:appendix_delta_logistic}
\end{equation}
Therefore, increasing tumor burden tends to decrease the proliferative fraction relative to its early-growth value.

For the TMZ-resistant population one obtains
\begin{equation}
\left.
\dot U_E
\right|_{U_E=\Delta}=
\Delta(1-\Delta)
\left[
r_2\left(1-\frac{T}{K}\right)-r_1
\right].
\label{eq:appendix_UE_logistic}
\end{equation}
Thus, the evolution of its proliferative fraction reflects both its distinct intrinsic proliferation rate and the progressive effect of logistic inhibition.

In the limiting regime $T\rightarrow K$, Eq.~\eqref{eq:appendix_U_general} becomes
\begin{equation}
\dot U
\longrightarrow
\beta_2-(\beta_1+\beta_2)U,
\end{equation}
whose stationary value is
\begin{equation}
U_{\infty}
=
\frac{\beta_2}{\beta_1+\beta_2}.
\label{eq:appendix_U_asymptotic}
\end{equation}
Hence, the Ki-67-based value $\Delta$ should be interpreted as a reference proliferative fraction characterizing the initial tumor state rather than as a quantity that necessarily remains constant throughout tumor evolution. Treatment can further perturb this balance by selectively modifying the different tumor compartments.

\subsection*{Relation with the reduced growth model}

Equation~\eqref{eq:appendix_total_growth} provides the connection between the intrinsic proliferation rates of the present model and the effective macroscopic growth rates employed in our previous TMZ--CAR-T formulation \cite{Sinelshchikov2025}. For each proliferative--quiescent phenotype,
\begin{equation}
\dot X=r_{\mathrm{eff}}(t)X\left(1-\frac{T}{K}\right),
\qquad
r_{\mathrm{eff}}(t)=rU(t).
\label{eq:appendix_reff}
\end{equation}
Thus, $r$ represents the intrinsic proliferation rate of actively dividing cells, whereas $r_{\mathrm{eff}}(t)$ represents the effective growth rate of the corresponding total population.

During the early growth phase, for populations whose proliferative fraction remains close to its reference value, $U\simeq\Delta$, we obtain
\begin{equation}
r_{\mathrm{eff}}
\simeq
\Delta r.
\label{eq:appendix_rate_relation}
\end{equation}
Accordingly, the growth parameters of the previous reduced formulation can be interpreted, to leading order during early tumor growth, as
\begin{equation}
r_{\mathrm{original}}
\simeq
\mathrm{Ki\mbox{-}67}\,
r_{\mathrm{extended}}.
\label{eq:appendix_original_extended}
\end{equation}

This relation is exact whenever $U=\Delta$ and otherwise provides an early-growth approximation. In the full model, the exact relation is $r_{\mathrm{eff}}(t)=rU(t)$, so that the effective macroscopic growth rate becomes time-dependent as the proliferative fraction evolves. In particular, differences between $r_1$ and $r_2$, logistic inhibition, and treatment-induced changes in tumor composition can cause the different phenotypes to exhibit distinct effective growth dynamics even when they are initialized with the same Ki-67-based proliferative fraction.

\section{Additional sensitivity analyses}
\label{app:additional_sensitivity}

\subsection{Sensitivity to the number of CAR-T infusions and total dose}
\label{app:number_cart_infusions}

To complement the summarized analysis presented in the main text, we report
here the complete results obtained by varying the number of CAR-T
administrations from one to ten while keeping the total administered CAR-T
dose fixed. CAR-T therapy is considered either before or after the complete
Stupp protocol, for total doses of $10^9$ and $2\times10^9$ cells (see Table \ref{tab:survival_times_SUPP+CAR-T_combined}).

\begin{table}[ht!]
    \caption{Median overall survival and 95\% confidence intervals obtained by varying the number of CAR-T infusions from one to ten, with CAR-T therapy administered either before or after the Stupp protocol. Results are shown for total CAR-T doses of $10^9$ and $2\times10^9$ cells.}
    \label{tab:survival_times_SUPP+CAR-T_combined}
    \centering
    \begin{ruledtabular}
        \begin{tabular}{lcc}
            \textbf{Protocol} & \textbf{$\tilde{T}_{1}$ [95\% CI]} & \textbf{$\tilde{T}_{2}$ [95\% CI] (months)} \\
            \hline
            Stupp -- 1C & 14.74 [14.41, 15.07] & 14.86 [14.48, 15.21] \\
            1C -- Stupp & 14.49 [14.22, 14.86] & 14.59 [14.32, 14.95] \\
            Stupp -- 2C & 14.74 [14.40, 15.07] & 14.86 [14.48, 15.21] \\
            2C -- Stupp & 14.53 [14.26, 14.90] & 14.67 [14.37, 15.02] \\
            Stupp -- 3C & 14.74 [14.40, 15.07] & 14.86 [14.49, 15.23] \\
            3C -- Stupp & 14.57 [14.29, 14.94] & 14.73 [14.43, 15.09] \\
            Stupp -- 4C & 14.74 [14.40, 15.06] & 14.87 [14.49, 15.22] \\
            4C -- Stupp & 14.61 [14.33, 14.98] & 14.78 [14.46, 15.16] \\
            Stupp -- 5C & 14.74 [14.41, 15.06] & 14.86 [14.49, 15.24] \\
            5C -- Stupp & 14.63 [14.33, 14.99] & 14.82 [14.49, 15.20] \\
            Stupp -- 6C & 14.74 [14.39, 15.07] & 14.87 [14.49, 15.22] \\
            6C -- Stupp & 14.65 [14.35, 15.00] & 14.85 [14.50, 15.22] \\
            Stupp -- 7C & 14.73 [14.39, 15.07] & 14.87 [14.49, 15.22] \\
            7C -- Stupp & 14.66 [14.38, 15.01] & 14.87 [14.52, 15.23] \\
            Stupp -- 8C & 14.73 [14.39, 15.07] & 14.88 [14.50, 15.22] \\
            8C -- Stupp & 14.67 [14.39, 15.02] & 14.89 [14.55, 15.25] \\
            Stupp -- 9C & 14.72 [14.40, 15.05] & 14.87 [14.50, 15.22] \\
            9C -- Stupp & 14.68 [14.40, 15.04] & 14.91 [14.57, 15.27] \\
            Stupp -- 10C & 14.72 [14.41, 15.05] & 14.86 [14.49, 15.20] \\
            10C -- Stupp & 14.69 [14.41, 15.06] & 14.92 [14.59, 15.28] \\
        \end{tabular}
    \end{ruledtabular}
\end{table}

\subsection{Sensitivity to treatment spacing}
\label{app:treatment_spacing}

We next report the complete sensitivity analysis for the temporal spacing
between CAR-T therapy and the Stupp protocol, together with the interval
between consecutive CAR-T infusions. We consider the 3C-Stupp,
Stupp-3C, and Stupp(C)-3C-Stupp(A) sequences, thereby examining
CAR-T administration before, after, and between the two main phases of
standard chemoradiotherapy (see Tables \ref{tab:3C-Stupp-gaps1},  \ref{tab:Stupp-3C-gaps2} and \ref{tab:StuppC-3C-StuppA-gaps4}).
\begin{table}[ht!]
    \caption{Median overall survival (months) for the 3C--Stupp protocol as a function of the interval between the final CAR-T infusion and initiation of standard chemoradiotherapy ($\mathrm{Gap}_{SC}$), and the interval between consecutive CAR-T infusions ($\mathrm{Gap}_{C}$).}
    \label{tab:3C-Stupp-gaps1}
    \centering
    \begin{ruledtabular}
        \begin{tabular}{lccccccccc}
            & \multicolumn{9}{c}{\textbf{$\mathrm{Gap}_{C}$ (days)}} \\
            \cline{2-10}
            \textbf{$\mathrm{Gap}_{SC}$ (days)} & \textbf{1} & \textbf{2} & \textbf{3} & \textbf{4} & \textbf{5} & \textbf{6} & \textbf{7} & \textbf{10} & \textbf{14} \\
            \hline
            \multicolumn{10}{c}{\textbf{Total CAR-T Dose: $10^9$ cells}} \\
            \hline
            \textbf{0}   & 14.42 & 14.45 & 14.48 & 14.52 & 14.55 & 14.59 & 14.61 & 14.70 & 14.82 \\
            \textbf{7}   & 14.57 & 14.60 & 14.63 & 14.66 & 14.70 & 14.73 & 14.76 & 14.83 & 14.93 \\
            \textbf{14}  & 14.69 & 14.73 & 14.76 & 14.79 & 14.81 & 14.84 & 14.86 & 14.94 & 15.04 \\
            \textbf{21}  & 14.81 & 14.83 & 14.86 & 14.88 & 14.91 & 14.93 & 14.96 & 15.02 & 15.12 \\
            \textbf{28}  & 14.89 & 14.92 & 14.94 & 14.97 & 14.99 & 15.01 & 15.03 & 15.10 & 15.23 \\
            \hline
            \multicolumn{10}{c}{\textbf{Total CAR-T Dose: $2\times10^9$ cells}} \\
            \hline
            \textbf{0}   & 14.47 & 14.51 & 14.55 & 14.60 & 14.63 & 14.66 & 14.70 & 14.80 & 14.93 \\
            \textbf{7}   & 14.65 & 14.69 & 14.73 & 14.76 & 14.79 & 14.83 & 14.86 & 14.95 & 15.10 \\
            \textbf{14}  & 14.80 & 14.83 & 14.87 & 14.90 & 14.94 & 14.98 & 15.02 & 15.11 & 15.21 \\
            \textbf{21}  & 14.92 & 14.96 & 15.01 & 15.03 & 15.06 & 15.10 & 15.12 & 15.20 & 15.31 \\
            \textbf{28}  & 15.06 & 15.08 & 15.11 & 15.13 & 15.16 & 15.18 & 15.20 & 15.30 & 15.43 \\
        \end{tabular}
    \end{ruledtabular}
\end{table}

\begin{table}[ht!]
    \caption{Median overall survival (months) for the Stupp--3C protocol as a function of the interval between completion of Stupp and initiation of CAR-T therapy ($\mathrm{Gap}_{SC}$), and the interval between consecutive CAR-T infusions ($\mathrm{Gap}_{C}$).}
    \label{tab:Stupp-3C-gaps2}
    \centering
    \begin{ruledtabular}
        \begin{tabular}{lccccccccc}
            & \multicolumn{9}{c}{\textbf{$\mathrm{Gap}_{C}$ (days)}} \\
            \cline{2-10}
            \textbf{$\mathrm{Gap}_{SC}$ (days)} & \textbf{1} & \textbf{2} & \textbf{3} & \textbf{4} & \textbf{5} & \textbf{6} & \textbf{7} & \textbf{10} & \textbf{14} \\
            \hline
            \multicolumn{10}{c}{\textbf{Total CAR-T Dose: $10^9$ cells}} \\
            \hline
            \textbf{0}   & 14.73 & 14.73 & 14.73 & 14.73 & 14.74 & 14.74 & 14.74 & 14.74 & 14.74 \\
            \textbf{7}   & 14.73 & 14.73 & 14.74 & 14.74 & 14.74 & 14.74 & 14.74 & 14.75 & 14.75 \\
            \textbf{14}  & 14.74 & 14.74 & 14.74 & 14.74 & 14.74 & 14.74 & 14.74 & 14.75 & 14.75 \\
            \textbf{21}  & 14.74 & 14.74 & 14.74 & 14.74 & 14.74 & 14.75 & 14.75 & 14.75 & 14.75 \\
            \textbf{28}  & 14.74 & 14.74 & 14.74 & 14.74 & 14.74 & 14.74 & 14.74 & 14.74 & 14.74 \\
            \hline
            \multicolumn{10}{c}{\textbf{Total CAR-T Dose: $2\times10^9$ cells}} \\
            \hline
            \textbf{0}   & 14.88 & 14.88 & 14.89 & 14.90 & 14.90 & 14.91 & 14.91 & 14.92 & 14.93 \\
            \textbf{7}   & 14.88 & 14.89 & 14.90 & 14.90 & 14.90 & 14.90 & 14.91 & 14.92 & 14.93 \\
            \textbf{14}  & 14.87 & 14.88 & 14.87 & 14.88 & 14.89 & 14.90 & 14.90 & 14.91 & 14.91 \\
            \textbf{21}  & 14.87 & 14.87 & 14.87 & 14.88 & 14.88 & 14.88 & 14.89 & 14.90 & 14.90 \\
            \textbf{28}  & 14.86 & 14.86 & 14.86 & 14.86 & 14.87 & 14.87 & 14.87 & 14.88 & 14.88 \\
        \end{tabular}
    \end{ruledtabular}
\end{table}

\begin{table}[ht!]
    \caption{Median overall survival (months) for the Stupp(C)--3C--Stupp(A) protocol as a function of the intervals flanking the intermediate CAR-T block ($\mathrm{Gap}_{SC}$) and the interval between consecutive CAR-T infusions ($\mathrm{Gap}_{C}$).}
    \label{tab:StuppC-3C-StuppA-gaps4}
    \centering
    \begin{ruledtabular}
        \begin{tabular}{lccccccccc}
            & \multicolumn{9}{c}{\textbf{$\mathrm{Gap}_{C}$ (days)}} \\
            \cline{2-10}
            \textbf{$\mathrm{Gap}_{SC}$ (days)} & \textbf{1} & \textbf{2} & \textbf{3} & \textbf{4} & \textbf{5} & \textbf{6} & \textbf{7} & \textbf{10} & \textbf{14} \\
            \hline
            \multicolumn{10}{c}{\textbf{Total CAR-T Dose: $10^9$ cells}} \\
            \hline
            \textbf{0}   & 14.78 & 14.76 & 14.76 & 14.75 & 14.75 & 14.74 & 14.74 & 14.74 & 14.74 \\
            \textbf{7}   & 14.73 & 14.74 & 14.74 & 14.74 & 14.74 & 14.74 & 14.74 & 14.74 & 14.74 \\
            \textbf{14}  & 14.74 & 14.73 & 14.74 & 14.74 & 14.74 & 14.74 & 14.74 & 14.73 & 14.72 \\
            \textbf{21}  & 14.73 & 14.73 & 14.73 & 14.72 & 14.72 & 14.72 & 14.71 & 14.71 & 14.69 \\
            \textbf{28}  & 14.71 & 14.71 & 14.70 & 14.69 & 14.69 & 14.69 & 14.69 & 14.68 & 14.66 \\
            \hline
            \multicolumn{10}{c}{\textbf{Total CAR-T Dose: $2\times10^9$ cells}} \\
            \hline
            \textbf{0}   & 14.99 & 14.97 & 14.96 & 14.97 & 14.97 & 14.97 & 14.97 & 14.97 & 14.96 \\
            \textbf{7}   & 14.94 & 14.93 & 14.94 & 14.94 & 14.94 & 14.94 & 14.94 & 14.95 & 14.94 \\
            \textbf{14}  & 14.92 & 14.92 & 14.92 & 14.92 & 14.92 & 14.93 & 14.92 & 14.93 & 14.92 \\
            \textbf{21}  & 14.91 & 14.91 & 14.91 & 14.91 & 14.91 & 14.91 & 14.91 & 14.91 & 14.89 \\
            \textbf{28}  & 14.89 & 14.88 & 14.88 & 14.88 & 14.88 & 14.88 & 14.87 & 14.86 & 14.85 \\
        \end{tabular}
    \end{ruledtabular}
\end{table}

\subsection{Distribution of CAR-T therapy around the Stupp protocol}
\label{app:prepost_distribution}

For protocols in which CAR-T therapy is administered both before and after
standard chemoradiotherapy, we further investigate how the distribution of
the total CAR-T dose between the pre-Stupp and post-Stupp blocks affects
survival. We first consider the 3C--Stupp--3C strategy for total administered
CAR-T doses of $10^9$ and $2\times10^9$ cells (see Tables \ref{tab:3C-Stupp-3C-10^9-massive-sweep} and \ref{tab:survival_gaps_matrix}).

\begin{table*}[ht!]
    \caption{Median overall survival for the \textrm{3C--Stupp--3C} protocol as a function of the pre/post distribution of CAR-T therapy and treatment spacing, for a total CAR-T dose of $10^9$ cells. }
    \label{tab:3C-Stupp-3C-10^9-massive-sweep}
    \centering
    \begin{ruledtabular}
        \begin{tabular}{lcccccccc}
            & \multicolumn{8}{c}{\textbf{GapSC (days)}} \\
            \cline{2-9}
            \textbf{Split (Pre-Post)} & \textbf{7} & \textbf{14} & \textbf{21} & \textbf{30} & \textbf{45} & \textbf{60} & \textbf{75} & \textbf{90} \\
            \hline
            \textbf{30-70\%} & 14.79 & 14.78 & 14.78 & 14.75 & 14.70 & 14.59 & 14.56 & 14.55 \\
            \textbf{40-60\%} & 14.78 & 14.79 & 14.78 & 14.77 & 14.73 & 14.64 & 14.60 & 14.59 \\
            \textbf{50-50\%} & 14.78 & 14.79 & 14.79 & 14.78 & 14.74 & 14.67 & 14.64 & 14.63 \\
            \textbf{60-40\%} & 14.77 & 14.79 & 14.79 & 14.79 & 14.75 & 14.70 & 14.68 & 14.67 \\
            \textbf{70-30\%} & 14.75 & 14.78 & 14.79 & 14.79 & 14.77 & 14.73 & 14.70 & 14.70 \\
        \end{tabular}
    \end{ruledtabular}
\end{table*}

\begin{table*}[ht!]
    \caption{Median overall survival for the \textrm{3C--Stupp--3C} protocol as a function of the pre/post distribution of CAR-T therapy and treatment spacing, for a total CAR-T dose of $2\times10^9$ cells.}
    \label{tab:survival_gaps_matrix}
    \centering
    \begin{ruledtabular}
        \begin{tabular}{lcccccccc}
            & \multicolumn{8}{c}{\textbf{GapSC (days)}} \\
            \cline{2-9}
            \textbf{Split (Pre-Post)} & \textbf{7} & \textbf{14} & \textbf{21} & \textbf{30} & \textbf{45} & \textbf{60} & \textbf{75} & \textbf{90} \\
            \hline
            \textbf{30-70\%} & 14.93 & 14.93 & 14.94 & 14.92 & 14.86 & 14.77 & 14.70 & 14.67 \\
            \textbf{40-60\%} & 14.93 & 14.96 & 14.95 & 14.95 & 14.89 & 14.83 & 14.76 & 14.75 \\
            \textbf{50-50\%} & 14.94 & 14.96 & 14.97 & 14.96 & 14.91 & 14.87 & 14.82 & 14.80 \\
            \textbf{60-40\%} & 14.92 & 14.95 & 14.97 & 14.97 & 14.95 & 14.91 & 14.88 & 14.85 \\
            \textbf{70-30\%} & 14.90 & 14.95 & 14.98 & 14.99 & 14.97 & 14.95 & 14.92 & 14.90 \\
        \end{tabular}
    \end{ruledtabular}
\end{table*}


We perform an analogous analysis for the
Stupp(C)--3C--Stupp(A)--3C strategy, varying the distribution of CAR-T
therapy between the intermediate and post-Stupp blocks and the corresponding
treatment spacing. Results are reported for total CAR-T doses of $10^9$ and
$2\times10^9$ cells (see Table \ref{tab:StuppC-3C-StuppA-3C-10^9-massive-sweep2} and \ref{tab:StuppC-3C-StuppA-3C-2*10^9-massive-sweep}).

\begin{table*}[ht!]
    \caption{Median overall survival for the \textrm{Stupp(C)--3C--Stupp(A)--3C} protocol as a function of the pre/post distribution of CAR-T therapy and treatment spacing, for a total CAR-T dose of $10^9$ cells.}
    \label{tab:StuppC-3C-StuppA-3C-10^9-massive-sweep2}
    \centering
    \begin{ruledtabular}
        \begin{tabular}{lcccccccc}
            & \multicolumn{8}{c}{\textbf{GapSC (days)}} \\
            \cline{2-9}
            \textbf{Split (Pre-Post)} & \textbf{7} & \textbf{14} & \textbf{21} & \textbf{30} & \textbf{45} & \textbf{60} & \textbf{75} & \textbf{90} \\
            \hline
            \textbf{30-70\%} & 14.73 & 14.71 & 14.68 & 14.64 & 14.54 & 14.52 & 14.50 & 14.50 \\
            \textbf{40-60\%} & 14.74 & 14.72 & 14.69 & 14.65 & 14.59 & 14.56 & 14.53 & 14.53 \\
            \textbf{50-50\%} & 14.74 & 14.74 & 14.71 & 14.69 & 14.62 & 14.59 & 14.57 & 14.56 \\
            \textbf{60-40\%} & 14.74 & 14.74 & 14.73 & 14.70 & 14.65 & 14.63 & 14.60 & 14.59 \\
            \textbf{70-30\%} & 14.73 & 14.74 & 14.75 & 14.73 & 14.69 & 14.67 & 14.65 & 14.62 \\
        \end{tabular}
    \end{ruledtabular}
\end{table*}

\begin{table*}[ht!]
    \caption{Median overall survival for the \textrm{Stupp(C)--3C--Stupp(A)--3C} protocol as a function of the pre/post distribution of CAR-T therapy and treatment spacing, for a total CAR-T dose of $2\times10^9$ cells.}
    \label{tab:StuppC-3C-StuppA-3C-2*10^9-massive-sweep}
    \centering
    \begin{ruledtabular}
        \begin{tabular}{lcccccccc}
            & \multicolumn{8}{c}{\textbf{GapSC (days)}} \\
            \cline{2-9}
            \textbf{Split (Pre-Post)} & \textbf{7} & \textbf{14} & \textbf{21} & \textbf{30} & \textbf{45} & \textbf{60} & \textbf{75} & \textbf{90} \\
            \hline
            \textbf{30-70\%} & 14.91 & 14.90 & 14.86 & 14.81 & 14.69 & 14.64 & 14.61 & 14.59 \\
            \textbf{40-60\%} & 14.93 & 14.92 & 14.89 & 14.83 & 14.76 & 14.71 & 14.68 & 14.65 \\
            \textbf{50-50\%} & 14.94 & 14.92 & 14.91 & 14.86 & 14.79 & 14.74 & 14.71 & 14.70 \\
            \textbf{60-40\%} & 14.95 & 14.95 & 14.93 & 14.89 & 14.84 & 14.80 & 14.77 & 14.75 \\
            \textbf{70-30\%} & 14.95 & 14.95 & 14.93 & 14.92 & 14.88 & 14.84 & 14.81 & 14.79 \\
        \end{tabular}
    \end{ruledtabular}
\end{table*}

\subsection{Sensitivity to the tumor proliferation rate}
\label{app:r1_sensitivity}

Finally, we examine the dependence of survival outcomes on the tumor proliferation rate $r_1$. For each value of $r_1$, we preserve the reference-cohort relation $r_2=r_1/2$. This analysis is conceptually distinct from the
CAR-T scheduling sensitivity considered above and evaluates how changes in
tumor growth kinetics affect the outcomes of no treatment, CAR-T
monotherapy, and the 3C-Stupp-3C protocol (see Table \ref{tab:r1_protocol_comparison_all}).

\begin{table*}[t!]
    \caption{Median overall survival and 95\% confidence intervals for no treatment, CAR-T monotherapy (3C), and the 3C--Stupp--3C protocol as a function of the tumor proliferation rate $r_1$. }
    \label{tab:r1_protocol_comparison_all}
    \centering
    \begin{ruledtabular}
        \begin{tabular}{lccc}
            \textbf{$r_1$ (days$^{-1}$)} & \textbf{No treatment} & \textbf{3C} & \textbf{3C--Stupp--3C} \\
            \hline
            0.15 & 14.59 [14.27, 14.90] & 16.12 [15.73, 16.50] & 25.52 [24.91, 26.05] \\
            0.16 & 13.68 [13.39, 13.98] & 15.11 [14.77, 15.46] & 23.83 [23.27, 24.36] \\
            0.17 & 12.88 [12.61, 13.16] & 14.25 [13.91, 14.56] & 22.33 [21.82, 22.78] \\
            0.18 & 12.17 [11.91, 12.43] & 13.46 [13.13, 13.77] & 21.01 [20.57, 21.49] \\
            0.19 & 11.53 [11.29, 11.78] & 12.77 [12.46, 13.04] & 19.78 [19.42, 20.28] \\
            0.20 & 10.96 [10.74, 11.19] & 12.14 [11.84, 12.40] & 18.73 [18.40, 19.19] \\
            0.21 & 10.44 [10.24, 10.66] & 11.57 [11.29, 11.82] & 17.76 [17.43, 18.22] \\
            0.22 & 9.98 [9.78, 10.20] & 11.05 [10.78, 11.29] & 16.90 [16.56, 17.34] \\
            0.23 & 9.55 [9.35, 9.76] & 10.58 [10.32, 10.83] & 16.11 [15.80, 16.53] \\
            0.24 & 9.15 [8.97, 9.37] & 10.16 [9.90, 10.38] & 15.39 [15.07, 15.76] \\
            0.25 & 8.79 [8.62, 9.00] & 9.75 [9.50, 9.98] & 14.74 [14.42, 15.08] \\
        \end{tabular}
    \end{ruledtabular}
\end{table*}
>> 
\bibliography{bibliog.bib}

\end{document}